\documentclass[12pt]{amsart}
\usepackage{color}
\usepackage{amsmath}
\usepackage{amssymb}
\usepackage{amsthm}
\usepackage{amscd}
\usepackage{eucal} % for \mathcal 
\usepackage{enumitem}
\usepackage{mathdots}
\usepackage{tikz}
\usepackage{hyperref}
\usepackage[utf8]{inputenc}
\usepackage[all]{xy}
\hypersetup{colorlinks=true}

\theoremstyle{plain}
\newtheorem{theorem}{Theorem}[section]

\newtheorem{lemma}[theorem]{Lemma}
\newtheorem{proposition}[theorem]{Proposition}
\newtheorem{corollary}[theorem]{Corollary}

\theoremstyle{definition}
\newtheorem{definition}[theorem]{Definition}
\newtheorem{remark}[theorem]{Remark}
\newtheorem{example}[theorem]{Example}

\numberwithin{equation}{section}

\newcommand\bovermat[2]{%
	\makebox[-2pt][l]{$\smash{\overbrace{\phantom{%
					\begin{bmatrix}#2\end{bmatrix}}}^{\text{#1}}}$}#2}

\newcommand\fantome[1]{}

\def\Z{\mathbb Z}

\def\opi{\overline \pi}

\def\tpi{\widetilde{\pi}}

\newcommand{\tr}{\mathrm{tr}}

\newcommand{\inorm}[1]{{\lvert #1 \rvert}}

\DeclareMathOperator{\GL}{GL}

\DeclareMathOperator{\Log}{Log}

\DeclareMathOperator{\Frob}{Frob}

\DeclareMathOperator{\Ker}{Ker}

\DeclareMathOperator{\Id}{Id}

\DeclareMathOperator{\Lie}{Lie}

\newcommand{\bbeta}{\mathcal{B}}

\newcommand{\F}{\mathbb{F}}

\newcommand{\C}{\mathbb{C}}

\newcommand{\bm}{\mathbf{m}}

\newcommand{\bn}{\mathbf{n}}

\newcommand{\bz}{\mathbf{z}}

\newcommand{\zz}{\bold{z}}

\newcommand{\inv}{\ensuremath ^{-1}}
\newcommand{\isom}{\ensuremath \cong}

\newcommand{\TT}{\mathbb{T}}

\DeclareMathOperator{\Exp}{Exp}
\DeclareMathOperator{\Res}{Res}

\DeclareMathOperator{\Mat}{Mat}

\newcommand{\twist}{^{(1)}}

\newcommand{\twistinv}{^{(-1)}}

\newcommand{\twistk}[1]{^{(#1)}}

\definecolor{ForestGreen}{rgb}{0.0, 0.5, 0.0}

	\author{O\u{g}uz Gezm\.{i}\c{s}}
\address{Department of Mathematics, National Tsing Hua University, Hsinchu City 30042, Taiwan R.O.C.}
\email{gezmis@math.nthu.edu.tw}

\author{Nathan Green}

\address{Department of Mathematics and Statistics, Louisiana Tech University, Ruston, LA, 71270, USA}
\email{ngreen@latech.edu}

\date{\today}



\title{Mellin transform formulas for Drinfeld modules}

\begin{document}
\maketitle
\begin{abstract} We introduce formulas for the logarithms of Drinfeld modules using a framework recently developed by the second author. We write the logarithm function as the evaluation under a motivic map of a product of rigid analytic trivializations of $t$-motives. We then specialize our formulas to express special values of Goss $L$-functions as Drinfeld periods multiplied by rigid analytic trivializations evaluated under this motivic map. We view these formulas as characteristic-$p$ analogues of integral representations of Hasse-Weil type zeta functions. We also apply this machinery  for Drinfeld modules tensored with the tensor powers of the Carlitz module, which serves as the Tate twist of a Drinfeld module.
\end{abstract}

\section{Introduction}
\subsection{Motivation}
The main result of this paper gives a positive-characteristic function field analogue of certain integral representations of Hasse-Weil type zeta functions. In order to make a comparison with our new results, we remind the reader first  some of the classical theory. The starting point is one of the original proofs of the functional equation and analytic continuation of the Riemann zeta function. The classical theta function, for $t\in \mathbb{C}$ with $\Re(t)>0$
\[\Theta(t) = \sum_{n\in \Z} e^{-\pi n^2 t},\]
satisfies the functional equation
\begin{equation}\label{E:Theta}
\Theta(t) = t^{-1/2} \Theta(1/t).
\end{equation}
We also recall the definition of the Mellin transform for a real-valued function $f(x)$ with suitable decay conditions at $x=0$ and $x=\infty$,
\begin{equation}\label{D:Mellin Def}
M(f)(s) = \int_0^\infty f(x)x^{s-1}dx,
\end{equation}
for suitable $s\in \C$. If we take the Mellin transform of a normalized version of $\Theta(t)$ (and account correctly for convergence, which is nontrivial), we get
\begin{equation}\label{E:Mellin transform of Theta}
\xi(s) = M\left (\frac{\Theta(t) - 1}{2}\right )(s/2),
\end{equation}
where $\xi(s) = \pi^{-s/2} \Gamma(s/2) \zeta(s)$ is the completed zeta function. Further, if we take the Mellin transform of \eqref{E:Theta} then we recover the functional equation for the Riemann zeta function,
\[\xi(s) = \xi(1-s).\]
These derivations also establish the analytic continuation of the Riemann zeta function. We refer the reader to \cite[\S 7.1]{Neu} for details on such constructions.

With this theory as our base point, there are several important directions we can generalize these ideas. First, if we replace the classical theta function with theta series involving characters, then the same theory gives the functional equation and analytic continuation of Dirichlet $L$-functions with characters. A further generalization to higher dimensional theta series then gives the same theory for Dedekind zeta functions. Again, all this theory is detailed in \cite[\S7.2-7.5]{Neu}.

On the other hand, we can instead investigate Hasse-Weil zeta functions attached to algebraic varieties (the previous case of Dedekind zeta functions can be seen as a special case in this setting --- that discussion is outside the scope of this introduction). In this setting, at least for elliptic curves defined over the rational numbers, Wiles's modularity theorem \cite{Wil95} shows that such zeta functions are given as the Mellin transform of special meromorphic functions, in this case modular forms. There are vast generalizations of this theory to motives and profound conjectures that come with them, such as Beilinson's conjectures (see \cite{DS91}) and various aspects of the Langlands program (see \cite{Lom18}).

Our results in this paper establish an analogy to those described above in the positive characteristic function field setting. We prove that certain special values of $L$-functions can be realized as an algebraic interpolation of a Mellin transform of certain special functions. On the one hand, these $L$-values are certainly of Hasse-Weil type, because they have an Euler product representation given by the characteristic polynomial of the Frobenius acting on certain modules (see \eqref{E:L func Euler expansion2} and \eqref{E:L func Euler expansion}). On the other hand, our formulas indicate that these $L$-values can be represented as a Mellin-type transform, not of Drinfeld modular forms as one might expect, but rather of rigid analytic trivializations of Drinfeld modules, which bear several similarities to classical theta function. Thus the results we present here should be viewed as a hybrid between the two generalizations given above: They express Hasse-Weil type $L$-values in terms of a Mellin transform of an analogue of the classical theta function. Whether there is a connection between the constructions in this paper and Drinfeld modular forms is an open question. We provide a few comments on this question in Remark \ref{R:Log* and Drinfeld Modular Forms}.

Before continuing we say a few words about the difficulty and significance of our results. In the classical setting, one uses analytic ideas (cycle integration) to connect theta functions and related objects directly to $L$-functions and zeta functions. Our setting occurs in characteristic $p$, where it is cumbersome to work with characteristic-$p$ valued measures and integration (see Remark \ref{R:2} for the comparison of our results with the already existing literature). Our proofs here provide an algebraic alternative to this integration theory which takes a round-about path to connect special values of $L$-functions with the analogue of the classical theta function. Namely, we connect $L$-function values to logarithm values using the work of Taelman \cite{Tae12} and the first author \cite{Gez21}. Our new formulas in this paper then connect values of the logarithm to rigid analytic trivializations of Anderson $t$-motives. Works of Maurischat \cite{Mau22}, Pellarin \cite{Pel12} and others then allow us to connect rigid analytic trivializations to periods and Anderson generating functions, which (as we explain below) are an analogue of theta functions. The main new technical advances in this paper include modifying a crucial construction from the work of the second author \cite{Gre22} to a tensor product of motives (this is our \eqref{E:Gn tensor def}), a very careful analysis of the convergence properties of \eqref{E:Gn tensor factorization Drinfeld} carried out in \S\ref{SS:alpha convergence}, as well as a particular choice of $t$-motive bases (discussed in \S\ref{SS:And t-motives}-\ref{SS:Dual t-motives}) to account for the $\Theta_{\phi,\tau}$ matrix in \eqref{D:Theta tau}.

\subsection{The Mellin transform of Drinfeld modules and $L$-functions}\label{SS:intro1.2}
We now briefly describe our main results, after which we will make some more precise comparisons to the classical theory. Let $q=p^r$ be a prime power, and let $A := \F_q[\theta]$ and $K:=\F_q(\theta)$. Let $K_\infty$ be the completion of $K$ at the infinite place with respect to the norm $\inorm{\cdot}$, normalized so that $\inorm{\theta}=q$. This completion equals the formal Laurent series ring $\F_q((1/\theta))$ with coefficients in $\mathbb{F}_q$. Let $\C_\infty$ be a completion of an algebraic closure of $K_\infty$. Consider the non-commutative power series ring $\mathbb{C}_{\infty}[[\tau]]$ which is defined subject to the condition $\tau c=c^q\tau$ for all $c\in \mathbb{C}_{\infty}$. We also let $\mathbb{C}_{\infty}[\tau]$ be the subring of $\mathbb{C}_{\infty}[[\tau]]$ consisting of polynomials in $\tau$. We define \textit{a Drinfeld module $\phi$ of rank $r$} to be an $\mathbb{F}_q$-algebra homomorphism $\phi:A \to \C_\infty[\tau]$ given by 
\begin{equation}\label{E:drinfeld}
\phi_{\theta}:=\phi(\theta):=\theta+k_1\tau+\cdots+k_r\tau^r, \ \ k_r\neq 0.
\end{equation}
We call each $k_i$ for $1\leq i \leq r$ \textit{a coefficient of $\phi$}. We also consider $\exp_\phi$ and $\log_\phi$, which are elements in $\mathbb{C}_{\infty}[[\tau]]$, to be the exponential and logarithm functions associated to $\phi$ (see \eqref{E:Exp def} for details). The function $\exp_\phi$ has a kernel $\Lambda_\phi$ which is a free $A$-module of rank $r$, called the period lattice of $\phi$. Let us denote a set of generating periods as $\lambda_1,\dots,\lambda_r$. The comparison is often made between a Drinfeld module $\phi$ and an elliptic curve $E$ defined over $\C$. The periods $\lambda_1,\dots,\lambda_r$ should then be compared with the Weierstrass periods of $E$ and the exponential function $\exp_\phi$ should be compared to the Weierstrass-$\wp$ function.

We now briefly define Anderson generating functions which are intimately connected with periods. For a given period $\lambda_i$, define
\[
f_i:=\sum_{i= 0}^{\infty}\exp_{\phi}\left(\frac{\lambda_i}{\theta^{i+1}}\right)t^i \in \C_\infty[[t]],
\]
where $t$ is a commuting variable (in fact, $f_i$ is in a Tate algebra, see \S\ref{SS:Anderson Generating Functions}). We then define the matrix
\[\Upsilon:=\begin{pmatrix} f_1 & \cdots & \cdots & f_r\\
f_1^{(1)} & \cdots & \cdots & f_r^{(1)}\\
\vdots  & & & \vdots\\
f_1^{(r-1)} & \cdots & \cdots &f_r^{(r-1)}
\end{pmatrix}\in \GL_r(\mathbb{T}),\]
where $\cdot \twistk{k}$ is the $k$-fold application of a Frobenius twisting automorphism (again, see \S\ref{SS:Anderson Generating Functions}). The matrix $\Upsilon$ is constructed to be a rigid analytic trivialization for the Drinfeld module $\phi$. Namely, there is a naturally defined matrix $\Theta \in \Mat_{r\times r}(\C_\infty)$ coming from the $t$-motive associated to $\phi$ such that we have the functional equation
\[\Theta \Upsilon = \Upsilon\twist.\]
Let $V\in \Mat_{r\times r}(\C_\infty)$ be a matrix of constants defined in \eqref{E:V def} and let 
\begin{equation}\label{E:Psidef}
\Psi:=V^{-1}((\Upsilon^{(1)})^{\tr})^{-1}.
\end{equation}
We explain all this theory more extensively in \S\ref{SS:And t-motives}.

One ingredient to state our first main theorem comes from a recent paper of the second author \cite{Gre22}. There, the second author develops a new map $\delta_{1,\bz}^{M_{\phi}}$ for a parameter $\bz\in \C_\infty$ from $M_\phi$, the Anderson $t$-motive associated to $\phi$, to $\C_\infty$ which recovers the structure of the Drinfeld module $\phi$ (see \eqref{D:delta_1 map} for a precise definition). This map $\delta_{1,\bz}^{M_{\phi}}$ should be viewed as an algebraic interpolation of cycle integration; in \cite[Cor 5.10]{Gre22} the second author proves an algebraic analogue of the Mellin transform formula which relates the exponential function with \textit{the Carlitz zeta values $\zeta_{A}(n)$} given by  
\[
\zeta_A(n):=\sum_{\substack{a\in A\\ a \text{ monic}}}\frac{1}{a^n}\in K_{\infty}.
\]

On the other hand, letting $\mathbb{M}$ be a certain subring of $\mathbb{C}_{\infty}[[\tau]]$, we also introduce a continuous and injective map $\varphi:\mathbb{M}\to \Mat_{1\times r}(\mathbb{T})$ in \S\ref{SS:varphi map}  which extends the isomorphism between $M_{\phi}$ and $\Mat_{1\times r}(\mathbb{C}_{\infty}[t])$. In particular, this map provides a link between the $\tau$-structure and the $t$-structure on the extension of $M_{\phi}$.

We now let $\mathcal{M}:=\varphi^{-1}$ and for any $\zz\in \mathbb{C}_{\infty}$, set $\mathcal{M}_{\zz}:=\delta_{1,\zz}^{M_{\phi}}\circ \varphi^{-1}$. For any Drinfeld module $\phi$ of rank $r$ with certain conditions on its coefficients, our first main theorem (restated as Theorem \ref{T:logdrinfeld} later) relates the logarithm series $\log_{\phi}$ and the value $\log_{\phi}(\zz)$, whenever  $\zz\in \mathbb{C}_{\infty}$ is in the domain of convergence, to the map $\mathcal{M}$ and $\mathcal{M}_{\zz}$ evaluated at a product of rigid analytic trivializations.

\begin{theorem}\label{T:Main theorem intro Drinfeld case}
Let $\phi$ be a Drinfeld module given by
	\[\phi_{\theta} = \theta + k_1 \tau + \dots + k_r\tau^r\]
so that $\inorm{k_i}\leq 1$ for each $1\leq i \leq r-1$ and $k_r\in \mathbb{F}_q^{\times}$.  Let $\opi := (\lambda_1,\dots,\lambda_r)$ be the vector of fundamental periods of $\phi$. Then
\[
\log_{\phi}=\mathcal{M}\left(-\frac{1}{t-\theta}\opi(\Psi^{\tr})^{(-1)}\right).
\]
Moreover,  for any $\zz\in \mathbb{C}_{\infty}$ in the domain of convergence of $\log_{\phi}$, we have
\[
	\log_{\phi}(\zz)=\mathcal{M}_{\zz}\left(-\frac{1}{t-\theta}\opi(\Psi^{\tr})^{(-1)}\right).
	\]
\end{theorem}
\begin{remark}  By using an idea in an unpublished note of Anderson (see also \cite[Prop. 3.1.3]{ABP04} for a result in the same direction), one can show that each entry of $\Psi$ consists of a power series of $t$ which converges at any element of $\mathbb{C}_{\infty}$, which we call \textit{an entire function of $t$}. Since the twisting is an automorphism on the space of entire functions of $t$, each entry of $\Psi^{(-1)}$ is also an entire function of $t$. Moreover, using the definition of $\Upsilon$ and the analytic properties of Anderson generating functions (see for example \cite[Prop. 3.2]{EGP14}), one can show that each entry of 
\[
-\frac{1}{t-\theta}\opi(\Psi^{\tr})^{(-1)}
\]
can be analytically continued to an entire function of $t$. Indeed, in Proposition \ref{P:Image varphi entire} we show that the image $\varphi(\mathbb M)$ is contained in the space of entire functions and we show in Proposition \ref{P:inj} that $\varphi$ is injective. However, we note in Remark \ref{R:techn} that $\varphi$ is not surjective onto the space of entire functions. Therefore, the maps $\mathcal{M}$ and $\mathcal{M}_{\zz}$ can be regarded as Mellin transforms for entire functions in the image of $\varphi$ --- the exact nature of what this image is is subtle, and we will pursue this in future work. We would like to thank the referee for pointing out this perspective to the authors.
\end{remark}
%\begin{remark}
%	Note that the conditions on the coefficients $k_i$ in Theorem \ref{T:Main theorem intro Drinfeld case} as well as in our following Corollary \ref{C:Main Corollary Drinfeld modules Intro} is due to a technical obstacle related to a certain constant depending on the Drinfeld module being equal to one (see the beginning of \S3 and Remark \ref{R:B}). This restriction is necessary for a calculation to determine the explicit value of a certain quantity $\alpha$ in Theorem \ref{T:alpha}. In the general case, since the matrix $B$ defined in \eqref{D:B def}, originally introduced in \cite[(3.6)]{KP23}, has non-constant coefficients in the Tate algebra, we can not use our methods in \S3.4 to make the same calculation. However, a promising way of improving our techniques which is briefly sketched in Remark \ref{R:techn} may provide Mellin transform type formulas for more general Drinfeld modules.
%\end{remark}

To describe our next result, in what follows, we briefly describe Goss $L$-functions attached to $\phi$ introduced by Goss \cite{Gos92}, inspired by the ideas of Gekeler \cite[Rem. 5.10]{Gek91}. For a given monic irreducible polynomial $w\in A$, we set $K_w$ to be the completion of $K$ at the place corresponding to $w$. We let $(\rho_w)$ be a family of continuous representations of the Galois group of $K^\text{sep}/K$ 
%on some finite dimensional $K_w$-vector space $V_w$
 which is strictly compatible in the usual sense, meaning that the characteristic polynomial 
\[P_v(X) := \det(1-X\cdot \rho_w(\Frob_v))\]
of the Frobenius at a place $v\neq w$ of $K$ acting on the $w$-adic Tate module of $\phi$ is independent of the choice of prime $w$ and has coefficients in $A$ (along with a ramification condition - see \cite[\S8.10]{Gos96} for full details). We further let $P_v(X)=(1-a_1X)\cdots (1-a_rX)$ for some $a_1,\dots,a_r$ lying in a fixed algebraic closure of $K$ and set 
\[
P_v^{\vee}(X):=(1-a_1^{-1}X)\cdots (1-a_r^{-1}X).
\]
 We then define \textit{the $L$-function of $\phi$} to be 
\begin{equation}\label{E:L func Euler expansion2}
L(\phi,n) := \prod_{v} P_v(v^{-n})\inv,
\end{equation}
and \textit{the  dual $L$-function of $\phi$} by
\begin{equation}\label{E:L func Euler expansion}
L(\phi^{\vee},n) := \prod_{v} P_v^{\vee}(v^{-n})\inv,
\end{equation}
where the product runs over all the finite places of $A$. In this definition, by \cite[Cor. 3.6]{ChangEl-GuindyPapanikolas}, we know that $L(\phi,n)$ converges in $K_{\infty}$ for all $n\in \mathbb{Z}_{\geq 1}$ and $L(\phi^{\vee},n)$ converges in $K_{\infty}$ for all $n\in \mathbb{Z}_{\geq 0}$ (there is a way to extend the domain of such $L$-functions to an analogue of the upper half plane --- since we do not use that here, we refer the reader to \cite[\S8.1]{Gos96}). We also note that when $\phi$ is \textit{the Carlitz module} given by $C_{\theta}:=\theta+\tau$, we have, for any positive integer $n$, $L(C^{\vee},n-1)=\zeta_{A}(n)$. We refer the reader to \cite{Gos96} and \cite{Gez21} for full details on these constructions.

If we set $\zz = 1$ in the previous theorem and choose a Drinfeld module $\phi$ as in Theorem \ref{T:Main theorem intro Drinfeld case} so that $k_i\in \mathbb{F}_q$ for each $1\leq i \leq r-1$, then the value of the logarithm becomes a special value of the (dual) Goss $L$-function of $\phi$. As a corollary to Theorem \ref{T:Main theorem intro Drinfeld case}, we get the following (restated as Corollary \ref{C:Mellin for Drinfeld modules} later).

\begin{corollary}\label{C:Main Corollary Drinfeld modules Intro}
Let $\phi$ be a Drinfeld module as in Theorem \ref{T:Main theorem intro Drinfeld case} so that each $k_i\in \mathbb{F}_q$ and $\zz=1$. Then we have
\[L(\phi^{\vee},0) =\mathcal{M}_{\zz}\left(-\frac{1}{t-\theta}\opi(\Psi^{\tr})^{(-1)}\right).\]
\end{corollary}

\begin{remark}\label{R:2} It is appropriate to make a brief comparison between our formulas and the results in \cite{Gos91} and \cite{Gos92b} on the Mellin transform in the function field setting. Let $A_v$ be the completion of $A$ at $v$. Inspired by the construction of formal $p$-adic Mellin transform, in \cite[\S3]{Gos91}, Goss developed the theory of $A_v$-valued measures on $A_v$ and defined the Mellin transform of the Carlitz zeta value $\zeta_{A}(n)$ to be an element in \textit{the divided power series ring} (see \cite[\S5]{Gos89} for the details on divided power series). Although its coefficients are arithmetically interesting and related to the Carlitz zeta values (see \cite[Thm. VII]{Tha90}), there is no immediate relation to $\zeta_{A}(n)$ as in Corollary \ref{C:Main Corollary Drinfeld modules Intro}. Hence our construction seems to be better-suited in this direction. Later on, using the seminal work of Teitelbaum \cite{Tei91} relating $v$-adic measures to Drinfeld cusp forms, Goss \cite[\S4]{Gos92b} defined the Mellin transform of a Drinfeld cusp form $f$ as a continuous function $L_f$ on $\mathbb{Z}_p$ whose values are attained in a finite extension $K_{\infty}$. However, several aspects of the theory is still missing such as the link between $f$ and the functional equation of $L_f$ as well as the appearance of $L_f$ as a Dirichlet series summed over the monic polynomials in $A$, which could be more parallel to the classical setting. It would be interesting to relate our construction in the present paper to the setting of Drinfeld modular forms to have a better understanding of  the Mellin transformation (see Remark \ref{R:Log* and Drinfeld Modular Forms} for the discussion on a potential link to Drinfeld modular forms).
\end{remark}

\subsection{Comparison with the classical theta functions}
Having stated our first two main results, we now make some precise comparisons between our setting and the classical theory discussed above. Fixing a $(q-1)$-st root of $-\theta$, we define \textit{the Carlitz fundamental period} by
\[
\tilde{\pi}:=\theta(-\theta)^{1/(q-1)}\prod_{j=1}^{\infty}\left(1-\theta^{1-q^j}\right)^{-1}\in \mathbb{C}_{\infty}^{\times}.
\]
In the case of the Carlitz module $C$, our main results discussed above reduce to a formula from \cite[Cor. 5.10]{Gre22}
\begin{equation}\label{E:Carlitz Mellin transform}
L(C^{\vee},0)=\zeta_{A}(1)=\mathcal{M}_{\zz}(-\tpi \Omega),
\end{equation}
where $\Omega := 1/\omega_C\twist$ is defined in \eqref{D:omega_C}. In this context, the function  $\Omega$ should be viewed as an analogue of the theta function $\Theta(z)$ for two reasons:
\begin{enumerate}
\item Taking the function field Mellin transform of $\Omega$ produces zeta values similar to formula \eqref{E:Mellin transform of Theta}.
\item It satisfies a similar functional equation to the classical theta function. Namely,
\begin{equation}\label{E:tOmega}
t\cdot \Omega = C_{\theta}^*(\Omega),
\end{equation}
where $C_{\theta}^*$ is the adjoint Carlitz operator $C_{\theta}^*(z) := \theta z + z^{1/q}$ (compare with \eqref{E:Theta}).
\end{enumerate}
\begin{remark}\label{R:techn} 
We note here that taking the Mellin transform of \eqref{E:Theta} (after some adjustments for convergence) gives the functional equation for the completed Riemann zeta function $\xi(s)=\xi(1-s)$. It is therefore natural to ask about what happens when we combine the functional equation \eqref{E:tOmega} with our function field Mellin transform \eqref{E:Carlitz Mellin transform}. We continue to let $\zz=1$ and consider the Anderson $t$-motive $M_C$ corresponding to the Carlitz module (see \S\ref{SS:t-motive for Drinfeld} for more details on $M_C$). 

Let $L_0:=1$ and for $n\geq 1$, define $L_n:=(\theta-\theta^{q^n})\cdots(\theta-\theta^q)\in A$.
Note that by \cite[Rem. 5.13]{APT16}, we have
\begin{equation}\label{E:piomegaidentity}
-\tilde{\pi}\Omega=1+\sum_{i=0}^{\infty}\frac{(t-\theta)\cdots (t-\theta^{q^{i}})}{L_{i+1}}.
\end{equation}
In what follows, using the identification $\tau^{i}=(t-\theta)\cdots(t-\theta^{q^{i-1}})$ via the isomorphism $M_{C}\cong \mathbb{C}_{\infty}[\tau]$  for each $i\geq 0$, we write $ (-\tilde{\pi}\Omega)^{(-1)} $ as a limit of certain elements in $\mathbb{C}_{\infty}[\tau]$. Note that
\begin{multline*}
\left(\frac{(t-\theta)(t-\theta^q)\cdots (t-\theta^{q^{i}})}{L_{i+1}}\right)^{(-1)}=\frac{(t-\theta^{1/q})\tau^i}{L_{i+1}^{(-1)}}
=\frac{\tau^it}{L_{i+1}^{(-1)}}-\frac{\theta^{1/q}\tau^i}{L_{i+1}^{(-1)}}=\frac{\theta^{q^i}\tau^i}{L_{i+1}^{(-1)}}+\frac{\tau^{i+1}}{L_{i+1}^{(-1)}}-\frac{\theta^{1/q}\tau^i}{L_{i+1}^{(-1)}}\\
=\frac{(\theta^{q^i}-\theta^{1/q})\tau^i}{L_{i+1}^{(-1)}}+\frac{\tau^{i+1}}{L_{i+1}^{(-1)}}=\frac{\tau^{i+1}}{L_{i+1}^{(-1)}}-\frac{\tau^i}{L_{i}^{(-1)}}.
\end{multline*}
Thus, since $L_0=L_0^{(-1)}=1$, we obtain a telescoping sum in the partial sums:
\[
(-\tilde{\pi}\Omega)^{(-1)}=-\tilde{\pi}^{1/q}\Omega^{(-1)}=\lim_{d\to\infty} \tau^0+\sum_{i=0}^{d}\left(\frac{\tau^{i+1}}{L_{i+1}^{(-1)}}-\frac{\tau^i}{L_{i}^{(-1)}}\right)=\lim_{d\to\infty} \frac{\tau^{d+1}}{L_{d+1}^{(-1)}}.
\]
Moreover, applying $\delta_{1,\zz}^{M_C}$ to the above  combined with the fact that $\frac{1}{L_{d+1}^{(-1)}}\to 0$ as $d\to \infty$ shows that $(-\tilde{\pi}\Omega)^{(-1)}$ is in the kernel of $\delta_{1,\zz}^{M_C}$, and thus so is $\Omega\twistinv$. 

Here, we are indebted to make an important remark. Using the nonarchimedean norm on $\mathbb{M}$ defined in \S\ref{SS:varphi map}, we see that the limit $\lim_{d\to \infty} \tau^0+\sum_{i=0}^{d}\left(\frac{\tau^{i+1}}{L_{i+1}^{(-1)}}-\frac{\tau^{i}}{L_{i}^{(-1)}}\right)$ of partial sums does not converge in $\mathbb{M}$ and hence $ (-\tilde{\pi}\Omega)^{(-1)}$ is not in the range of $\varphi$. Therefore, in our next calculations, we avoid using the map $\varphi$ and we directly identify $-\tilde{\pi}\Omega$ with its unique preimage $\log_{C}$ under the map $\varphi$.

Now, noting that $\theta \zz$ lies in the domain of convergence of $\log_{C}$, we have the transformation
\[\delta_{1,\zz}^{M_C}(-t\tpi\Omega) = \delta_{1,C_\theta(\zz)}^{M_C}(-\tpi\Omega) = \delta_{1,\theta \zz}^{M_C}(-\tpi\Omega) + \delta_{1, \zz\twist}^{M_C}(-\tpi\Omega) = \log_C(\theta) + \zeta_A(1)\]
where the first equality follows from \cite[Prop. 2.15(2)]{Gre22}. Hence, we find that
\[\delta_{1,\zz}^{M_C}(-t\tpi\Omega) =\delta_{1,\zz}^{M_C}(-\tpi C_\theta^*(\Omega)) = \delta_{1,\zz}^{M_C}(-\theta\tpi \Omega) - \tpi\delta_{1,\zz}^{M_C}(\Omega\twistinv) =\theta\delta_{1,\zz}^{M_C}(-\tpi \Omega) =\theta\zeta_A(1).\]
After recalling Carlitz's formula that $\log_C(1) = \zeta_A(1)$, we arrive at
\[\log_C(C_\theta(1)) = \theta \log_C(1),\]
so we have recovered the functional equation for the Carlitz logarithm. We suspect that a similar phenomenon happens in the case of Drinfeld modules and more general $t$-modules. In fact, it seems possible that one could reverse the direction of these calculations to prove our logarithm formulas in \S \ref{S:Log of Drinfeld Modules} in an alternate way. However, there are many details to work out so we leave this as a question to be answered in future work.
\end{remark}

In the case of Drinfeld modules of rank $r$ discussed in the present paper, the matrix $\Psi$ from Corollary \ref{C:Main Corollary Drinfeld modules Intro} is a higher-rank generalization of $\Omega$ discussed above and should be viewed as a higher dimensional theta function. Indeed, it satisfies the functional equation
\begin{equation}\label{E:Phi func equation intro}
\Phi \Psi = \Psi\twistinv,
\end{equation}
where $\Phi\in \Mat_{r\times r}(K[t])$ is defined in \eqref{E:Phi def Drinfeld modules}. Analyzing this functional equation shows that if we denote the top row of $\Psi$ as $(g_1,\dots,g_r)$, then each $g_i$ satisfies
\begin{equation}
t\cdot g_i = \phi_{\theta}^*(g_i),
\end{equation}
where $\phi^*$ is the adjoint of the Drinfeld module $\phi$ given by $\phi^*_{\theta} := \theta + k_1^{1/q} \tau\inv + \dots + k_r^{1/q^r}\tau^{-r}$ (see \cite[\S4.14]{Gos96} for more details). Our Corollary \ref{C:Main Corollary Drinfeld modules Intro} then says that taking the function field Mellin transform of a vector of periods multiplied by this analogue of a theta function gives a Hasse-Weil type zeta value.

\subsection{Tate twists of Drinfeld modules}\label{SS:Tate twists intro}
We also give a version of our main theorems for Drinfeld modules tensored with the positive powers of the Carlitz module. This is akin to taking the Tate twist of a motive, and shifts the value of the corresponding $L$-function allowing us to get formulas for values $n$ larger than $1$. Our result provides an interesting link between certain coordinates of the logarithms of Tate twists of Drinfeld modules and their periods as well as quasi-periods.

In this setting, let $\phi$ be a Drinfeld module as in Theorem \ref{T:Main theorem intro Drinfeld case}. For any $1\leq \ell \leq r-1$, we set $F_{\tau^{\ell}}:\mathbb{C}_{\infty}\to \mathbb{C}_{\infty}$ to be the unique entire function satisfying
\[
F_{\tau^{\ell}}(\theta z)-\theta F_{\tau^{\ell}}(z)=\exp_{\phi}(z)^{q^{\ell}}
\] 
for all $z\in \mathbb{C}_{\infty}$. Let $k$ be a positive integer.  Furthermore, for any $1\leq i \leq rk+1$, by a slight abuse of notation, we let $p_{i}:\mathbb{C}_{\infty}[[\tau]]^{rk+1}\to \mathbb{C}_{\infty}$ and $p_{i}:\mathbb{C}_{\infty}^{rk+1}\to \mathbb{C}_{\infty}$ be the projection onto the $i$-th coordinate of elements in $ \mathbb{C}_{\infty}[[\tau]]^{rk+1} $ and $ \mathbb{C}_{\infty}^{rk+1} $ respectively.

In what follows, similar to the Drinfeld module case, letting $\mathbb{M}_{\text{tens}}$ be a certain subring of $\mathbb{C}_{\infty}[[\tau]]^{rk+1}$, we introduce a continuous and injective map $\varphi_{\text{tens}}:\mathbb{M}_{\text{tens}}\to \Mat_{1\times r}(\mathbb{T})$ in \S\ref{SS:varphi map2}  which extends the isomorphism between the Anderson $t$-motive $M_{\phi\otimes C^{\otimes k}}$ and $\Mat_{1\times r}(\mathbb{C}_{\infty}[t])$, providing a link between the $\tau$-structure and the $t$-structure on the extension of $ M_{\phi\otimes C^{\otimes k}} $.

We set $\mathcal{M}_{\text{tens}}:=\varphi_{\text{tens}}^{-1}$ and for $\zz\in \mathbb{C}_{\infty}^{rk+1}$,  let $\mathcal{M}_{\text{tens},\zz}:=\delta_{1,\zz}^{M_{\phi\otimes C^{\otimes k}}}\circ \varphi_{\text{tens}}^{-1}$. Recall also the fundamental periods $\lambda_1,\dots,\lambda_r$ of $\phi$ and the row vector $\opi=(\lambda_1,\dots,\lambda_r)$. Our next result (restated as Theorem \ref{T:last} later) can be described as follows.

\begin{theorem}\label{T:2} We have 
	\begin{multline*}
p_{rk+1-(j-1)}(\Log_{\phi\otimes C^{\otimes k}})=\\ \begin{cases}
\mathcal{M}_{\text{tens}}\left(\frac{\tilde{\pi}^k}{\omega_C^k(\theta-t)}\opi(\Psi^{\tr})^{(-1)}\right) & \text{ if } j=1\\
\mathcal{M}_{\text{tens}}\left(\frac{\tilde{\pi}^k}{\omega_C^k(t-\theta)}(F_{\tau^{r-(j-1)}}(\lambda_1),\dots, F_{\tau^{r-(j-1)}}(\lambda_r))(\Psi^{\tr})^{(-1)}\right) & \text{ if } 2\leq j \leq r.
\end{cases}
\end{multline*}
Let $\zz\in \mathbb{C}_{\infty}^{rk+1}$ be an element in the domain of convergence of $\Log_{\phi\otimes C^{\otimes k}}$. Then
	\begin{multline*}
p_{rk+1-(j-1)}(\Log_{\phi\otimes C^{\otimes k}}(\zz))=\\ \begin{cases}
\mathcal{M}_{\text{tens},\zz}\left(\frac{\tilde{\pi}^k}{\omega_C^k(\theta-t)}\opi(\Psi^{\tr})^{(-1)}\right) & \text{ if } j=1\\
\mathcal{M}_{\text{tens},\zz}\left(\frac{\tilde{\pi}^k}{\omega_C^k(t-\theta)}(F_{\tau^{r-(j-1)}}(\lambda_1),\dots, F_{\tau^{r-(j-1)}}(\lambda_r))(\Psi^{\tr})^{(-1)}\right) & \text{ if } 2\leq j \leq r.
\end{cases}
\end{multline*}
\end{theorem}

In our last result, we analyze the special values of Goss $L$-functions of Drinfeld modules defined over $\mathbb{F}_q$. Let $\phi$ be a Drinfeld module of rank 2 given as in \eqref{E:drinfeld} such that $k_1,k_2\in \mathbb{F}_q$. Let us also consider the Drinfeld module $\tilde{\phi}$ given by
\[
\tilde{\phi}_{\theta}:=\theta-k_1k_2^{-1}\tau+k_2^{-1}\tau^2.
\] 
There exists a particular relation between certain coordinates of logarithms of Anderson $t$-module $\tilde{\phi}\otimes C^{\otimes k}$ and $L(\phi, k+1)$ (see \S\ref{S:proofoflastresult} for details). Using this relation allows us to obtain the following corollary of Theorem \ref{T:2}. By a slight abuse of notation, we continue to denote the map $\delta_{1,\zz}^{M_{\tilde{\phi}\otimes C^{\otimes k}}}\circ \varphi_{\text{tens}}^{-1}$ by $ \mathcal{M}_{\text{tens},\zz} $.
\begin{corollary}\label{C:corspecvaluestens} Let $\zz_i\in \Mat_{(2k+1)\times 1}(\mathbb{F}_q)$ be the $i$-th unit vector. Furthermore, we set $\mathcal{L}_1:=\frac{\tilde{\pi}^k}{\omega_C^k(t-\theta)}(\tilde{F}_{\tau}(\lambda_1),\tilde{F}_{\tau}(\lambda_2))(\Psi_{\tilde{\phi}}^{\tr})^{(-1)}\in \Mat_{1\times 2}(\mathbb{T})$ and $\mathcal{L}_2:=\frac{\tilde{\pi}^k}{\omega_C^k(t-\theta)}(-\lambda_1, -\lambda_2)(\Psi_{\tilde{\phi}}^{\tr})^{(-1)}\in \Mat_{1\times 2}(\mathbb{T})$.
We have 
	\[
	L(\phi,k+1)=
	\det\begin{bmatrix}
	\mathcal{M}_{\text{tens},\zz_{2k}}(\mathcal{L}_1)&\mathcal{M}_{\text{tens},\zz_{2k+1}}(\mathcal{L}_1)\\
\mathcal{M}_{\text{tens},\zz_{2k}}(\mathcal{L}_2)& \mathcal{M}_{\text{tens},\zz_{2k+1}}(\mathcal{L}_2)
	\end{bmatrix}
\]
	where $\Psi_{\tilde{\phi}}$ is the matrix defined as in \eqref{E:Psidef} with respect to $\tilde{\phi}$ and  $\tilde{F}:\mathbb{C}_{\infty}\to \mathbb{C}_{\infty}$ is the unique entire function satisfying
	\[
	\tilde{F}_{\tau}(\theta z)-\theta \tilde{F}_{\tau}(z)=\exp_{\tilde{\phi}}(z)^{q}
	\] 
	for all $z\in \mathbb{C}_{\infty}$.
\end{corollary}
\begin{remark} As it is explicitly discussed in \cite[\S3.2.1]{GN21}, the Goss $L$-functions of Drinfeld modules of rank $r\geq 2$ are related to Taelman $L$-values of $(r-1)$-st exterior powers of Drinfeld modules as well as their tensor product with Carlitz tensor powers. Although, when $r=2$, $(r-1)$-st exterior power of a Drinfeld module is again a Drinfeld module, when $r>2$, it gives rise to a higher dimensional $t$-module. The extensive study of such $t$-modules (see \cite{GN21,GN24,Ham93}) would allow one to have a similar setup that we have in \S\ref{S:Log for D otimes C}. However, the main obstacle we encounter is the lack of  a technique to explicitly calculate a certain quantity $\eta_n$ (see \S\ref{S:theelementeta}), which is needed to provide an explicit formula as in Theorem \ref{T:last}, for the case of $(r-1)$-st exterior power of Drinfeld modules as well as their tensor product with Carlitz tensor powers. Once we overcome this problem, combining our techniques in the present paper with an extension of the methods used in \cite{Gez21} for Drinfeld modules of rank $r>2$ to calculate special values would lead a generalization of Theorem \ref{T:last} for arbitrary rank Drinfeld modules. We hope to come back to this problem in the near future.
\end{remark}

\subsection{Generality of the Arguments}
In this short subsection we address the arguments and techniques used in the proofs of the main results of the paper and discuss which ones apply specifically to the cases of Drinfeld modules and their tensor powers, which apply in general and which are readily generalizable. 

The foundation of this paper is Theorem \ref{T:log} and the related Proposition \ref{P:deltazeromap}, which comes from \cite{Gre22} and is proven there in great generality: for any abelian (equivalently, $A$-finite) Anderson $A$-module. This includes the case for coefficient rings $A$ of smooth, geometrically connected, projective curves, rather than just $\F_q[t]$, as is considered here. In that sense, the main ideas of this paper should apply in great generality.

However, in order to even state a theorem like Theorem \ref{T:logdrinfeld} for such a general case (let alone prove it), one must work out substantial convergence details, such as showing that the quantity $1/(t-\theta)\opi(\Psi^{\tr})^{(-1)}$ converges in some reasonable Tate algebra and is able to be evaluated under the $\delta_{1,\bz}^{M_G}$ map once it is transferred to an element in $\Mat_{1\times d}(\mathbb{C}_{\infty})[[\tau]]$ using the interaction between the $t-$ and $\tau-$ structure on the Anderson $t$-motive $M_G$. In order to do that in the cases we consider, we take advantage of an explicit description of $\Psi$ as given in \cite[Thm. A]{KP23} for Drinfeld modules (and which is easily extended to tensor product with Carlitz tensor powers). To apply these arguments to more general cases, one would need to first work out a theory equivalent to \cite{KP23} for a larger class of $t$-modules, such as almost strictly pure $t$-modules in the sense of Namoijam and Papanikolas \cite[\S4.5]{NP21}. While this is a future project planned by the authors, it would go substantially beyond the scope of this paper.

Recall the map $\varphi:\mathbb{M}\to \Mat_{1\times r}(\mathbb{T})$ which extends the isomorphism between $M_{\phi}$ and $\Mat_{1\times r}(\mathbb{T})$ for certain Drinfeld modules. Again, the definition of the extension and of the norm we impose on $\mathbb M$ can be generalized easily to arbitrary Drinfeld modules, but the estimates of Lemma \ref{L:t degree of phi} and subsequent arguments showing the injectivity of $\varphi$ in the proof of Proposition \ref{P:inj} rely heavily on the particular structure of the Drinfeld module. It is likely one could construct arguments for the general case that mimic those we give, but again, since we do not need it for the paper, we do not pursue it here.

\subsection{Outline of the paper} In \S \ref{S: Prelim and Background}, we introduce Anderson $t$-modules, Anderson $t$-motives, dual $t$-motives and the formulas obtained by the second author in \cite{Gre22} for the logarithms of Anderson $t$-modules. In \S \ref{S:Log of Drinfeld Modules}, after discussing the tensor construction for Drinfeld modules by using our results in \S \ref{SS:Tensor for Drinfeld modules}, we provide a proof for Theorem \ref{T:Main theorem intro Drinfeld case} and Corollary \ref{C:Main Corollary Drinfeld modules Intro}. Finally, in \S \ref{S:Log for D otimes C}, we discuss the structure of a certain motivic map (see \S \ref{SS:delta_0 map D otimes C}) and then, using our ideas established in \S \ref{SS:Tensor for Drinfeld modules}, we prove Theorem \ref{T:2}.

\subsection*{Acknowledgments} The authors would like to express their gratitude to Gebhard B\"{o}ckle, Chieh-Yu Chang, Matt Papanikolas, Federico Pellarin and Wei-Lun Tsai for fruitful discussions and useful suggestions. The authors are also grateful to the anonymous referee for careful reading of the manuscript and several valuable suggestions on the technical details of the present paper that clarify the conceptual structure of our results. The second author expresses grateful support for funding from the state of Louisiana Board of Regents and from the NSF. This material is based upon work supported by the National Science Foundation under Grant No. (2302399). The first author acknowledges support by Deutsche Forschungsgemeinschaft (DFG) through CRC-TR 326 `Geometry and Arithmetic of Uniformized Structures', project number 444845124. The first author was also supported by NSTC Grant 113-2115-M-007-001-MY3.

\section{Preliminaries and background}\label{S: Prelim and Background}
Our goal in this section is to review the notion of Anderson $t$-modules, Anderson $t$-motives and dual $t$-motives as well as a formula for the logarithms of Anderson $t$-modules derived in \cite{Gre22}. The main references for our exposition are \cite{And86}, \cite{BP20}, \cite{Gre22} and \cite[\S2.3--2.5]{HJ20}.

\subsection{Anderson $t$-modules} For any matrix $\mathcal{C}=(m_{\mu \nu})\in \Mat_{d_1\times d_2}(\mathbb{C}_{\infty})$ and $i\in \mathbb{Z}$, we define \textit{the $i$-th twist of $\mathcal{C}$} by $\mathcal{C}^{(i)}:=(m_{\mu \nu}^{q^i})$. Furthermore, we let 
\[
\Mat_{d_1\times d_2}(\mathbb{C}_{\infty})[[\tau]]:=\left\{\sum_{i\geq 0}\mathcal{C}_i\tau^i \ \ | \ \ \mathcal{C}_i\in \Mat_{d_1\times d_2}(\mathbb{C}_{\infty})\right\}
\]
and when $d=d_1=d_2$, we define the non-commutative power series ring $\Mat_{d}(\mathbb{C}_{\infty})[[\tau]]$ subject to the condition 
\[
\tau \mathcal{C}=\mathcal{C}^{(1)}\tau.
\]
We also let $\Mat_{d}(\mathbb{C}_{\infty})[\tau]$ be the subring of $\Mat_{d}(\mathbb{C}_{\infty})[[\tau]]$ consisting of polynomials in $\tau$.

\begin{definition} \label{D:t-modules}
	\begin{itemize}
		\item[(i)] \textit{An Anderson $t$-module $G$ of dimension $d\geq 1$}  is a tuple $(\mathbb{G}_{a/\mathbb{C}_{\infty}}^d,\phi)$ consisting of the $d$-dimensional additive algebraic group $\mathbb{G}^d_{a/\mathbb{C}_{\infty}}$ defined over $\mathbb{C}_{\infty}$ and an $\mathbb{F}_q$-algebra homomorphism $\phi:A\to \Mat_{d}(\mathbb{C}_{\infty})[\tau]$ given by 
		\begin{equation}\label{E:tmodule}
		\phi_{\theta}:=d[\theta]+A_1\tau+\dots+A_{\ell}\tau^{\ell}
		\end{equation}
		so that $\ell\in \mathbb{Z}_{\geq 1}$ and 
		$
		d[\theta]:=\theta\Id_{d}+\mathfrak{N}
		$
		for some nilpotent matrix $\mathfrak{N}$.
		\item[(ii)] The morphisms between Anderson $t$-modules $G_1=(\mathbb{G}^{d_1}_{a/\mathbb{C}_{\infty}},\phi)$ and $G_2=(\mathbb{G}^{d_2}_{a/\mathbb{C}_{\infty}},\psi)$ are defined to be the morphisms $g:\mathbb{G}^{d_1}_{a/\mathbb{C}_{\infty}}\to \mathbb{G}^{d_2}_{a/\mathbb{C}_{\infty}}$ of algebraic groups satisfying $g \phi_{\theta}=\psi_{\theta}g$.
	\end{itemize}
\end{definition}

We define  $G(\mathbb{C}_{\infty}):=\Mat_{d\times 1}(\mathbb{C}_{\infty})$ equipped with the $A$-module structure given by 
\[
\theta \cdot \zz=\phi_{\theta}(z):=d[\theta]\zz+A_1\zz^{(1)}+\dots+A_{\ell}\zz^{(\ell)}, \ \ \zz\in \Mat_{d\times 1}(\mathbb{C}_{\infty}).
\]

We also consider $\Lie(G)(\mathbb{C}_{\infty}):=\Mat_{d\times 1}(\mathbb{C}_{\infty})$ which is equipped with the $A$-module action defined by 
\[
\theta \cdot \zz:=d[\theta]\zz.
\]

It is known, due to Anderson \cite[\S2]{And86}, that there exists a unique infinite series $\Exp_{G}:=\sum_{i\geq 0}Q_i\tau^i \in \Mat_{d}(\mathbb{C}_{\infty})[[\tau]]$ satisfying $Q_0=\Id_d$ and 
\[
\Exp_{G}d[\theta]=\phi_{\theta}\Exp_{G}.
\]
Moreover, it induces an entire function $\Exp_{G}:\Lie(G)(\mathbb{C}_{\infty})\to G(\mathbb{C}_{\infty})$ given by 
\begin{equation}\label{E:Exp def}
\Exp_{G}(\zz):=\sum_{i=0}^{\infty}Q_i\zz^{(i)}.
\end{equation}

We let $\Log_{G}:=\sum_{i\geq 0} P_i\tau^i \in \Mat_{d}(\mathbb{C}_{\infty})[[\tau]]$ be the formal inverse of $\Exp_{G}\in \Mat_{d}(\mathbb{C}_{\infty})[[\tau]]$. On a certain subset $\mathcal{D}_{G}$ of $G(\mathbb{C}_{\infty})$, $\Log_{G}$ induces a vector valued function $\Log_{G}:\mathcal{D}_{G}\to \Lie(G)(\mathbb{C}_{\infty})$ defined by 
\[
\Log_{G}(\zz):=\sum_{i=0}^{\infty}P_i\zz^{(i)}.
\]
For further details on the exponential and the logarithm function, we refer the reader to \cite[\S2.5.1]{HJ20}.

In what follows, we provide some examples of Anderson $t$-modules.
\begin{example}\label{Ex:1}
	\begin{itemize}
		\item[(i)] Any Drinfeld module $\phi$ is an Anderson $t$-module $(\mathbb{G}_{a/\mathbb{C}_\infty},\phi)$ of dimension one. 
		\item[(ii)] Let $C:A\to \mathbb{C}_{\infty}[\tau]$ be the Carlitz module and  $k\in \mathbb{Z}_{\geq 1}$. We consider \textit{the $k$-th tensor power of the Carlitz module} $C^{\otimes k}:=(\mathbb{G}_{a/\mathbb{C}_{\infty}}^k, \psi)$ where $\psi:A\to \Mat_{k}(\mathbb{C}_{\infty})[\tau]$ is given by (see \cite{AT90})
		\[
		\psi_{\theta}:=\begin{pmatrix}
		\theta&1& & \\
		& \ddots&\ddots & \\
		& & \ddots & 1\\
		& & & \theta  
		\end{pmatrix}+\begin{pmatrix}
		0&\dots&\dots &0 \\
		\vdots& & &\vdots \\
		0& &  & \vdots\\
		1&0&\dots& 0 
		\end{pmatrix}\tau.
		\]
		\item[(iii)] Let $\phi$ be a Drinfeld module of rank $r$ given as in \eqref{E:drinfeld}. We define \textit{the tensor product $\phi$ and the $k$-th tensor power of the Carlitz module} as $\phi\otimes C^{\otimes k}:=(\mathbb{G}_{a/\mathbb{C}_{\infty}}^{rk+1},\rho)$ where 
		\[
		\rho:A\to \Mat_{rk+1}(\mathbb{C}_{\infty})[\tau]
		\] is given by
		\[
		\rho_{\theta}:=	\begin{pmatrix}
		\theta&\cdots &0&\bovermat{$rk+1-r$}{1&0& \cdots & 0}  \\
		& \ddots& & \ddots& \ddots & &\vdots \\
		& &\ddots& &\ddots & \ddots&0\\
		& &  & \theta&\cdots&0 & 1\\
		& &  & & \theta&\cdots& 0\\
		& &  & & &\ddots& \vdots\\
		& &  & & & & \theta\\
		\end{pmatrix}+\begin{pmatrix}
		0&\cdots&\cdots & \cdots &\cdots & \cdots & 0\\
		\vdots& & & & & &\vdots\\
		0& & & & & & 0\\
		1&0 &\cdots &\cdots &\cdots &\cdots &0\\
		&\ddots& \ddots&  & & &\vdots\\
		& &1& \ddots & & & \vdots\\
		k_1&\cdots&\cdots &k_r&0&\cdots& 0
		\end{pmatrix}\tau.
		\]
	\end{itemize}
	For more details on the tensor product of Drinfeld modules of arbitrary rank with the tensor powers of the Carlitz module, we refer the reader to  \cite{Gez21, GN21, Ham93, Hua23, Kha21}.
\end{example}

\begin{remark}\label{R:roclog} Let $\phi$ be a Drinfeld module of rank $r$ as in \eqref{E:drinfeld}. We let $\mathcal{N}_{\phi}$ be a subset of $\{1,\dots,r\}$ containing indices $i$ such that $k_i\neq 0$. For any $i\in \mathcal{N}_{\phi}$, we further let $\nu_i:=\frac{\log_q|k_i|-q^i}{q^i-1}$. Then, in \cite[Cor. 4.5]{KP23}, Khaochim and Papanikolas showed that any $z\in \mathbb{C}_{\infty}$ satisfying $|z|<q^{-\nu_m}$ where $m$ is the smallest element in $\mathcal{N}_{\phi}$ such that  $\nu_m\geq \nu_i$ for any $i\in \mathcal{N}_{\phi}$ lies in $\mathcal{D}_{\phi}  $. For later use, we also emphasize that for our choice of Drinfeld module $\phi$ as in Theorem \ref{T:Main theorem intro Drinfeld case}, we see that $\log_{\phi}$ converges at any element $z\in \mathbb{C}_{\infty}$ satisfying $|z|<q^{q^r/(q^r-1)}$. For the tensor powers of the Carlitz module, Anderson and Thakur \cite[Prop. 2.4.3]{AT90} provided a condition for the elements that lie in $\mathcal{D}_{C^{\otimes k}}$. Although, the set $\mathcal{D}_{\phi\otimes C^{\otimes k }}$ is not explicitly studied in the general case, in Proposition \ref{P:02}(ii), we will provide some analysis on the norm of elements lying in $\mathcal{D}_{\phi\otimes C^{\otimes k }}$ under a particular condition on the coefficients of $\phi$.    
\end{remark}

Consider $\Lambda_G:=\Ker(\Exp_{G})\subset \Lie(G)(\mathbb{C}_{\infty})$. By the work of Anderson \cite[Lem. 2.4.1]{And86}, we know that, under a certain condition on $G$,  $\Lambda_G$ forms a finitely generated and discrete $A$-module. We call any non-zero element of $\Lambda_G$ \textit{a period of $G$}. %Furthermore, if the set $\{\lambda_1,\dots,\lambda_s\}$ forms an $A$-basis for $\Lambda_G$, we call each $\lambda_i$ \textit{a fundamental period of $G$}.
Indeed, by \cite[Thm.4]{And86}, when $G$ is the Anderson $t$-module either in Example \ref{Ex:1}(i) or in \ref{Ex:1}(iii), $\Lambda_G$ is  free of rank $r$ as an $A$-module. Moreover, if $G$ is the $k$-th tensor power of the Carlitz module, then $\Lambda_G$ is free of rank one. 
\subsection{Anderson generating functions}\label{SS:Anderson Generating Functions} 
For any $c\in \mathbb{C}_{\infty}^{\times}$, we define the Tate algebra
\[
\mathbb{T}_{c}:=\left\{ g=\sum_{i\geq 0}a_it^i\in \mathbb{C}_{\infty}[[t]] \ \ | \ \ \inorm{c^ia_i}\to 0 \ \ \text{ as } i\to \infty \right\}.
\]
It is equipped with the multiplicative norm $\lVert \cdot\rVert_c$ given by 
\[
\lVert g\rVert_c:=\max\{|c^i||a_i| \ \ | \ \ i\geq 0\}.
\]
To ease the notation, we denote $\mathbb{T}_{1}$ by $\mathbb{T}$ and $\lVert \cdot \rVert_1$ by $\lVert \cdot \rVert$. By using a slight abuse of notation, we further extend the norm $\lVert \cdot \rVert$ on $\Mat_{m\times \ell}(\mathbb{T})$ so that for any $B=(b_{ij})\in \Mat_{m\times \ell}(\mathbb{T})$, 
\[
\lVert B\rVert:=\max_{i,j} \lVert b_{ij}\rVert.
\]

Let $\phi$ be a Drinfeld module of rank $r$ given as in \eqref{E:drinfeld}. In what follows, we define a certain element in $\mathbb{T}$ which will be later useful to describe a particular property of Anderson $t$-motives of Drinfeld modules. For any $z\in \mathbb{C}_{\infty}$, \textit{the Anderson generating function $s_{\phi}(z;t)$} is given by 
\[
s_{\phi}(z;t):=\sum_{i= 0}^{\infty}\exp_{\phi}\left(\frac{z}{\theta^{i+1}}\right)t^i \in \mathbb{T}.
\]
Let $t$ be a variable over $\mathbb{C}_{\infty}$. For any $f=\sum_{i\geq 0}a_it^i\in \mathbb{C}_{\infty}[[t]]$ and $j\in \mathbb{Z}$, we set $f^{(j)}:=\sum_{i\geq 0}a_i^{q^j}t^i\in \mathbb{C}_{\infty}[[t]]$. We now state a fundamental property of Anderson generating functions due to Pellarin.
\begin{proposition}[{Pellarin,\cite[\S4.2]{Pel08}}]  \label{P:AGF} Let $\lambda\in \Ker(\exp_{\phi})$. Then 
	\[
	(t-\theta)s_{\phi}(\lambda;t)=k_1s_{\phi}(\lambda;t)^{(1)}+\dots+k_rs_{\phi}(\lambda;t)^{(r)}.
	\]
\end{proposition}

\subsection{Anderson $t$-motives}\label{SS:And t-motives} We define the non-commutative ring $\mathbb{C}_{\infty}[t,\tau]:=\mathbb{C}_{\infty}[t][\tau]$ with respect to the condition
$
\tau f=f^{(1)}\tau$ where $f=\sum_{i\geq 0}a_it^i\in \mathbb{C}_{\infty}[t]$.

\begin{definition}
	\begin{itemize}
		\item[(i)] \textit{An Anderson $t$-motive $M$} is a left $\mathbb{C}_{\infty}[t,\tau]$-module which is free and finitely generated over $\mathbb{C}_{\infty}[t]$ and $\mathbb{C}_{\infty}[\tau]$ (possibly of different ranks) such that there exists a non-negative integer $\mu$ satisfying
		\[
		(t-\theta)^{\mu}M\subset \tau M.
		\] 
		\item[(ii)] Morphisms of Anderson $t$-motives are given by morphisms of left $\mathbb{C}_{\infty}[t,\tau]$-modules.
		\item[(iii)] Let $M_1$ and $M_2$ be two Anderson $t$-modules. \textit{The tensor product of $M_1$ and $M_2$} is the Anderson $t$-motive $M_1\otimes_{\mathbb{C}_{\infty}[t]}M_2$ where $\tau$ acts diagonally.
	\end{itemize}
\end{definition}

Let $\bm\in \Mat_{d\times 1}(M)$ be a $\mathbb{C}_{\infty}[t]$-basis for $M$ and $\mathfrak{Q}\in  \Mat_r(\mathbb{C}_{\infty}[t])$ be such that 
\[
\tau \cdot \bm=\mathfrak{Q}\bm.
\] 
We call $M$ \textit{rigid analytically trivial} if there exists $\Upsilon\in \GL_r(\mathbb{T})$ such that 
\[
\Upsilon^{(1)}=\mathfrak{Q}\Upsilon.
\]
We also call $\Upsilon$ \textit{a rigid analytic trivialization of $M$}.

Anderson \cite[Thm. 1]{And86} constructs a functor which attaches to each Anderson $t$-motive an Anderson $t$-module, which are in literature called \textit{abelian $t$-modules}. This moreover gives an anti-equivalence of categories of abelian $t$-modules and Anderson $t$-motives. We briefly describe this functor now. Given an Anderson $t$-module $G=(\mathbb{G}^d_{a/\mathbb{C}_{\infty}},\phi)$, there exists a unique Anderson $t$-motive $M_G$ given by the group of morphisms $\mathbb{G}^d_{a/\mathbb{C}_{\infty}}\to \mathbb{G}_{a/\mathbb{C}_{\infty}}$ of $\mathbb{C}_{\infty}$-algebraic groups. This group of morphisms is naturally a $\mathbb{C}_{\infty}[\tau]$-module and is isomorphic to $\Mat_{1\times d}(\mathbb{C}_{\infty})[\tau]$ as $\mathbb{C}_{\infty}[\tau]$-modules. It is equipped with a $\mathbb{C}_{\infty}[t,\tau]$-module structure given by 
\[
ct^i\cdot m:=c\circ m \circ \phi_{\theta^i}, \ \ m\in M_G.
\]

In what follows, we describe the Anderson $t$-motives corresponding to the Anderson $t$-modules given in Example \ref{Ex:1}.

\subsubsection{Anderson $t$-motive of Drinfeld modules}\label{SS:t-motive for Drinfeld} Let $\phi$ be the Drinfeld module of rank $r$ given as in \eqref{E:drinfeld}. We define $M_{\phi}:=\mathbb{C}_{\infty}[\tau]$ and equip it with the $\mathbb{C}_{\infty}[t]$-module structure given by 
\[
ct^i\cdot g\tau^j:=cg\tau^j\phi_{\theta^i},  \ \ c,g\in \mathbb{C}_{\infty}.
\]
One can see that $M_{\phi}$ forms a left $\mathbb{C}_{\infty}[t,\tau]$-module, satisfying $(t-\theta)M_{\phi}\subset \tau M_{\phi}$, which is free and finitely generated over $\mathbb{C}_{\infty}[t]$ and $\mathbb{C}_{\infty}[\tau]$.  We define the matrix
\[
\Theta:=\begin{pmatrix}
&1& & &  \\
& & \ddots & & \\
& & & \ddots & \\
& & & &  1\\
\frac{t-\theta}{k_r}&-\frac{k_1}{k_r} & \dots & \dots & -\frac{k_{r-1}}{k_r}
\end{pmatrix}\in \GL_r(\mathbb{T})\cap \Mat_r(\mathbb{C}_{\infty}[t]).
\]
We choose $m:=[m_1,\dots,m_r]^{\tr}\in \Mat_{r\times 1}(M_{\phi})$ to be a $\mathbb{C}_{\infty}[t]$-basis for $M_{\phi}$ so that 
\[
\tau \cdot m=\Theta m.
\]
Observe that $\{m_1\}$ forms a $\mathbb{C}_{\infty}[\tau]$-basis for $M_{\phi}$. 

Let $\{\lambda_1,\dots,\lambda_r\}$ be an $A$-basis for the period lattice $\Lambda_{\phi}$. For any $i\in \{1,\dots,r\}$, we define the Anderson generating function $f_i:=s_{\phi}(\lambda_i;t)$. Consider the matrix 
\begin{equation}\label{E:Upsilon}
\Upsilon:=\begin{pmatrix} f_1 & \cdots & \cdots & f_r\\
f_1^{(1)} & \cdots & \cdots & f_r^{(1)}\\
\vdots  & & & \vdots\\
f_1^{(r-1)} & \cdots & \cdots &f_r^{(r-1)}
\end{pmatrix}\in \Mat_{r\times r}(\mathbb{T}).
\end{equation}
By \cite[\S4.2]{Pel08}, we know that $\Upsilon\in \GL_r(\mathbb{T})$ and moreover it satisfies 
\[
\Upsilon^{(1)}=\Theta \Upsilon.
\]
Hence $M_{\phi}$ is rigid analytically trivial.

For later use, we also consider another $\mathbb{C}_{\infty}[t]$-basis 
\begin{multline*}
\mathfrak{c}^{\phi}:=[\mathfrak{c}^{\phi}_1,\dots, \mathfrak{c}^{\phi}_r]^{\tr}:=[k_1^{(-1)}m_1+k_2^{(-1)}m_2+\dots+k_r^{(-1)}m_r,k_2^{(-2)}m_1+k_3^{(-2)}m_2+\dots +k_r^{(-2)}m_{r-1}
\\,\dots,k_{r-1}^{(1-r)}m_1+k_r^{(1-r)}m_2, k_r^{(-r)}m_1]^{\tr}\in \Mat_{r\times 1}(M_{\phi})
\end{multline*}
and note that 
\begin{equation}\label{E:tauaction}
\tau \cdot \mathfrak{c}^{\phi}=\Phi^{\tr} \mathfrak{c}^{\phi}
\end{equation}
where 
\begin{equation}\label{E:Phi def Drinfeld modules}
\Phi:=\begin{pmatrix}
&1& & &  \\
& & \ddots & & \\
& & & \ddots & \\
& & & &  1\\
\frac{t-\theta}{k_r^{(-r)}}&-\frac{k_1^{(-1)}}{k_r^{(-r)}} & \dots & \dots & -\frac{k_{r-1}^{(-(r-1))}}{k_r^{(-r)}}
\end{pmatrix}\in  \GL_r(\mathbb{T})\cap \Mat_r(\mathbb{C}_{\infty}[t]).
\end{equation}

\subsubsection{Anderson $t$-motive of the tensor powers of the Carlitz module}\label{SS:t-motive for powers of Carlitz} Let $k\in \mathbb{Z}_{\geq 1}$. We consider the left $\mathbb{C}_{\infty}[t,\tau]$-module 
\[
M_{C^{\otimes k}}:=M_{C}\otimes_{\mathbb{C}_{\infty}[t]} \cdots \otimes_{\mathbb{C}_{\infty}[t]} M_{C}=\mathbb{C}_{\infty}[\tau]\otimes_{\mathbb{C}_{\infty}[t]}\cdots \otimes_{\mathbb{C}_{\infty}[t]} \mathbb{C}_{\infty}[\tau]
\]
so that $\tau$ acts diagonally. Let $m_1$ be a basis for $M_C$ as a $\C_\infty[t]$-module. Then $m:=m_1\otimes \cdots \otimes m_1$ is a $\mathbb{C}_{\infty}[t]$-basis for $M_{C^{\otimes k}}$ so that 
\[
\tau m=(t-\theta)^k m.
\]
Moreover, the set $\{m, (t-\theta)m,\dots,(t-\theta)^{k-1}m\}$ forms a $\mathbb{C}_{\infty}[\tau]$-basis for $M_{C^{\otimes k}}$ and hence it is of dimension $k$ over $\mathbb{C}_{\infty}[\tau]$. In particular, $M_{C^{\otimes k}}\cong \Mat_{1\times k}(\mathbb{C}_{\infty})[\tau]$ as $\mathbb{C}_{\infty}[\tau]$-modules. 

We now fix a $(q-1)$-st root of $-\theta$ and define \textit{the Anderson-Thakur element} $\omega_C$ by 
\begin{equation}\label{D:omega_C}
\omega_C:=(-\theta)^{1/(q-1)}\prod_{j=0}^{\infty}\left(1-\frac{t}{\theta^{q^j}}\right)^{-1}\in \mathbb{T}.
\end{equation}

One can observe that $(\omega_C^k)^{(1)}=(t-\theta)^k\omega_C^k$ and hence $M_{C^{\otimes k}}$ is rigid analytically trivial.

\subsubsection{Anderson $t$-motive of the tensor product of Drinfeld modules with the tensor powers of the Carlitz module}\label{SS:t-motive for D tensor C} We consider the left  $\mathbb{C}_{\infty}[t,\tau]$-module 
\[
M_{\phi\otimes C^{\otimes k}}:=M_{\phi}\otimes_{\mathbb{C}_{\infty}[t]} M_{C^{\otimes k}}=\mathbb{C}_{\infty}[\tau]\otimes_{\mathbb{C}_{\infty}[t]}\left(\mathbb{C}_{\infty}[\tau]\otimes_{\mathbb{C}_{\infty}[t]}\cdots \otimes_{\mathbb{C}_{\infty}[t]} \mathbb{C}_{\infty}[\tau]\right)
\]
so that $\tau$ acts diagonally. Observe that $(t-\theta)^{k+1}M_{\phi\otimes C^{\otimes k}}\subset \tau M_{\phi\otimes C^{\otimes k}}$. Moreover, $M_{\phi\otimes C^{\otimes k}}$ is free and finitely generated over $\mathbb{C}_{\infty}[t]$ and $\mathbb{C}_{\infty}[\tau]$. We consider a $\mathbb{C}_{\infty}[t]$-basis $\mathfrak{m}$ for $M_{\phi\otimes C^{\otimes k}}$ given by $\mathfrak{m}:=[\mathfrak{m}_1,\dots,\mathfrak{m}_r]^{\tr}:=[m_1\otimes m,\dots,m_r\otimes m]^{\tr}$, where $m_i$ are the basis elements from \S \ref{SS:t-motive for Drinfeld} and $m$ is from \S \ref{SS:t-motive for powers of Carlitz}. Note that
\[ 
\tau \cdot \mathfrak{m}=(t-\theta)^k\Theta \mathfrak{m}.
\]
Let $\tilde{\Upsilon}:=\omega_C^k\Upsilon\in \GL_r(\mathbb{T})$. Then it is easy to see that 
$
\tilde{\Upsilon}^{(1)}=(t-\theta)^k\Theta\tilde{\Upsilon}
$
and hence $M_{\phi\otimes C^{\otimes k}}$ is rigid analytically trivial.

We further define another $\mathbb{C}_{\infty}[t]$-basis 
\[
\mathfrak{c}:=[\mathfrak{c}_1,\dots, \mathfrak{c}_r]^{\tr}:=[\mathfrak{c}^{\phi}_1\otimes m, \dots, \mathfrak{c}^{\phi}_r\otimes m]^{\tr}.
\]
Moreover, we note that 
\begin{equation}\label{E:tauactiontens}
\tau \cdot \mathfrak{c}=(t-\theta)^k\Phi^{\tr} \mathfrak{c}.
\end{equation}
Lastly, we define a $\mathbb{C}_{\infty}[\tau]$-basis 
\begin{multline*}
\mathfrak{g}:=[g_1,\dots,g_{rk+1}]^{\tr}:=[\mathfrak{m}_1, \mathfrak{m}_2,\dots, \mathfrak{m}_r,(t-\theta)\mathfrak{m}_1, (t-\theta)\mathfrak{m}_2,\dots, (t-\theta)\mathfrak{m}_r,\dots,\\
(t-\theta)^{k-1}\mathfrak{m}_1, (t-\theta)^{k-1}\mathfrak{m}_2,\dots, (t-\theta)^{k-1}\mathfrak{m}_r, (t-\theta)^{k}\mathfrak{m}_1 ]^{\tr}.
\end{multline*}
One now sees that 
\[
t\cdot \mathfrak{g}=\rho_{\theta}\mathfrak{g}
\]
where $\rho_{\theta}$ is given as in Example \ref{Ex:1}(iii).

\subsection{Dual $t$-motives}\label{SS:Dual t-motives} We define $\mathbb{C}_{\infty}[t,\sigma]:=\mathbb{C}_{\infty}[t][\sigma]$ to be the ring of polynomials of $\sigma$ with coefficients in $\mathbb{C}_{\infty}[t]$ subject to the condition
\[
\sigma f=f^{(-1)}\sigma, \ \ f\in \mathbb{C}_{\infty}[t].
\]
We further define the $*$-operation on elements in $\mathbb{C}_{\infty}[\tau]$ by 
\[
g^{*}:=\sum_{i\geq 0}c_i^{(-i)}\sigma^i, \ \ g=\sum_{i\geq 0}c_i\tau^i.
\]
We extend this operation to elements in $\Mat_{d}(\mathbb{C}_{\infty})[\tau]$ by defining $\mathcal{C}^{*}:=((m^{*}_{\mu \nu}))^{\tr}$ for any $\mathcal{C}=(m_{\mu \nu})\in \Mat_{d}(\mathbb{C}_{\infty})[\tau]$.

\begin{definition}
	\begin{itemize}
		\item[(i)] \textit{A dual $t$-motive $N$} is a left $\mathbb{C}_{\infty}[t,\sigma]$-module which is free and finitely generated over $\mathbb{C}_{\infty}[t]$ and $\mathbb{C}_{\infty}[\sigma]$ such that there exists $\ell\in \mathbb{Z}_{\geq 0}$ satisfying
		\[
		(t-\theta)^{\ell} N\subset \sigma N.
		\] 
		\item[(ii)] The morphisms of dual $t$-motives are given by left $\mathbb{C}_{\infty}[t,\sigma]$-module homomorphisms.
		\item[(iii)] The tensor product of dual $t$-motives $N_1$ and $N_2$ is defined to be the left $\mathbb{C}_{\infty}[t,\sigma]$-module $N_1\otimes N_2:=N_1\otimes_{\mathbb{C}_{\infty}[t]} N_2$ where $\sigma$ acts diagonally.
	\end{itemize}
\end{definition}

Let $\bn\in \Mat_{r\times 1}(N)$ be a $\mathbb{C}_{\infty}[t]$-basis for $N$ and $\mathfrak{Z}\in \Mat_r(\mathbb{C}_{\infty}[t])$ be such that 
\[
\sigma \cdot \bn=\mathfrak{Z}\bn.
\] 
We say that $M$ is \textit{rigid analytically trivial} if there exists $\Psi\in \GL_r(\mathbb{T})$ such that 
\[
\Psi^{(-1)}=\mathfrak{Z}\Psi.
\]
We further call $\Psi$ \textit{a rigid analytic trivialization of $N$}.

In his unpublished work (see also \cite[\S2.5]{HJ20}), Anderson constructs a functor which attaches to each dual $t$-motive an Anderson $t$-module, which are in literature called \textit{$A$-finite $t$-modules}. The aforementioned functor indeed describes an equivalence between the category of $A$-finite Anderson $t$-modules and the category of dual $t$-motives. In other words, for any $A$-finite Anderson $t$-module $G=(\mathbb{G}^d_{a/\mathbb{C}_{\infty}},\phi)$, there exists a unique Anderson dual $t$-motive $N_G:=\Mat_{1\times d}(\mathbb{C}_{\infty}[\sigma])$  equipped with a $\mathbb{C}_{\infty}[t,\tau]$-module structure given by 
\[
ct^i\cdot n:=cn \phi^{*}_{\theta^i}, \ \ n\in N_G.
\]

\begin{remark}
	We emphasize that since our Anderson $t$-motives and dual $t$-motives always correspond to abelian $t$-modules and $A$-finite $t$-modules respectively, the Anderson $t$-modules considered throughout this paper will be always one of those kind. Indeed, in \cite[Thm. A]{Mau21}, Maurischat showed that being an abelian $t$-module is equivalent to being an $A$-finite $t$-module. 
\end{remark}

In what follows, we describe the dual $t$-motives corresponding to Anderson $t$-modules given in Example \ref{Ex:1}.

\subsubsection{Dual $t$-motive of Drinfeld modules}\label{SS:Dual t-motives for Drinfeld} Let $\phi$ be a Drinfeld module given as in \eqref{E:drinfeld}. We define $N_{\phi}$ to be the  $\mathbb{C}_{\infty}[\sigma]$-module $\mathbb{C}_{\infty}[\sigma]$ equipped with the $\mathbb{C}_{\infty}[t]$-module action given by 
\[
ct^i\cdot g\sigma^j:=cg\sigma^j\phi_{\theta^i}^{*},  \ \  c,g\in \mathbb{C}_{\infty}.
\]
It is free and finitely generated over $\mathbb{C}_{\infty}[t]$ and $\mathbb{C}_{\infty}[\sigma]$ satisfying $(t-\theta)N_{\phi}\subset \sigma N_{\phi}$. We choose a $\mathbb{C}_{\infty}[t]$-basis $\mathfrak{d}^{\phi}:=[\mathfrak{d}^{\phi}_1,\dots,\mathfrak{d}^{\phi}_r]^{\tr}\in \Mat_{r\times 1}(N_{\phi})$ for $N_{\phi}$ satisfying 
\begin{equation}\label{E:tauact}
\sigma \cdot \mathfrak{d}^{\phi}=\Phi \mathfrak{d}^{\phi}.
\end{equation}
Moreover, $\{\mathfrak{d}^{\phi}_1\}$ forms a $\mathbb{C}_{\infty}[\sigma]$-basis for $N_{\phi}$. 

Following the notation in \cite[\S3.3]{CP12}, set 
\begin{equation}\label{E:V def}
V:=\begin{pmatrix}
k_1&k_2^{(-1)} &k_3^{(-2)}& \dots &k_r^{(1-r)}\\
\vdots &\vdots&\vdots&  \iddots & \\
\vdots & \vdots& k_r^{(-2)} & & \\
\vdots &k_r^{(-1)} & & & \\
k_r &  & & & 
\end{pmatrix}\in \GL_r(\mathbb{C}_{\infty})
\end{equation}
and consider the matrix $\Psi:=V^{-1}((\Upsilon^{(1)})^{\tr})^{-1}\in \GL_r(\mathbb{T})$. Then, by Proposition \ref{P:AGF}, we obtain $(\Upsilon^{(1)})^{\tr}=\Upsilon^{\tr}\Theta^{\tr}$. Moreover, one has 
\begin{equation}\label{E:Phieq}
V^{(-1)}\Phi=\Theta^{\tr}V.
\end{equation}
Thus, we have $\Psi^{(-1)}=\Phi\Psi$ and hence $N_{\phi}$ is rigid analytically trivial.   

\subsubsection{Dual $t$-motive of the tensor powers of the Carlitz module}\label{SS:Dual t-motive for power of Carlitz} Set
\[
N_{C^{\otimes k}}:=N_{C}\otimes_{\mathbb{C}_{\infty}[t]} \cdots \otimes_{\mathbb{C}_{\infty}[t]} N_{C}=\mathbb{C}_{\infty}[\sigma]\otimes_{\mathbb{C}_{\infty}[t]}\cdots \otimes_{\mathbb{C}_{\infty}[t]} \mathbb{C}_{\infty}[\sigma]
\]
and equip it with the diagonal $\sigma$-action. Thus $N_{C^{\otimes k}}$ is a left $\mathbb{C}_{\infty}[t,\sigma]$-module. One can choose $\mathbb{C}_{\infty}[t]$-basis $n:=\mathfrak{d}^{C}_1\otimes \cdots \otimes\mathfrak{d}^{C}_1$ for $N_{C^{\otimes k}}$ so that 
\[
\sigma n=(t-\theta)^kn.
\]
On the other hand, the set $\{ n,(t-\theta)n,\dots,(t-\theta)^{k-1}n\}$ forms a $\mathbb{C}_{\infty}[\sigma]$-basis for $N_{C^{\otimes k}}$ and hence $N_{C^{\otimes k}}\cong \Mat_{1\times k}(\mathbb{C}_{\infty})[\sigma]$. Consider the element $\Omega:=(\omega_C^{(1)})^{-1}$. It can be easily seen that $(\Omega^k)^{(-1)}=(t-\theta)^k\Omega$ and thus implies the rigid analytic triviality of $N_{C^{\otimes k}}$.

\subsubsection{Dual $t$-motive of the tensor product of Drinfeld modules with the tensor powers of the Carlitz module}\label{SS:Dual t-motive for D tensor C} We set 
\[
N_{\phi\otimes C^{\otimes k}}:=N_{\phi}\otimes_{\mathbb{C}_{\infty}[t]}N_{C^{\otimes k}}=\mathbb{C}_{\infty}[\sigma]\otimes_{\mathbb{C}_{\infty}[t]}\left(\mathbb{C}_{\infty}[\sigma]\otimes_{\mathbb{C}_{\infty}[t]}\cdots \otimes_{\mathbb{C}_{\infty}[t]} \mathbb{C}_{\infty}[\sigma]\right)
\]
and equip it with the diagonal $\sigma$-action. It can be seen that $N_{\phi\otimes C^{\otimes k}}$ forms a left $\mathbb{C}_{\infty}[t,\sigma]$-module and it is also a free and finitely generated over $\mathbb{C}_{\infty}[t]$ and $\mathbb{C}_{\infty}[\sigma]$. Moreover, 
\[
(t-\theta)^{k+1}N_{\phi\otimes C^{\otimes k}}\subset \sigma N_{\phi\otimes C^{\otimes k}}.
\] 
We consider the $\mathbb{C}_{\infty}[t]$-basis for $N_{\phi\otimes C^{\otimes k}}$ given by 
$
\mathfrak{d}:=[\mathfrak{d}_1,\dots,\mathfrak{d}_r]:=[\mathfrak{d}^{\phi}_1\otimes n,\dots, \mathfrak{d}^{\phi}_r\otimes n]^{\tr}
$ for $N_{\phi\otimes C^{\otimes k}}$. Note that 
\begin{equation}\label{E:sigmatensor}
\sigma \cdot \mathfrak{d}=(t-\theta)^k\Phi \mathfrak{d}. 
\end{equation}
To see that $N_{\phi\otimes C^{\otimes k}}$ is rigid analytically trivial, we define the matrix $\tilde{\Psi}:=\Omega^{k}\Psi\in \GL_r(\mathbb{T})$ and observe that $\tilde{\Psi}^{(-1)}=(t-\theta)^k\Phi\tilde{\Psi}$.

We set $\tilde{\mathfrak{h}}_r:=1$ and for each $i\in \{1,\dots,r-1\}$, we let 
\[
\tilde{\mathfrak{h}}_i:= k_{i+1}^{(-1)}\mathfrak{d}^{\phi}_2 +k_{i+1}^{(-2)}\mathfrak{d}^{\phi}_3+\dots+k_r^{(-(r-i))}\mathfrak{d}^{\phi}_{r-i+1}.
\]
Moreover, we consider the $\mathbb{C}_{\infty}[\sigma]$-basis for $N_{\phi\otimes C^{\otimes k}}$ defined by
\begin{multline*}
\mathfrak{h}:=\{h_1,\dots,h_{rk+1}\}:=[(t-\theta)^k\tilde{\mathfrak{h}}_r\otimes n,(t-\theta)^{k-1}\tilde{\mathfrak{h}}_1\otimes n,\dots, (t-\theta)^{k-1}\tilde{\mathfrak{h}}_r\otimes n, \\
(t-\theta)\tilde{\mathfrak{h}}_1\otimes n,\dots, (t-\theta)\tilde{\mathfrak{h}}_r\otimes n,
\tilde{\mathfrak{h}}_1\otimes n,\dots, \tilde{\mathfrak{h}}_r\otimes n]^{\tr}
\end{multline*}
and observe that 
\[
t\cdot \mathfrak{h}=\rho_{\theta}^{*}\mathfrak{h}.
\]

\subsection{Logarithms of Anderson $t$-modules}\label{SS:Log of And t-modules} In this section we review the background and some of the main theorems of \cite{Gre22} which gives a factorization theorem for the logarithm function of a $t$-module. We state our first lemma which describes a particular choice of bases for Anderson $t$-motives and dual $t$-motives.

\begin{lemma}\cite[Lem. 2.10]{Gre22} \label{L:basis} Let $G$ be an Anderson $t$-module and $M_{G}$ ($N_G$ respectively) be the corresponding Anderson $t$-motive (dual $t$-motive respectively).
	\begin{itemize}
		\item[(i)] There exists a $\mathbb{C}_{\infty}[t]$-basis $\{c_1,\dots,c_r\}$  ($\{d_1,\dots,d_r\}$ respectively) for $M_G$ ($N_G$ respectively) such that 
		\[
		\tau [c_1,\dots,c_r]^{\tr}= \mathfrak{Q}[c_1,\dots,c_r]^{\tr}
		\]
		and 
		\[
		\sigma [d_1,\dots,d_r]^{\tr}= \mathfrak{Q}^{\tr}[d_1,\dots,d_r]^{\tr}
		\]
		for some $\mathfrak{Q}\in \GL_r(\mathbb{T})\cap \Mat_r(\mathbb{C}_{\infty}[t])$.
		\item[(ii)]  There exists a $\mathbb{C}_{\infty}[\tau]$-basis $\mathcal{G}:=[g_1,\dots,g_d]^{\tr}$ for $M_G$ and a $\mathbb{C}_{\infty}[\sigma]$-basis $\mathcal{H}:=[h_1,\dots,h_d]^{\tr}$ for $N_G$ such that 
		\[
		t\cdot \mathcal{G}= \mathfrak{V}\mathcal{G}
		\]
		and 
		\[
		t\cdot  \mathcal{H}= \mathfrak{V}^{*}\mathcal{H}
		\]
		for some $ \mathfrak{V}\in \Mat_{d\times d}(\mathbb{C}_{\infty})[\tau]$.
	\end{itemize}	
	
\end{lemma}
\begin{remark}
We comment that the formulas in this paper have no dependence on which particular $\C_\infty[t]$-bases we choose for $M_G$ and $N_G$. However, if we make a change of basis for $\mathcal G$ by $P\in \GL_d(\C_\infty[\tau])$ then we switch to the basis $\mathcal G_1 := P\mathcal G$. This amounts to changing the $\mathbb{F}_q$-algebra homomorphism $\phi$ to the $\mathbb{F}_q$-algebra homomorphism $\tilde{\phi}$ given by $\tilde{\phi}_{\theta}:=P\phi P\inv$ which gives rise to an Anderson $t$-module $(\mathbb{G}^d_{a/\mathbb{C}_{\infty}},\tilde{\phi})$ isomorphic to $G$. We then make a corresponding change of basis for $\mathcal H$ given by $\mathcal H_1 = (P^*)\inv\mathcal H$ so that the conditions of the above corollary are still satisfied. This would change our formulas such as in Theorem \ref{T:log} involving the logarithm function to get instead $P\cdot \Log_G(P\inv\bz)$.
\end{remark}

Let $N\cong \Mat_{1\times r}(\mathbb{C}_{\infty}[t])$ be a dual $t$-motive for some $r\in \mathbb{Z}_{\geq 1}$ and let $h=\{h_1,\dots,h_d\}$ be a $\mathbb{C}_{\infty}[\sigma]$-basis. Any $n\in N$ can be written as
\[
n=\sum_{i=1}^d\left(\sum_{j=0}^{m_i}\alpha_{i,j}\sigma^j\right)h_i
\]  
for some $\alpha_{i,j}\in \mathbb{C}_{\infty}$ and $m_{i}\in \mathbb{Z}_{\geq 0}$. Then we define the map $\delta_{0}^N:N\to \mathbb{C}_{\infty}^d$ by 
\[
\delta_{0}^N(n):=\begin{pmatrix}
\alpha_{0,1}\\
\vdots\\
\alpha_{0,d}
\end{pmatrix}.
\]
Now let $\{d_1,\dots,d_r\}$ be a $\mathbb{C}_{\infty}[t]$-basis for $N$ as in Lemma \ref{L:basis}(i). We consider 
\[
\tilde{N}:=\oplus_{i=1}^r \mathbb{C}_{\infty}(t)d_i\cong \Mat_{1\times r}(\mathbb{C}_{\infty}(t))
\]
and for any $\tilde{n}\in \tilde{N}$, write $\tilde{n}=\sum_{i=1}^da_i d_i$ for some $a_i\in \mathbb{C}_{\infty}(t)$. We define 
\begin{equation}\label{E:sigmainv}
\sigma^{-1}(\tilde{n}):=(\mathfrak{Q}^{-1})^{(1)}\begin{pmatrix}
a_1\\
\vdots\\
a_r
\end{pmatrix}^{(1)}.
\end{equation}
Moreover, we consider $N_{\theta}:=N\otimes_{\mathbb{C}_{\infty}[t]}\TT_{\theta}$. By \cite[Prop. 2.18]{Gre22}, there exists an extension of $\delta_{0}^{N}$-map to $\delta_{0}^{N}:N_{\theta}\to \mathbb{C}_{\infty}^d$.

\begin{remark} If $n\in N$, then one can write $n=\sum_{i=0}^ra_i c_i$ for some $a_i\in \mathbb{C}_{\infty}[t]$. Since $\det(\mathfrak{Q})=(t-\theta)^{\ell}g$ for some $\ell\in \mathbb{Z}_{\geq 1}$ and $g\in \mathbb{C}_{\infty}^{\times}$, each entry of $\sigma^{-1}(n)\in \tilde{N}$ can be written as a ratio $F(t)/G(t)$ of polynomials $F(t),G(t)\in \mathbb{C}_{\infty}[t]$ so that $G(t)=(t-\theta^q)^{\ell'}g'$ for some $\ell'\in \mathbb{Z}_{\geq 0}$ and $g'\in \mathbb{C}_{\infty}^{\times}$. Thus, one can evaluate $\sigma^{-j}(n)$ at $\delta_{0}^{N}$ for any integer $j$. We refer the reader to \cite[Prop. 2.18]{Gre22} for details on this extensions of $\delta_0^N$.
\end{remark}

\begin{definition}\label{D:delta_1 extension} We define another crucial map for our purposes. Let $M\cong \Mat_{1\times d}(\mathbb{C}_{\infty})[\tau]$ be an Anderson $t$-motive and fix $\zz=(z_1,\dots,z_d)^{\tr} \in \mathbb{C}_{\infty}^d$. We define $\delta_{1,\zz}^M:M\to \mathbb{C}_{\infty}$ by 
\begin{equation}\label{D:delta_1 map}
\delta_{1,\zz}^M(m):=m\zz:=m_1(z_1)+\dots+m_d(z_d), \ \ m=[m_1,\dots,m_d]\in \Mat_{1\times d}(\mathbb{C}_{\infty})[\tau],
\end{equation}
where we view $\tau$ as acting as the $q$-power Frobenius. We further define $M_{\zz}$ to be the set of elements $(\mathfrak{a}_1,\dots,\mathfrak{a}_d)$ where, for each $i\in \{1,\dots,d\}$,  $\mathfrak{a}_i=\sum_{j=0}^{\infty}a_{i,j}\tau^j\in \mathbb{C}_{\infty}[[\tau]]$ satisfies 
$
\left(a_{1,\mu}\tau^\mu,\dots,a_{d,\mu}\tau^\mu \right)\zz\to 0 $ as $\mu\to \infty$. Then we extend the map $\delta_{1,\zz}^M$ to $M_{\zz}$ by defining $\delta_{1,\zz}^{M}:M_{\zz}\to \mathbb{C}_{\infty}$ as
$
\delta_{1,\zz}^M(\tilde{m}):=\lim_{\mu\to \infty}\delta_{1,\zz}([\mathfrak{a}^{\mu}_1,\dots,\mathfrak{a}^{\mu}_d])
$
where
$
\tilde{m}=\left[\sum_{j=0}^{\infty}a_{1,j}\tau^j,\dots, \sum_{j=0}^{\infty}a_{d,j}\tau^j\right]
$
and 
$\mathfrak{a}_i^{\mu}:=\sum_{j=0}^{\mu}a_{i,j}\tau^j$. Finally, we extend $\delta_{1,\bz}^{M}$ to vectors in $M_\bz^d$ by acting coordinate-wise. Again, we refer the reader to \cite[Def. 2.19]{Gre22} for full details on this extension.
\end{definition}

The following was one of the main theorems of \cite{Gre22} and gives an interpretation of the logarithm function of an Anderson $t$-module in terms of a limit of evaluations of the motivic maps $\delta_{1,\zz}^M$ and $\delta_0^N$ given above. After substituting definitions, this formula becomes an infinite product of matrices (or a finite sum of such terms), hence we call it a factorization of the logarithm. Before stating it, we note that, in what follows, using a $\mathbb{C}_{\infty}[\tau]$-basis $\{g_1,\dots,g_d\}$ of $M_G$ as described in Lemma \ref{L:basis}, we identify $M_G$ via an isomorphism between $M_G$ and the space $\Mat_{1\times d}(\mathbb{C}_{\infty})[\tau]$ of $d$-dimensional row vectors that send $g_{\ell}$ to the $\ell$-th unit vector for each $1\leq \ell \leq d$. Hence for each $n\geq 0$, we emphasize that $\tau^n(g_{\ell})\in \Mat_{1\times d}(\mathbb{C}_{\infty})[\tau] $.

\begin{theorem}\cite[Thm. 4.4]{Gre22}\label{T:log} Let $M_{G}$ ($N_{G}$ respectively) be the Anderson $t$-motive (dual $t$-motive respectively) corresponding to $G$. Let $\mathcal{G}$ and $\mathcal{H}$ be the $\mathbb{C}_{\infty}[\tau]$-basis ($\mathbb{C}_{\infty}[\sigma]$-basis respectively) of $M_G$ ($N_G$ respectively) as in Lemma \ref{L:basis}(ii). Then
\[
\Log_{G}=\lim_{n\to \infty}\sum_{i=0}^{n}\sum_{\mu=1}^d\delta_{0}^{N_G}(\sigma^{-i}(h_\mu))\tau^i(g_\mu).
\]
Moreover, let $\zz$ be an element in the domain of convergence of $\Log_{G}$. Then
\[
\Log_{G}(\zz)=\delta_{1,\zz}^{M_{G}}\left(\lim_{n\to \infty}\sum_{i=0}^{n}\sum_{\mu=1}^d\delta_{0}^{N_G}(\sigma^{-i}(h_\mu))\tau^i(g_\mu)\right).
\]
%	\[
%	\Log_{G}(\zz)=\lim_{n\to \infty}\delta_{1,\zz}^{M_G}\left((t\Id_d-d[\theta])^{-1}\sum_{\mu=1}^d\sum_{\nu=0}^{\ell-1}\delta_{0}^{N_G}(\sigma^{\nu-n}(h_\mu))\tau^n(\mathcal{G}^{\tr}\Theta_{\phi,\tau^{\ell-\nu}}^{\tr}\mathfrak{e}_\mu )\right).
%	\]
\end{theorem}
\begin{proof} For completeness, we sketch the details of the proof. Let $\Log_{G}=\sum_{i=0}^{\infty} P_i\tau^i$ and  we let $P_i=(a_{j,k,i})_{j,k}\in \Mat_{d}(\mathbb{C}_{\infty})$.
We write
\begin{multline*}
\Log_G=\sum_{i=0}^{\infty}\begin{pmatrix} a_{1,1,i}&a_{1,2,i}&\cdots & a_{1,d,i}\\
a_{2,1,i}&a_{2,2,i}&\cdots & a_{2,d,i}\\
\vdots & \vdots & & \vdots\\
a_{d,1,i}&a_{d,2,i}&\cdots & a_{d,d,i}
\end{pmatrix}\tau^i=
\sum_{i=0}^{\infty}\begin{pmatrix} a_{1,1,i}\tau^i&a_{1,2,i}\tau^i&\cdots & a_{1,d,i}\tau^i\\
a_{2,1,i}\tau^i&a_{2,2,i}\tau^i&\cdots & a_{2,d,i}\tau^i\\
\vdots & \vdots & & \vdots\\
a_{d,1,i}\tau^i&a_{d,2,i}\tau^n&\cdots & a_{d,d,i}\tau^i
\end{pmatrix}\\
=\sum_{i=0}^{\infty}\begin{pmatrix}
    a_{1,1,i}\\
    \vdots\\
    a_{d,1,i}
\end{pmatrix}\tau^i[1,0,\dots,0]+\cdots+ \sum_{i= 0}^{\infty}\begin{pmatrix}
    a_{1,d,i}\\
    \vdots\\
    a_{d,d,i}
\end{pmatrix}\tau^i[0,\dots,0,1]\\
=\sum_{i=0}^{\infty}\begin{pmatrix}
    a_{1,1,i}\\
    \vdots\\
    a_{d,1,i}
\end{pmatrix}\tau^i(g_1)+\cdots+ \sum_{n\geq 0}\begin{pmatrix}
    a_{1,d,i}\\
    \vdots\\
    a_{d,d,i}
\end{pmatrix}\tau^i(g_d)\\
=\lim_{n\to \infty}\left(\sum_{i=0}^n\sum_{\mu=1}^d \delta_0^{N_G}(\sigma^{-i}(h_\mu)) \tau^i(g_{\mu})\right)
\end{multline*}
where the last equality follows from the fact that $ \delta_0^{N_G}(\sigma^{-i}(h_\mu))=(a_{1,\mu,i},\dots,a_{d,\mu,i})^{\tr} $ (\cite[Cor. 4.5]{Gre22}). On the other hand, for $\zz=[z_1,\dots,z_d]\in \mathbb{C}_{\infty}^{d}$ in the domain of convergence of $\Log_G$, using the definition of $\delta_{1,\zz}^{M_G}$, we obtain
\begin{multline*}
\delta_{1,\zz}^{M_{G}}\left(\lim_{n\to \infty}\sum_{i=0}^{n}\sum_{\mu=1}^d\delta_{0}^{N_G}(\sigma^{-i}(h_\mu))\tau^i(g_\mu)\right)\\=\sum_{i= 0}^{\infty}\begin{pmatrix} a_{1,1,i}\tau^i&a_{1,2,i}\tau^i&\cdots & a_{1,d,i}\tau^i\\
a_{2,1,i}\tau^i&a_{2,2,i}\tau^i&\cdots & a_{2,d,i}\tau^i\\
\vdots & \vdots & & \vdots\\
a_{d,1,i}\tau^i&a_{d,2,i}\tau^i&\cdots & a_{d,d,i}\tau^i
\end{pmatrix}\begin{pmatrix}
    z_1\\
    \vdots\\
    \vdots\\
    z_d
\end{pmatrix}=\Log_G(\zz)
\end{multline*}
as desired.
\end{proof}

Let $G=(\mathbb{G}_{a/\mathbb{C}_{\infty}}^d,\phi)$ be an Anderson $t$-module given as in Definition \ref{D:t-modules}. For each $j\in \{0,\dots,\ell-1\}$, we set 
\begin{equation}\label{D:Theta tau}
\Theta_{\phi,\tau^{\ell-j}}:=A_{j+1}^{(-j)}\tau+\dots + A_{\ell}^{(-j)}\tau^{\ell-j}.
\end{equation} 

Continuing with the notation of Theorem \ref{T:log}, our next proposition may be deduced from \cite[Prop. 2.15, Prop. 5.4.3]{Gre22}.
\begin{proposition} \label{P:deltazeromap} Let $\mathfrak{e}_i\in \Mat_{d\times 1}(\mathbb{F}_q)$ be the $i$-th unit vector. Then we have 	
\[
	(t\Id_d-d[\theta])\sum_{i=0}^{n}\sum_{\mu=1}^d\delta_{0}^{N_G}(\sigma^{-i}(h_\mu))\tau^i(g_\mu)=\sum_{\mu = 1}^{d}\sum_{\nu=0}^{\ell-1}\delta_0^{N_G}(\sigma^{\nu-n}(h_\mu))\tau^n\left(\mathfrak{e}_\mu^{\tr}\Theta_{\phi,\tau^{\ell-\nu}} \mathcal{G} \right).
    \]
    \end{proposition}

%We finish this subsection with the following useful lemma.
%\begin{lemma}\label{L:sum in Mz} Let $\zz$ be an element in the domain of convergence of $\Log_G$. Then we have 
%	\[
%	\lim_{n\to \infty}(t\Id_d-d[\theta])^{-1}\sum_{\mu = 1}^{d}\sum_{\nu=0}^{\ell-1}\delta_0^{N_G}(\sigma^{\nu-n}(h_\mu))  \tau^n\left(\mathfrak{e}_{\mu}^{\tr}\Theta_{\phi,\tau^{\ell-\nu}} \mathcal{G} \right)\in (M_G)_{\zz}.
%	\]
%\end{lemma}
%\begin{proof} Set 
%	\[
%	G(1,1;\zz):=\lim_{n\to \infty}\delta_{1,\zz}^{M_G}\left(\sum_{i=0}^{n}\sum_{\mu=1}^d\delta_{0}^{N_G}(\sigma^{-i}(h_\mu))\tau^i(g_\mu)\right).
%	\] 
%	Then by  \cite[Thm. 4.4]{Gre22}, we have $G(1,1;\zz)=\Log_{G}(\zz)$. This indeed implies that 
%	\begin{equation}\label{E:limiting}
%	\lim_{n\to \infty}\left(\sum_{i=0}^{n}\sum_{\mu=1}^d\delta_{0}^{N_G}(\sigma^{-i}(h_\mu))\tau^i(g_\mu)\right)\in (M_G)_{\zz}.
%	\end{equation}
%	On the other hand by \cite[Prop. 2.15, Prop. 5.4(3)]{Gre22}, we also have 
%	\[
%	(t\Id_d-d[\theta])\sum_{i=0}^{n}\sum_{\mu=1}^d\delta_{0}^{N_G}(\sigma^{-i}(h_\mu))\tau^i(g_\mu)=\sum_{\mu = 1}^{d}\sum_{\nu=0}^{\ell-1}\delta_0^{N_G}(\sigma^{\nu-n}(h_\mu))  \tau^n\left(\mathfrak{e}_\mu^{\tr}\Theta_{\phi,\tau^{\ell-\nu}} \mathcal{G} \right),
%	\]
%which finishes the proof of the lemma after taking the limit as $n\to \infty$.
%\end{proof}

\subsection{The map $\varphi$} \label{SS:varphi map}
Throughout this subsection, we fix a Drinfeld module $\phi$ given by
\begin{equation}\label{E:DMfixed}
\phi_{\theta} = \theta + k_1 \tau + \dots + k_r\tau^r
\end{equation}
so that $\inorm{k_i}\leq 1$ for each $1\leq i \leq r-1$ and $k_r\in \mathbb{F}_q^{\times}$. 

Recall that $\{\mathfrak{c}_r\}$ constitutes a $\mathbb{C}_{\infty}[\tau]$-basis for $M_{\phi}$.  Our goal in this subsection is to construct an extension of the isomorphism  $\tilde{\varphi}:M_{\phi}\cong \Mat_{1\times r}(\mathbb{C}_{\infty}[t])$ of $\mathbb{C}_{\infty}[t,\tau]$-modules given by 
\[
\tilde{\varphi}\left(\sum_{n\geq 0}a_n\tau^n(\mathfrak{c}_r)\right):=a_0[0,\dots,0,1]+\sum_{n\geq 1}a_n[0,\dots,0,1]\prod_{\ell=1}^{n}(\Phi^{\tr})^{(n-\ell)},\ \ a_n\in \mathbb{C}_{\infty}.
\]

In order to construct such an extension, we will make some analysis on the entries of $\tilde{\varphi}(\tau^n(\mathfrak{c}_r))$ to determine the domain of our extension as well as to prove that it is injective in Proposition \ref{P:inj}.

We now let $f_{r,0}:=1$, $f_{1,0}=\cdots=f_{r-1,0}:=0$ and for $n\geq 1$, define $f_{1,n},\dots,f_{r,n}\in \mathbb{C}_{\infty}[t]$ so that 
\begin{equation}\label{E:fnotation}
\tilde{\varphi}(\tau^n(\mathfrak{c}_r))=[0,\dots,0,1]\prod_{\ell=1}^{n}(\Phi^{\tr})^{(n-\ell)}=[f_{1,n},\dots,f_{r,n}]\in \Mat_{1\times r}(\mathbb{C}_{\infty}[t]).
\end{equation}

Next we prove two key lemmas. Before stating them, for any $\ell\in \mathbb{Z}_{\geq 1}$ and $0\leq m \leq r-1$, we consider the degree $\ell$-polynomial in $t$ with coefficients in $A$ given by
\[
\mathfrak{p}_{\ell,m}(t):=\prod_{\mu=0}^{\ell-1}(t-\theta^{q^{m+\mu r}})=(t-\theta^{q^m})(t-\theta^{q^{m+r}})\cdots(t-\theta^{q^{m+(\ell-1)r}})\in A[t].
\]
We also let $\mathfrak{p}_{0,m}(t):=1$.
\begin{lemma}\label{L:t degree of phi} Let $n = sr + j$ for $s\in \mathbb{Z}_{\geq 0}$ and $0\leq j\leq r-1$. Then the following statements hold.
\begin{itemize}
\item[(i)]  Let $\mathfrak{M}:=\{z\in \mathbb{C}_{\infty} \ \ | \ \ |z|\leq 1\}$. For each $1\leq i \leq r$, $f_{i,n}$ can be written as an $\mathfrak{M}$-linear combination of polynomials $\mathfrak{p}_{\tilde{s},\tilde{j}}(t)$ so that $0\leq \tilde{s} \leq s$, $0\leq \tilde{j}\leq r-1$ and $\tilde{s}r+\tilde{j}\leq n$. Moreover, we have $\deg_t(f_{r-j,n}) = s$ and 
\begin{equation}\label{E:strucoff}
f_{r-j,n}=a\mathfrak{p}_{s,j}(t)  +\sum_{\substack{0\leq \tilde{s}<s\\0\leq \tilde{j}\leq r-1}}\beta_{\tilde{s},\tilde{j}}\mathfrak{p}_{\tilde{s},\tilde{j}}(t)
\end{equation}
for some  $a\in \F_q^\times$ and $\beta_{\tilde{s},\tilde{j}}\in \mathfrak{M}$ for each $\tilde{s}$ and $\tilde{j}$.
\item[(ii)] $\deg_t(f_{r-j,n}) \geq \deg_t(f_{r-i,n})$ for $i<j$.
\item[(iii)] $\deg_t(f_{r-j,n}) > \deg_t(f_{r-i,n})$ for $i>j$.
\item[(iv)] $\tilde{\varphi}(\tau^{j}(\mathfrak{c}_r))=[0,\dots,1,*,\dots,*]$ where the coordinates having $*$ consist of elements in $\mathfrak{M}$ and $1$ occurs in the $(r-j)$-th coordinate.
%\item[(iv)] $\deg_\theta(f_{r-j,n}) = q^j + q^{r+j} + \dots + q^{(s-1)r+j}$
%\item[(v)] $\deg_\theta(f_{r-j,n}) \geq \deg_\theta(f_{r-i,n})$ for $i<j$
%\item[(vi)] $\deg_\theta(f_{r-j,n}) > \deg_\theta(f_{r-i,n})$ for $i>j$
 
\end{itemize}
\end{lemma}

\begin{proof} Since $|k_i|\leq 1$ for $i=1,\dots,r-1$ and $k_r\in \mathbb{F}_q^{\times}$, (iv) immediately follows from a direct computation. We prove the remaining parts. First, by part (iv),
we obtain  
\[
\tilde{\varphi}(\tau^{r-1}(\mathfrak{c}_r))=[1,b_2,\dots,b_r],
\]
where $b_i\in \mathfrak{M}$. Thus
\[
\tilde{\varphi}(\tau^{r}(\mathfrak{c}_r))=[c_1,c_2,\dots,(t-\theta)/k_r + c_r],
\]
again for constants $c_i\in \C_\infty$ with $|c_i|\leq 1$, so we see directly that the lemma is true for $r$ (with $s=1$ and $j=0$). By direct computation involving \eqref{E:Phi def Drinfeld modules}, we see that
\begin{equation}\label{E:multip}
\tilde{\varphi}(\tau^{n+1}(\mathfrak{c}_r)) = [f_{1,n},\dots,f_{r,n}]\twist\begin{pmatrix}
&& & &  \frac{t-\theta}{k_r}\\
1& &  & & -\frac{k_1^{(-1)}}{k_r}\\
&\ddots &  &  & \vdots\\
& &\ddots &  &  \\
& &  & 1 & -\frac{k_{r-1}^{(-(r-1))}}{k_r}
\end{pmatrix} = [f_{1,n+1},\dots,f_{r,n+1}].
\end{equation}
Observe by \eqref{E:multip} that 
\begin{equation}\label{E:multip2}
f_{r-\ell-1,n+1}=f_{r-\ell,n}^{(1)} \ \ \textit{ for } \ \  0\leq \ell \leq r-2
\end{equation}
and 
\begin{equation}\label{E:multip222}
f_{r,n+1}=\frac{1}{k_r}\left(f_{1,n}^{(1)}(t-\theta)-f^{(1)}_{2,n}k_1-\cdots-f_{r-1,n}^{(1)}k_{r-2}-f_r^{(1)}k_{r-1}\right).
\end{equation}
Again, a direct computation, by using \eqref{E:multip}, implies that the lemma holds for $n=r+j$ for $1\leq j \leq r-1$. Then, we assume by induction that the lemma is true for $n$. We further note that using (iv) and assuming \eqref{E:strucoff}, the first assertion of (i) follows from \eqref{E:multip2} and \eqref{E:multip222}. Hence, we divide our argument into three cases to show (ii), (iii) and \eqref{E:strucoff} to finish the proof of the lemma.

\textbf{Case 1:} If $j=0$, then $n+1=sr+1$. By the induction hypothesis, $\deg_t(f_{r,n}) > \deg_t(f_{r-i,n})$ for $i>0$. Thus, since, $\deg_t(f_{r,n+1})\leq \deg_t(f_{1,n})+1\leq \deg_t(f_{r,n})=\deg_t(f_{r-1,n+1})$, by \eqref{E:multip2}, we see that (ii) and (iii) also hold for $n+1$. For (i), we simply obtain by the induction hypothesis 
\[
f_{r-1,n+1}=f_{r,n}^{(1)}=a(t-\theta^q)(t-\theta^{q^{r+1}})\cdots (t-\theta^{q^{(s-1)r+1}})+\sum_{\substack{0\leq \tilde{s}<s\\0\leq \tilde{j}\leq r-1}}\beta_{\tilde{s},\tilde{j}}^q\mathfrak{p}_{\tilde{s},\tilde{j}}(t)^{(1)}
\]
implying that \eqref{E:multip2} holds for $n+1$ where  $a\in \F_q^\times$ and $\beta_{\tilde{s},\tilde{j}}\in \mathfrak{M}$ for each $\tilde{s}$ and $\tilde{j}$.

\textbf{Case 2:} If $0<j<r-1$, then $n+1=sr+j+1$.  By the induction hypothesis, $\deg_t(f_{r-j,n}) \geq \deg_t(f_{r-i,n})$ for $i<j$ and $\deg_t(f_{r-j,n}) > \deg_t(f_{r-i,n})$ for $i>j$. Thus, combining with \eqref{E:multip2}, we have $\deg_t(f_{r-(j+1),n+1})=\deg_t(f_{r-j,n}) \geq \deg_t(f_{r-i,n})=\deg_t(f_{r-(i+1),n+1})$ for $i<j$ and  $\deg_t(f_{r-(j+1),n+1})=\deg_t(f_{r-j,n}) > \deg_t(f_{r-i,n})=\deg_t(f_{r-(i+1),n+1})$ for $i>j$. Hence (ii) and (iii) also hold for $n+1$. For \eqref{E:multip2}, we again obtain by the induction hypothesis 
\[
f_{r-(j+1),n+1}=f_{r-j,n}^{(1)}=a(t-\theta^{q^{j+1}})(t-\theta^{q^{r+j+1}})\cdots (t-\theta^{q^{(s-1)r+j+1}})+\sum_{\substack{0\leq \tilde{s}<s\\0\leq \tilde{j}\leq r-1}}\beta_{\tilde{s},\tilde{j}}^q\mathfrak{p}_{\tilde{s},\tilde{j}}(t)^{(1)}
\]
implying that \eqref{E:strucoff} holds for $n+1$ where  $a\in \F_q^\times$ and $\beta_{\tilde{s},\tilde{j}}\in \mathfrak{M}$ for each $\tilde{s}$ and $\tilde{j}$.

\textbf{Case 3:} If $j=r-1$, then $n+1=sr+r=(s+1)r$. By the induction hypothesis, $\deg_t(f_{1,n}) \geq \deg_t(f_{r-m,n})$ for $0\leq m\leq r-1$. Thus,  by \eqref{E:multip222} and the fact that $\deg_t(f_{r-i,n+1})=\deg_t(f_{r-i+1,n})<\deg_t(f_{1,n+1})+1=\deg_t(f_{r,n+1})$ for $0<i<r$, we see that (iii) also holds for $n+1$ ((ii) is an empty statement in this case). For \eqref{E:multip2}, we obtain by the induction hypothesis that 
\[
f_{r,n+1}=
a(t-\theta)(t-\theta^{q^{r}})\cdots (t-\theta^{q^{(s-1)r+r}})+\sum_{\substack{0\leq \tilde{s}<s+1\\0\leq \tilde{j}\leq r-1}}\beta_{\tilde{s},\tilde{j}}^q\mathfrak{p}_{\tilde{s},\tilde{j}}(t)^{(1)}
\]
implying that \eqref{E:strucoff} holds for $n+1$ where  $a\in \F_q^\times$ and $\beta_{\tilde{s},\tilde{j}}\in \mathfrak{M}$ for each $\tilde{s}$ and $\tilde{j}$. Hence it finishes the proof of part (i), (ii) and (iii). 
\end{proof}

Next, using the notation in \eqref{E:fnotation}, for $1\leq \mu \leq r$ and  $n,\nu\geq 0$, we define $b_{\mu,n,\nu}\in \mathbb{C}_{\infty}$ given by the equality 
\begin{equation}\label{E:coefb}
f_{\mu,n}=\sum_{\nu\geq 0}b_{\mu,n,\nu}t^{\nu}\in \mathbb{C}_{\infty}[t].
\end{equation}

\begin{lemma}\label{L:maxnorm} Let $n=sr+j$  for  $s\in \mathbb{Z}_{\geq 0}$ and $0\leq j\leq r-1$. Let $0\leq \nu\leq s$ and set $\alpha_{n,\nu}:=\max\{|b_{1,n,\nu}|,\dots,|b_{r,n,\nu}|\}$, that is, the maximum among the $|\cdot|$-norms of the $t^{\nu}$-coefficients of the entries of $\tilde{\varphi}(\tau^{n}(\mathfrak{c}_r))\in \Mat_{1\times r}(\mathbb{C}_{\infty}[t])$. Then we have
\[
\log_q(\alpha_{n,\nu})\leq \frac{q^{n}-q^{\nu r+j}}{q^r-1}.
\]
\end{lemma}
\begin{proof} Let
 $c_{\ell,m,i}$ denote the $t^i$-coefficient of $\mathfrak{p}_{\ell,m}(t)$. By Vieta's formulas, we have 
\[
c_{\ell,m,i}=\begin{cases}
1& \text{ if } i=\ell\\
(-1)^{\ell-i}\displaystyle\sum_{0\leq \mu_1<\mu_2<\cdots<\mu_{\ell-i}\leq \ell}\left(\prod_{\nu=1}^{\ell-i}\theta^{q^{m+\mu_\nu r}}\right) & \text{ if } i<\ell.
\end{cases}
\]
Since $|\cdot|$ is a nonarchimedean norm, letting $\mu_\nu=i+\nu-1$ for $1\leq \nu \leq \ell-i$ above, we see that the norm of $c_{\ell,m,i}$ is bounded by the norm of $\theta^{q^{m+ir}}\cdots \theta^{q^{m+(\ell-1) r}}$. In particular, we have
\[
\log_q(|c_{\ell,m,i}|)\leq q^{m}\left(q^{ir}+\cdots+q^{(\ell-1)r} \right).
\]

 We now compare the $t^{\nu}$-coefficient of coordinates of $\tilde{\varphi}(\tau^{n}(\mathfrak{c}_r))$. Since, by Lemma \ref{L:t degree of phi}(i), for any $1\leq \mu \leq r$, $f_{\mu,n}$ can be written as an $\mathfrak{M}$-linear combination of polynomials $\mathfrak{p}_{\tilde{s},\tilde{j}}(t)$ so that $0\leq \tilde{s} \leq s$,  $0\leq \tilde{j}\leq r-1$ and $\tilde{s}r+\tilde{j}\leq n=sr+j$, it suffices to analyze the norm of the coefficients $c_{\tilde{s},\tilde{j},\nu}$ to prove the lemma. Thus we have
\begin{multline*}
  \log_q(|c_{\tilde{s},\tilde{j},\nu}|)\leq  q^{\tilde{j}}\left(q^{\nu r}+\cdots+q^{(\tilde{s}-1)r} \right) \leq q^{j}\left(q^{\nu r}+\cdots+q^{(s-1)r} \right)=q^{j}\left(\frac{q^{sr}-1}{q^r-1}- \frac{q^{\nu r}-1}{q^r-1}\right)\\= \frac{q^{n}-q^{\nu r+j}}{q^r-1}
    \end{multline*}
as desired.
\end{proof}

Before we introduce an extension of the map $\tilde{\varphi}$, we state our next lemma.

\begin{lemma}\label{L:boundtau} For each $n\geq 0$, we have 
\[
\log_q\left(||\tilde{\varphi}(\tau^n(\mathfrak{c}_r))||\right)\leq \frac{q^{n+r-1}}{q^r-1}-\frac{q^{r-1}}{q^r-1}.
\]
\end{lemma}
\begin{proof} Let again $n=sr+j$ for some $s\in \mathbb{Z}_{\geq 0}$ and $0\leq j\leq r-1$. By Lemma \ref{L:t degree of phi}, we see that 
\begin{equation}\label{E:estimate0}
\log_q\left(||\tilde{\varphi}(\tau^n(\mathfrak{c}_r))||\right)=\log_q(||f_{r-j,n}||)=\log_q(||\mathfrak{p}_{s,j}(t)||)=q^j(1+q^r+\cdots+q^{(s-1)r})=q^j\left(\frac{q^{sr}-1}{q^r-1}\right).
\end{equation}
Finally noting that $0\leq j \leq r-1$ and using \eqref{E:estimate0}, we obtain 
\[
\log_q\left(||\tilde{\varphi}(\tau^n(\mathfrak{c}_r))||\right)= q^j\left(\frac{q^{sr}-1}{q^r-1}\right)\leq q^{r-1}\left(\frac{q^{sr}-1}{q^r-1}\right)\leq q^{r-1}\left(\frac{q^{n}-1}{q^r-1}\right)
\]
as desired.
\end{proof}

Identifying $M_{\phi}$ with $\mathbb{C}_{\infty}[\tau]$ by sending each $k_r^{-1}\tau^n(\mathfrak{c}_r)$ to $\tau^n$ for $n\geq 0$, by a slight abuse of notation, we now denote the aforementioned isomorphism of the $\mathbb{C}_{\infty}[t,\tau]$-modules by the map $\tilde{\varphi}:M_{\phi}\to\Mat_{1\times r}(\mathbb{C}_{\infty}[t])$ given by
\[
\tilde{\varphi}\left(\sum_{n\geq 0}a_n\tau^n\right):=\left[\sum_{n\geq 0}k_r^{-1}a_nf_{1,n},\dots, \sum_{n\geq 0}k_r^{-1}a_nf_{r,n}\right],\ \ a_n\in \mathbb{C}_{\infty}.
\]

Our next goal is to construct the domain of the extension of the map $\tilde{\varphi}$. To ease the notation in what follows, let us set $\mathfrak{v}:=q^{\frac{q^{r-1}}{q^r-1}}$ and define 
\[
\mathbb{M}:=\left\{\sum_{n=0}^{\infty}a_n\tau^n\ \ | \ \ a_n\in \mathbb{C}_{\infty}, \ \ |a_n|\mathfrak{v}^{q^n}\to 0 \text{ as } n\to \infty\right\}.
\]
We further set
\[
\left|\sum_{n=0}^{\infty}a_n\tau^n\right|_{\mathfrak{v}}:=\max_{n} |a_n|\mathfrak{v}^{q^n}, \ \ \ \ \  \sum_{n=0}^{\infty}a_n\tau^n\in \mathbb{M}.
\]
It is clear that $(\mathbb{M},|\cdot|_\mathfrak{v})$ forms a normed $\mathbb{C}_{\infty}$-vector space and $(M_{\phi},|\cdot|_{\mathfrak{v}})$ is a dense normed $\mathbb{C}_{\infty}$-vector subspace of $(\mathbb{M},|\cdot|_\mathfrak{v})$. Moreover, by Remark \ref{R:roclog}, we see that, for each $r\geq 1$, the logarithm function $\log_{\phi}$ converges at each point of the disk of radius $q^{q^{r-1}/(q^r-1)}<q^{q^r/(q^r-1)}$ centered at $0$  and hence $\log_{\phi}\in \mathbb{M}$.

Let $\mathcal{G}=\sum_{n=0}^ma_n\tau^n\in M_{\phi}\subset \mathbb{M}$ for some $m\in \mathbb{Z}_{\geq 0}$. By the ultrametric property of $||\cdot||$ on $\Mat_{1\times r}(\mathbb{C}_{\infty}[t])$ and Lemma \ref{L:boundtau}, it can be seen that 
\[
 ||\tilde{\varphi}(\mathcal{G})||=||\sum_{n=0}^ma_n\tilde{\varphi}(\tau^n(\mathfrak{c}_r))||\leq \max_{0\leq n \leq m} |a_n|||\tilde{\varphi}(\tau^n(\mathfrak{c}_r)|| \leq \max_{0\leq n \leq m} |a_n|\mathfrak{v}^{q^n}=|\mathcal{G}|_{\mathfrak{v}}.
\]
Hence, $\tilde{\varphi}$ is a continuous and bounded $\mathbb{C}_{\infty}$-linear map. Since  $(\Mat_{1\times r}(\mathbb{T}),||\cdot||)$ is a Banach space over $\mathbb{C}_{\infty}$ and $M_{\phi}$ is dense in $\mathbb{M}$, there exists a unique bounded extension $\varphi:\mathbb{M}\to \Mat_{1\times r}(\mathbb{T})$ of $\tilde{\varphi}$ defined by 
\[
\varphi\left(\lim_{n\to \infty}\mathcal{G}_n\right):=\lim_{n\to \infty}\tilde{\varphi}(\mathcal{G}_n)
\]
provided that $\lim_{n\to \infty}\mathcal{G}_n $ exists and lies in $\mathbb{M}$ (see \cite[Thm. 5.19]{HN01}).

Our final goal is to show that the extension map $\varphi$ is injective.

\begin{proposition}\label{P:inj}
Let
\[f = \varphi\left(\sum_{n=0}^{\infty}a_n\tau^n\right) , \ \ \ \sum_{n=0}^{\infty}a_n\tau^n\in \mathbb{M}. \]
 Then $f=0$ if and only if each $a_n=0$. In particular, $\varphi$ is injective.
\end{proposition}

\begin{proof}  Before beginning the proof, we comment by way of aiding the reader's understanding that the proof presented here is a more complicated version of the proof of a similar result from \cite[(2.4.3)]{CGM18}. Since one direction is obvious, we prove the other direction. Moreover, since $k_r\in \mathbb{F}_q^{\times}$, without loss of generality, we assume that $k_r=1$ and hence we simply identify each $\tau^n(\mathfrak{c}_r)$ with $\tau^n$ for $n\geq 0$. 
Using the coefficients $b_{\mu,n,\nu}\in \mathbb{C}_{\infty}$ defined in \eqref{E:coefb}, we have
\[
\varphi\left(\sum_{n=0}^{\infty}a_n\tau^n\right)=\left[\sum_{n=0}^{\infty}a_nf_{1,n},\dots,\sum_{n =0}^{\infty}a_nf_{r,n} \right]=\left[\sum_{\nu=0}^{\infty}\left(\sum_{n=0}^{\infty}a_nb_{1,n,\nu}\right)t^{\nu},\dots, \sum_{\nu= 0}^{\infty}\left(\sum_{n=0}^{\infty}a_nb_{r,n,\nu}\right)t^{\nu} \right].
\]
Then we write
\begin{equation}\label{E:Tate module infinite sums0}
f=\begin{pmatrix}
(\sum_{n=0}^{\infty} a_nb_{1,n,0}) + (\sum_{n=0}^{\infty} a_nb_{1,n,1})t + (\sum_{n=0}^{\infty} a_nb_{1,n,2})t^2+\dots\\
\vdots\\
(\sum_{n=0}^{\infty} a_nb_{r,n,0}) + (\sum_{n=0}^{\infty} a_nb_{r,n,1})t + (\sum_{n=0}^{\infty} a_nb_{r,n,2})t^2+ \dots\\
\end{pmatrix}^{\tr}\in \mathbb{T}^r.
\end{equation}
Now let $f=0$. Then, in each coordinate of \eqref{E:Tate module infinite sums0}, the coefficient of each power of $t$ is identically 0. This gives a sequence of infinite series so that
$\sum_{n=0}^{\infty} a_nb_{\mu,n,\nu} = 0$ for each $1\leq \mu\leq r$ and $\nu\geq 0$. 

Assume to the contrary that there exists a non-negative integer $n_0$ such that $a_{n_0}\neq 0$. We then write $n_0 = s_0r+j_0$ with $s_0\in \mathbb{Z}_{\geq 0}$ and $0\leq j_0\leq r-1$. Note, by Lemma \ref{L:t degree of phi}(i) that, as a polynomial in $t$, the leading term of the $(r-j_0)$-th coordinate $f_{r-j_0,n_0}$ of $\tilde{\varphi}(\tau^{n_0}( \mathfrak{c}_r))$ is $at^{s_0}$ for some constant $a\in \mathbb{F}_q^{\times}$. That is, $b_{r-j_0,n_0,s_0}=a$. Moreover, again by Lemma \ref{L:t degree of phi}, the coefficient of $t^{s_0}$ in the coordinates of $\tilde{\varphi}(\tau^{\ell}(\mathfrak{c}_r))$ is zero for $\ell<n_0$. Thus, we have 
\begin{equation}\label{E:series1}
\sum_{n=0}^{\infty} a_nb_{(r-j_0),n,s_0} = a_{n_0}a+\sum_{n\neq n_0} a_nb_{(r-j_0),n,s_0}=a_{n_0}a+\sum_{n> n_0} a_nb_{(r-j_0),n,s_0}=0
\end{equation}
where, we note that, the left hand side of \eqref{E:series1} is the $t^{s_0}$-coefficient of the $(r-j_0)$-th coordinate of $f$. 
Since $a_{n_0}\neq 0$ and the norm $| \cdot |$ is nonarchimedian, there must exist $n_1>n_0$  such that
\begin{equation}\label{E:boundan}
|a_{n_0}a|=|a_{n_0}|  \leq  |a_{n_1} b_{(r-j_0),n_1,s_0}|.
\end{equation}
Now let us write $n_1=s_1r+j_1$ with $s_1\in \mathbb{Z}_{\geq 0}$ and $0\leq j_1\leq r-1$. Then by Lemma \ref{L:maxnorm}, we obtain
\begin{equation}\label{E:est10}
 \log_q(|b_{(r-j_0),n_1,s_0}|)\leq \log_q(\alpha_{n_1,s_0})\leq \frac{q^{n_1}-q^{s_0 r+j_1}}{q^r-1}\leq  \frac{q^{n_1+r-1}}{q^r-1}-\frac{q^{n_0+r-1}}{q^r-1}.
\end{equation}
Here the last inequality follows from the fact that $n_1>n_0$ and $0\leq j_0,j_1\leq r-1$.
 Thus, \eqref{E:boundan} and \eqref{E:est10} yield
\[
|a_{n_0}||\theta|^{\frac{q^{n_0+r-1}}{q^r-1}} \leq |a_{n_1}||\theta|^{\frac{q^{n_1+r-1}}{q^r-1}}.
\]

We now apply our algorithm once again. More precisely, we first note, by Lemma \ref{L:t degree of phi}(i) that, as a polynomial in $t$, the leading term of the $(r-j_1)$-st coordinate $f_{r-j_1,n_1}$ of $\tilde{\varphi}(\tau^{n_1}( \mathfrak{c}_r))$ is $\tilde{a}t^{s_1}$ for some constant $\tilde{a}\in \mathbb{F}_q^{\times}$. That is, $b_{r-j_1,n_1,s_1}=\tilde{a}$. Moreover, again by Lemma \ref{L:t degree of phi}, the coefficient of $t^{s_1}$ in the coordinates of $\tilde{\varphi}(\tau^{\ell}(\mathfrak{c}_r))$ is zero for $\ell<n_1$. Thus, we have  
\begin{equation}\label{E:series2}
\sum_{n=0}^{\infty} a_nb_{(r-j_1),n,s_1} = a_{n_1}\tilde{a}+\sum_{n\neq n_1} a_nb_{(r-j_1),n,s_1}=a_{n_1}\tilde{a}+\sum_{n> n_1} a_nb_{(r-j_1),n,s_1}=0
\end{equation}
where, we note that, the left hand side of \eqref{E:series2} is the $t^{s_1}$-coefficient of the $(r-j_1)$-st coordinate of $f$. Since $a_{n_1}\neq 0$ and the norm $| \cdot |$ is nonarchimedian, there must be an $n_2>n_1$ such that
\begin{equation}\label{E:boundan1}
|a_{n_1}\tilde{a}|=|a_{n_1}|  \leq  |a_{n_2} b_{(r-j_1),n_2,s_1}|.
\end{equation}
 Now let us write $n_2=s_2r+j_2$ with $s_2\in \mathbb{Z}_{\geq 0}$ and $0\leq j_2\leq r-1$.   Then by Lemma \ref{L:maxnorm}, we obtain
\begin{equation}\label{E:est2}
 \log_q(|b_{(r-j_1),n_2,s_1}|)\leq \log_q(\alpha_{n_2,s_1})\leq \frac{q^{n_2}-q^{s_1 r+j_2}}{q^r-1}\leq  \frac{q^{n_2+r-1}}{q^r-1}-\frac{q^{n_1+r-1}}{q^r-1}.
\end{equation}
Here, again, the last inequality follows from the fact that $n_2>n_1$ and $0\leq j_1,j_2\leq r-1$. Thus, \eqref{E:boundan1} and \eqref{E:est2} yield
\[
|a_{n_1}||\theta|^{\frac{q^{n_1+r-1}}{q^r-1}} \leq |a_{n_2}||\theta|^{\frac{q^{n_2+r-1}}{q^r-1}}.
\]
Continuing in this manner, we obtain a chain of integers  $n_0<n_1<n_2<\cdots<n_w<\cdots $ and an increasing sequence $\{|a_{n_w} \theta^{\frac{q^{n_w+r-1}}{q^r-1}}|\}_{w\geq 0}$. On the other hand, since  $\sum_{n=0}^{\infty}a_n\tau^n\in \mathbb{M}$, we have
\[
\left |a_{n_w} \theta^{\frac{q^{n_w+r-1}}{q^r-1}}\right |= |a_{n_w}|(|\theta|^{\frac{q^{r-1}}{q^r-1}})^{q^{n_w}}= |a_{n_w}|\mathfrak{v}^{q^{n_w}} \to 0\]
as $w\to \infty$. However, this contradicts to the fact that $\{|a_{n_w} \theta^{\frac{q^{n_w+r-1}}{q^r-1}}|\}_{w\geq 0}$ is an increasing sequence. Hence $a_{n_0}$ must be equal to zero, finishing the proof of the proposition.
\end{proof}

As a follow up to the above discussion, here we establish that the image of $\varphi$ is contained in the subspace of entire functions.

\begin{proposition}\label{P:Image varphi entire} Let $\phi$ be a Drinfeld module as in \eqref{E:DMfixed} and let $\mathbb{E}$ be the space of entire functions of $t$, i.e. the set of all $F=\sum_{i\geq 0}a_it^i\in \mathbb{C}_{\infty}[[t]]$ so that $F$ converges for any value of $t\in \mathbb{C}_{\infty}$. 
Then each entry of an element in the image of $\varphi$ can be analytically continued to an element in $\mathbb{E}$. In particular, for any $c\in \mathbb{C}_{\infty}^{\times}$ 
\[
\varphi(\mathbb{M})\subset \Mat_{1\times r}(\mathbb{E})\subset \Mat_{1\times r}(\mathbb{T}_c).
\]
\end{proposition}

\begin{proof}
Let $g=\sum_{n\geq 0}a_n\tau^n\in \mathbb{M}$. We recall the elements $b_{\mu,n,\nu}\in \mathbb{C}_{\infty}$ from \eqref{E:coefb}. We will prove that for any $0\leq m\leq r-1$, the $(r-m)$-th entry of $\varphi(g)$
\[
\sum_{s=0}^{\infty}\left(\sum_{n=0}^{\infty} a_nb_{(r-j_1),n,s}\right)t^s\in \mathbb{T}
\]
which can be seen as a function of $t$, converges for any $c\in \mathbb{C}_{\infty}$. By the assumption on elements in $\mathbb{M}$, we have 
\[
\lim_{n\to \infty}\left(\log_q(|a_n|)+\frac{q^{n+r-1}}{q^r-1}\right)=-\infty.
\]
Therefore, there must exist $S$ such that for all $n>S$
\[
\log_q(|a_n|)+\frac{q^{n+r-1}}{q^r-1}<0.
\]

For $s>S$ and some $0\leq j_0\leq r-1$, let us investigate the $t^{s}$-th coefficient of the $(r-j_0)$-th
coordinate of $\varphi(g)$ which is given by
\[
\sum_{n=0}^{\infty} a_nb_{(r-j_0),n,s} = a_{n_0}a+\sum_{n\neq n_0} a_nb_{(r-j_0),n,s}=a_{n_0}a+\sum_{n> n_0} a_nb_{(r-j_0),n,s}
\]
where the last equality follows from the discussion just before \eqref{E:series1}. 

Note that $n_0=sr+j_0>S$. Then, by Lemma \ref{L:maxnorm} and the assumption on $a_n$'s for $n=\tilde{s}r+j>S$ with $0\leq j \leq r-1$ and $\tilde{s}\in \mathbb{Z}_{\geq 0}$, we see that 
\begin{multline}\label{E:cal2}
\log_{q}|a_nb_{r-j_0,n,s}|\leq \log_q|a_n|+\log_q |b_{r-j_0,n,s}|\leq \log_q|a_n|+\frac{q^{n}}{q^r-1}-\frac{q^{sr+j}}{q^r-1}\\\leq \log_q|a_n|+\frac{q^{n+r-1}}{q^r-1}-\frac{q^{sr+j}}{q^r-1}
< -\frac{q^{sr+j}}{q^r-1}\leq -\frac{q^{sr}}{q^r-1}.
\end{multline}
In other words,  $\log_q(|\sum_{n=0}^{\infty} a_nb_{(r-j_0),n,s}|)$ is bounded by 
\[
\mathfrak{b}_{s}=-\frac{q^{sr}}{q^r-1}.
\]
Then we see that for any $c\in \mathbb{C}_{\infty}$ and $s>S$, 
\[
\log_q\left(\left|\left(\sum_{n=0}^{\infty} a_nb_{(r-j_1),n,s}\right)c^s\right|\right)\leq -\frac{q^{sr}}{q^r-1}+s\log_q(|c|) 
\]
and hence
\[
\log_q\left(\left|\left(\sum_{n=0}^{\infty} a_nb_{(r-j_1),n,s}\right)c^s\right|\right) \to -\infty
\]
as $s\to \infty$. This shows our claim when $m=j_0$. For the case, $m\neq j_0$, we apply the same calculation as in \eqref{E:cal2} due to the fact that 
\[
\log_{q}|b_{r-m,n,s}|\leq \log_q |b_{r-j_0,n,s}|\leq \frac{q^{n}}{q^r-1}-\frac{q^{sr+j}}{q^r-1}
\]
which follows from Lemma \ref{L:t degree of phi}(i). This finishes the proof.

\end{proof}

\begin{example}\label{Ex:11}
Since the map $\delta_{1,\bz}^M$ is central to the main formulas of our paper, we give a short example showing how one computes the image of this map, at least in the case of Carlitz module (compare this calculation with the framework in \cite[\S6.5]{GP21}). Furthermore, to ease the notation, we let $\mathbb{S}_0:=t-\theta\in \mathbb{C}_{\infty}[t]$ and for $n\geq 1$, let  $\mathbb{S}_n:=(t-\theta^{q^n})\cdots (t-\theta)\in \mathbb{C}_{\infty}[t]$. By convention, we also set $\mathbb{S}_{n}=1$ if $n< 0$.  Letting $\zz\in \mathcal{D}_{C}$ and noting that the radius of convergence of $\log_C$ is $q^{q/(q-1)}$, we see that $\log_{C}\in \mathbb{M}\cap (M_C)_{\zz}$. On the other hand, we have
\begin{align*}
    \varphi(\log_C)=\left(\sum_{n=0}^{\infty}\frac{\tau^n}{L_n}\right)=\lim_{n\to \infty}\tilde{\varphi}\left(\sum_{i=0}^{n}\frac{\tau^i}{L_i}\right)=\lim_{n\to \infty}\left(1+\frac{\mathbb{S}_{0}}{L_1}+\cdots+\frac{\mathbb{S}_{n-1}}{L_n}\right)=-\tilde{\pi}\Omega\in \mathbb{T}
    \end{align*}
where the last identity follows from \eqref{E:piomegaidentity}. Since, by Proposition \ref{P:inj}, the map $\varphi$ is injective, we have $\varphi^{-1}(-\tilde{\pi}\Omega)=\log_{C}$. By Theorem \ref{T:log}, we then have
\begin{align*}
\mathcal{M}_{\zz}(-\tpi \Omega) &= z + \frac{1}{L_1}z^q + \frac{1}{L_2} z^{q^2} + \dots\\
&= \log_C(\zz),
\end{align*}
which is consistent with \cite[Cor. 5.7]{Gre22}.
\end{example}

Our goal from now on is to establish an analogous result as in Example \ref{Ex:11} for the Drinfeld module $\phi$ given in \eqref{E:DMfixed}. In order to do this, we need one more technicality to be discussed in the next subsection, which can be applied to arbitrary Anderson $t$-modules.

\subsection{Tensor construction}
Let $G=(\mathbb{G}_{a/\mathbb{C}_{\infty}}^d,\phi)$ be an Anderson $t$-module given as in Definition \ref{D:t-modules}. In this subsection, we detail a modified construction of the pairing $G(x,y)$ found in \cite{Gre22}. In the present paper, it will allow us to more easily analyze the convergence of the quantities described in \S \ref{S:Log of Drinfeld Modules}. 

For the rest of the present section, we emphasize that all the tensor products are over $\mathbb{C}_{\infty}$ unless explicitly noted otherwise and hence, to ease the notation, we avoid the subscript in our tensor product notation.

Recall the bases $\mathcal{G}$ and $\mathcal{H}$ given in Lemma \ref{L:basis}(ii).  For $x\in \C_\infty[t,\sigma]$ and $y\in \C_\infty[t,\tau]$, define
\begin{equation}\label{E:Gn tensor def}
G_n^\otimes(x,y) :=  \sum_{i=0}^n \sum_{k = 1}^{d} \sigma^{-i}(x(h_k))\otimes \tau^i(y(g_k)))\in \tilde{N}_{G}\otimes M_{G},
\end{equation}
where we regard $\tilde{N}_{G}\otimes M_{G}$ to be a $\mathbb{C}_{\infty}[t]\otimes \mathbb{C}_{\infty}[t]$-module. We note that if we apply the map $\delta_0^N$ to the first coordinate of each simple tensor in \eqref{E:Gn tensor def}, then the resulting sum is in $\C_\infty^d \otimes M \isom M^d$ and we recover $G_n(x,y)$ of \cite[Def. 5.1]{Gre22}. In fact, since $\delta_0^N$ is $\C_\infty$-linear, this is equivalent to applying $\delta_0^N\otimes 1$ to the whole sum \eqref{E:Gn tensor def}.

The pairing $G_n^\otimes(x,y)$ has many similar properties to $G_n(x,y)$ (detailed in \cite[Prop. 5.4]{Gre22}). We briefly discuss the properties of  $G_n^\otimes(x,y)$ here. For convenience we recall Definition \cite[5.3]{Gre22}:

\begin{definition}\label{D:Theta sigma}
For each $j\in \{0,\dots, \ell-1\}$, we  define
\[
\Theta_{\phi,\sigma^{\ell-j}} :=A_{j+1}\tau+\cdots+A_{\ell}\tau^{\ell-j}\in \Mat_{d}(\mathbb{C}_{\infty}[\tau]),
\]
where the $A_i$ are the coefficients of the $t$-module as in Definition \ref{D:t-modules}
\end{definition}

\begin{proposition}\label{P:Gntensor linearity properties}
Let $x\in \C_\infty[t,\sigma]$ and $y\in \C_\infty[t,\tau]$.
\begin{enumerate}
\item For any $c\in \C_\infty$ we have
\[G_n^\otimes(cx,y) = G_n^\otimes(x,cy).\]
\item We have
\[G_n^\otimes(x,\tau y) - G_n^\otimes(\sigma x,y) = \sum_{k = 1}^{d}\sigma^{-n}(xh_k)\otimes \tau^{n+1}(yg_k) - \sigma(xh_k)\otimes yg_k,\]
and more generally for $m<n$ and $c\in \C_\infty$
\begin{multline*}
G_n^\otimes(x,c\tau^m y) - G_n^\otimes(c\twistk{-m}\sigma^m x,y) = \sum_{k = 1}^{d}\sum_{u=0}^{m-1}\sigma^{u-n}(xh_k)\otimes \tau^{n}(c\twistk{-u}\tau^{m-u} yg_k)\\ - c\twistk{u-m}\sigma^{u-m}(xh_k)\otimes \tau^u(yg_k).
\end{multline*}
\item Let $\mathfrak{e}_k\in \Mat_{d\times 1}(\mathbb{F}_q)$ be the $k$-th standard basis. We have
\[G_n^\otimes(1,t) - G_n^\otimes(t,1) = \sum_{k = 1}^{d}\sum_{m=0}^{\ell-1}\sigma^{m-n}(h_k)  \otimes \tau^n\left(\mathfrak{e}_k^{\tr}\Theta_{\phi,\tau^{\ell-m}} \mathcal{G} \right ) - \mathfrak{e}_k^{\tr}\Theta_{\phi,\sigma^{\ell-m}}^*\mathcal{H}  \otimes \tau^m g_k. \]
\end{enumerate}
\end{proposition}

\begin{proof}
Part (1) is a straightforward calculation. The first part of (2) follows because the two terms being subtracted create a telescoping series, which leaves the highest and lowest degree (in $\tau$) terms after cancellation. The second part follows by using part (1), recalling that $a\tau = \tau a\twistinv$, and then repeatedly applying the first part of (2). Part (3) follows by recalling from \S \ref{SS:t-motive for Drinfeld} and \S \ref{SS:Dual t-motives for Drinfeld} that $t$ acts as $\phi_{\theta}$ on $M_G$ and as $\phi_{\theta}^*$ on $N_G$, then by applying parts (1) and (2) to the individual terms of $\phi_\theta$ and $\phi_\theta^*$.
\end{proof}

In what follows, we also obtain a factorization of $G_n^\otimes(1,1)$ similarly to \cite[Thm. 5.4(3)]{Gre22}. Let us denote
\[G_n^\otimes := G_n^\otimes(1,1).\]
\begin{proposition}\label{P:Gntensor factorization}
 We have the following factorization of $G_n^\otimes$:
\[((1\otimes t) - (t\otimes 1))G_n^\otimes = \sum_{k = 1}^{d}\sum_{m=0}^{\ell-1}\sigma^{m-n}(h_k)  \otimes \tau^n\left(\mathfrak{e}_k^{\tr}\Theta_{\phi,\tau^{\ell-m}} \mathcal{G} \right ) - \mathfrak{e}_k^{\tr}\Theta_{\phi,\sigma^{\ell-m}}^*\mathcal{H}  \otimes \tau^m g_k.\]
\end{proposition}

\begin{proof}
This proposition follows from Proposition \ref{P:Gntensor linearity properties}(3) after noting that $G_n^\otimes(1,t) = (1\otimes t)G_n^\otimes(1,1)$, and that $G_n^\otimes(t,1) = (t\otimes 1)G_n^\otimes(1,1)$.
\end{proof}

\section{Logarithms of Drinfeld modules}\label{S:Log of Drinfeld Modules}
Our goal in this section is to interpret the logarithms of Drinfeld modules in terms of formulas investigated in \cite{Gre22}. 

We continue to fix a Drinfeld module $\phi$ given by
\begin{equation}\label{E:DrinfeldSec3}
\phi_{\theta} = \theta + k_1 \tau + \dots + k_r\tau^r
\end{equation}
so that $\inorm{k_i}\leq 1$ for each $1\leq i \leq r-1$ and $k_r\in \mathbb{F}_q^{\times}$. For any positive integer $n$, consider \textit{the set $\phi[\theta^n]$ of $\theta^n$-torsion points} which consists of elements $z\in \mathbb{C}_{\infty}$ such that $\phi_{\theta^n}(z)=0$. Observe, by the Newton polygon method, that each non-zero element in $\phi[\theta]$ has norm $q^{1/(q^r-1)}$. Since, by Remark \ref{R:roclog}, $\log_{\phi}$ converges at any element $z\in \mathbb{C}_{\infty}$ satisfying $|z|<q^{q^r/(q^r-1)}$,  for each $1\leq i \leq r$, we may consider $\lambda_i:=\theta\log_{\phi}(\xi_i)\in \Ker(\exp_{\phi})$ where $\{\xi_1,\dots,\xi_r\}$ is a fixed $\mathbb{F}_q$-basis for $\phi[\theta]$. We call each $\lambda_i$  \textit{a fundamental period of $\phi$}. Furthermore, the set $\{\lambda_1,\dots,\lambda_r\}$ forms an $A$-basis for the period lattice $\Lambda_\phi$.

\subsection{The product formula for $\Upsilon$}\label{S:product}
Consider
\begin{equation}\label{D:B def}
B:=\begin{pmatrix}
\xi_1 & \xi_2 & \dots & \xi_r\\
\xi_1^q & \xi_2^q & \dots & \xi_r^q\\
\vdots &\vdots &  & \vdots \\
\xi_1^{q^{r-1}} &\xi_{2}^{q^{r-1}}& \dots & \xi_r^{q^{r-1}}
\end{pmatrix}\in \Mat_{r\times r}(\mathbb{C}_{\infty}).
\end{equation} 
Since $\xi_1,\dots,\xi_r$ are $\mathbb{F}_q$-linearly independent and $B$ is a Moore matrix, the inverse of $B$ exists. 

In what follows, we define certain quantities $\bbeta_n(t) \in K(t)$ from \cite[(6.4)]{EGP14} and refer the reader to \cite[\S5,6]{EGP14} for further details. By \textit{a partition of a set $S$}, we mean a disjoint union of sets whose union is equal to $S$. For any set $S\subset \mathbb{Z}$ and an integer $j$, we let $S+j:=\{s+j\ \ |s\in S\}$. For any $r\in \mathbb{Z}_{\geq 1}$, set $P_r(0):=\{(\emptyset,\dots,\emptyset)\}$ and for $n\in \mathbb{Z}_{\geq 1}$, we define $P_r(n)$ to be the collection of sets $(S_1,\dots,S_r)$ so that each $S_i$ is a subset of $\{0,1,\dots,n-1\}$ and the tuple $\{S_i+j\ \ | 1\leq i \leq r, \ \ 0\leq j \leq i-1\}$ forms a partition for $\{0,1,\dots,n-1\}$. We finally define 
\[
\bbeta_n(t) := \sum_{(S_1,\dots,S_r)\in P_r(n)} \prod_{i=1}^r \prod_{j\in S_i} \frac{k_i^{q^j}}{t-\theta^{q^{i+j}}}.
\]
We comment that, by \cite[(6.5)]{EGP13}, if we set 
\[
\log_{\phi}:=\sum_{n\geq 0}\beta_n\tau^n
\]
 then we have $\bbeta_n(\theta) = \beta_n$. 

%Further, we have by \cite[Cor. 6.9]{EGP14} that $\lVert \bbeta \rVert_\theta = |\beta_n|$. 

Let $\{\lambda_1,\dots,\lambda_r\}$ be the set of fundamental periods constructed as above by using $\{\xi_1,\dots,\xi_r\}$, forming an $A$-basis for $\Ker(\exp_{\phi})$. We define accordingly $\Upsilon \in \GL_r(\mathbb{T})$ given in \eqref{E:Upsilon}. We also set $F:=B^{-1}\Theta^{-1}B^{(1)}\in \GL_r(\mathbb{T})$ and $\Pi_n:=B\prod_{i=0}^{n}F^{(i)}\in \GL_r(\mathbb{T})$.

Khaochim and Papanikolas obtained a product formula for $\Upsilon$ as well as a certain expression for the entries of $\Pi_n$ in terms of $\bbeta_n$ which will later be essential for us to prove our main results.
\begin{theorem}[{Khaochim and Papanikolas, \cite[Prop. 4.3, Thm. 4.4]{KP23}}] \label{T:product} The following identities hold.
\begin{itemize}
\item[(i)] 
\[
(\Pi_{n})_{ij}=\left(\xi_{j}-\frac{t}{t-\theta}\sum_{\mu=0}^{n-(i-1)}\bbeta_{\mu}(t)\xi_{j}^{q^\mu}\right)^{(i-1)}.
\]
In particular, $\lim_{n\to \infty} (\Pi_{n})_{ij}^{(1)}$ exists with respect to the norm $\lVert \cdot\rVert_{\theta}$ on $\mathbb{T}_{\theta}$.
\item[(ii)] 
\[
\Upsilon=\lim_{n\to \infty} \Pi_n=B\prod_{n=0}^{\infty}F^{(n)}.
\]
\end{itemize}
\end{theorem}

%\begin{remark}\label{R:B} In a more general case, namely, when the coefficients $k_1,\dots,k_r$ of $\phi_{\theta}$ are chosen to be arbitrary elements in $\mathbb{C}_{\infty}$, Khaochim and Papanikolas constructed the matrix $B$ to be an element in $\Mat_{r\times r}(\mathbb{T})$ (see \cite[(3.6)]{KP23}). However, since our conditions on our Drinfeld module $\phi$ given in \eqref{D:B def} implies that $n_{\phi}=1$, by \cite[Prop. 3.17, Def. 3.21]{KP23}, $B$ becomes a matrix defined over $\mathbb{C}_{\infty}.$
%\end{remark}

Recall the matrices $\Theta$, $\Phi$ and $V$ from \S \ref{S: Prelim and Background}. For each $n\geq 1$, let us set 
\[
\mathcal{P}_n:=((\Phi^{\tr})^{-1})^{(1)}\cdots ((\Phi^{\tr})^{-1})^{(n)}\in \GL_r(\mathbb{T}).
\] 
We further define 
\begin{equation}\label{D:M def}
\mathfrak{U}:=V^{\tr}B^{(1)}.
\end{equation}
By \eqref{E:Phieq}, we have 
\begin{equation}\label{E:taufirst}
\begin{split}
&\mathcal{P}_{n+1}^{(-1)}\\
&=(\Phi^{\tr})^{-1}((\Phi^{\tr})^{-1})^{(1)}\cdots ((\Phi^{\tr})^{-1})^{(n)}\\
&=(V^{(-1)})^{\tr}\Theta^{-1}(\Theta^{-1})^{(1)}\cdots (\Theta^{-1})^{(n)}((V^{-1})^{\tr})^{(n)}\\
&=(V^{(-1)})^{\tr}B(B^{-1}\Theta^{-1}B^{(1)})(B^{-1}\Theta^{-1}B^{(1)})^{(1)}\cdots (B^{-1}\Theta^{-1}B^{(1)})^{(n)}(B^{-1})^{(n+1)}((V^{-1})^{\tr})^{(n)}\\
&=(V^{(-1)})^{\tr}\Pi_n (\mathfrak{U}^{-1})^{(n)}.
\end{split}
\end{equation}
Thus, by \eqref{E:taufirst}, we obtain 
\begin{equation}\label{E:sigmaeq}
\mathcal{P}_n=(\mathcal{P}_{n}^{(-1)})^{(1)}=((\Phi^{\tr})^{-1}((\Phi^{\tr})^{-1})^{(1)}\cdots ((\Phi^{\tr})^{-1})^{(n-1)})^{(1)}=V^{\tr} \Pi_{n-1}^{(1)} (\mathfrak{U}^{-1})^{(n)}.
\end{equation}

For each $n\geq 1$, we further set $\Psi_{n}:=V^{-1}((\Pi_n^{(1)})^{\tr})^{-1}$. Recall the invertible matrix $\Psi$ from \S \ref{SS:Dual t-motives for Drinfeld} and observe, by Theorem \ref{T:product}, that 
\begin{equation}\label{E:rigiddual}
\Psi=\lim_{n\to \infty}\Psi_{n}=V^{-1}((\Upsilon^{(1)})^{\tr})^{-1}.
\end{equation}
By taking the inverse of very left and right hand side of \eqref{E:taufirst},  we have
\begin{equation}\label{E:tau}
\mathcal{S}_n=(\Phi^{\tr})^{(n)}(\Phi^{\tr})^{(n-1)}\cdots (\Phi^{\tr})=\mathfrak{U}^{(n)}\Pi_n^{-1}((V^{(-1)})^{-1})^{\tr}=\mathfrak{U}^{(n)}(\Psi_{n}^{(-1)})^{\tr}.
\end{equation}
Thus, for any $\tilde{n}\in \tilde{N}_{\phi}$ ($\tilde{m}\in M_{\phi}$ respectively) given by $\tilde{n}=\sum_{i=1}^ra_i \mathfrak{d}^{\phi}_i$ ($\tilde{m}=\sum_{i=1}^rb_i \mathfrak{c}^{\phi}_i$ respectively), using \eqref{E:tauaction}, \eqref{E:tauact} and \eqref{E:sigmainv}, we have 
\begin{equation}\label{E:sigmatau}
\sigma^{-n}(\tilde{n})=\mathcal{P}_n\begin{bmatrix}
a_1\\
\vdots\\
a_r
\end{bmatrix}^{(n)} \text{ and } \ \tau^{n}(\tilde{m})=[b_1,\dots,b_r]^{(n)} \mathcal{S}_{n-1}.
\end{equation}

\subsection{Tensor construction for Drinfeld modules}\label{SS:Tensor for Drinfeld modules} For the remaining of this section, all the tensor products are over $\mathbb{C}_{\infty}$ unless otherwise explicitly stated and hence, to ease the notation, we again avoid the subscript in our tensor product notation.

We use the bases described in \S \ref{SS:t-motive for Drinfeld} and \S \ref{SS:Dual t-motives for Drinfeld}. Let $G=(\mathbb{G}_{a/\mathbb{C}_{\infty}},\phi)$ where $\phi$ is as given in \eqref{E:DrinfeldSec3}. In this case, if we set
\begin{equation}\label{E:Gn tensor def Drinfeld module}
G_n^\otimes =  \sum_{i=0}^n  \sigma^{-i}(\mathfrak{d}^{\phi}_1)\otimes \tau^i(m_1)= \sum_{i=0}^n  \sigma^{-i}(1)\otimes \tau^i(1)\in \tilde{N}_{\phi}\otimes M_{\phi}
\end{equation}
then Proposition \ref{P:Gntensor factorization} reduces to
\begin{equation}\label{E:Gn tensor factorization Drinfeld}
\begin{split}
((1\otimes t) - (t\otimes 1))G_n^\otimes &=  \sum_{\ell=0}^{r-1}\sigma^{\ell-n}(\mathfrak{d}^{\phi}_1) \otimes \tau^n\left(\Theta_{\phi,\tau^{r-\ell}}\cdot m_1\right) - \Theta_{\phi,\sigma^{r-\ell}}^*(\mathfrak{d}^{\phi}_1)  \otimes \tau^\ell(m_1) \\
&= \sigma^{r-1-n}(\mathfrak{d}_1^{\phi})\otimes \tau^{n}(k_r \tau )+ \sum_{\ell=0}^{r-2}\sigma^{\ell-n}(\mathfrak{d}^{\phi}_1)  \otimes \tau^{n+1}(\mathfrak{c}^{\phi}_{\ell+1})\\
&\ \ \ \ \ - \sum_{\ell=0}^{r-1}\Theta_{\phi,\sigma^{r-\ell}}^*(\mathfrak{d}^{\phi}_1)  \otimes \tau^\ell(m_1) \\
&=\sum_{\ell=0}^{r-1}\sigma^{\ell-n}(\mathfrak{d}^{\phi}_1)  \otimes \tau^{n+1}(\mathfrak{c}^{\phi}_{\ell+1}) - \sum_{\ell=0}^{r-1}\Theta_{\phi,\sigma^{r-\ell}}^*(\mathfrak{d}^{\phi}_1)  \otimes \tau^\ell(m_1).
\end{split}
\end{equation}
 Going forward we will denote $\gamma := \sum_{\ell=0}^{r-1}\Theta_{\phi,\sigma^{r-\ell}}^*(\mathfrak{d}^{\phi}_1)  \otimes \tau^\ell(m_1)$. This term has no dependence on $n$ and thus makes no contribution towards the convergence of the left hand side, thus we will minimize the notation of these terms throughout the following discussion. On other hand, by \cite[Cor. 4.5]{Gre22}, for $n\geq 0$, we have $\delta_{0}^{N_{\phi}}(\sigma^{-n}(\mathfrak{d}_1^{\phi}))=\beta_n$. Then Proposition \ref{P:deltazeromap} yields the following useful lemma.
 \begin{lemma}\label{L:logcoef} For $n\geq 0$, in $M_{\phi}$, we have
 \[
 (t-\theta)\sum_{i=0}^n\beta_i\tau^i(m_1)=\sum_{\ell=0}^{r-1}\beta_{n-\ell}\tau^{n+1}(\mathfrak{c}^{\phi}_{\ell+1}).
 \]
 \end{lemma}

We now wish to move towards viewing these identities as living in rings of matrices over Tate algebras. To this end, as discussed in \S\ref{SS:varphi map}, we identify $M_\phi \isom \Mat_{1\times r}(\C_\infty[t])$ and identify $N_\phi \isom \Mat_{r\times 1}(\C_\infty[t])$ using the bases described above. Let  $\mathfrak{e}_{\ell}\in \Mat_{r\times 1}(\mathbb{F}_q)$ be the $\ell$-th standard basis. Applying the definition of the $\tau-$ and $\sigma-$ action on these bases detailed in \eqref{E:sigmatau}, when $n\geq r-1$, formula \eqref{E:Gn tensor factorization Drinfeld} becomes
\begin{equation}\label{E:Gnexpression}
((1\otimes t) - (t\otimes 1))G_n^\otimes = \sum_{\ell=0}^{r-1}\mathcal{P}_{n-\ell}\mathfrak{e}_1\otimes \mathfrak{e}_{\ell+1}^{\tr}\mathcal{S}_{n} - \gamma.
\end{equation}
A short calculation shows that both of these (finite) sums are in $\TT_\theta^r \otimes \TT_\theta^r$.

\begin{remark}
We note that we write vectors to the left of the tensor as a column and vectors to the right as a row in order to simplify notation in what comes next, namely, so that we can multiply by $r\times r$ matrices on the left and on the right of such a simple tensor and it is clear what that means. To avoid cumbersome notation, we will denote such elements as living in $\TT_\theta^r \otimes \TT_\theta^r$ rather than $\TT_\theta^r \otimes \Mat_{1\times r}(\TT_\theta)$.
\end{remark}

Our immediate goal is to prove that the right hand side of \eqref{E:Gnexpression} converges in some ring of Tate algebras as $n\to \infty$. 

\begin{definition}
Let $c\in \mathbb{C}_{\infty}^{\times}$. Recall the norm $\lVert \cdot\rVert_c$ on $\TT_c^r$ from \S \ref{SS:Anderson Generating Functions}. We extend this norm to simple tensors $a\otimes b \in \TT_c^r \otimes \TT_c^r$ by setting 
\[\lVert a\otimes b\rVert_c = \lVert a\rVert_c \cdot \lVert b\rVert_c,\]
then extending it to all $\TT_c^r \otimes \TT_c^r$ by taking the supremum over all sums involving simple tensors. It follows trivially from the definition that this is in fact a non-archimedean (or ultrametric) norm on $\TT_c^r \otimes \TT_c^r$. In fact, this is an example of a cross norm on the tensor product of two Banach spaces (see \cite[\S6]{Rya02} for more details on cross norms). We then form the completion of $\TT_c^r \otimes \TT_c^r$ under this norm, and denote the resulting space $\widehat{\TT_c^r \otimes \TT_c^r}$.
\end{definition}

\begin{lemma}\label{L:Totimes convergence}
 For $a_n,b_n\in \TT_c^r$, the sum of simple tensors
\[\sum_{n=0}^\infty a_n\otimes b_n\]
converges in $\widehat{\TT_c^r \otimes \TT_c^r}$ if and only if  $\lVert a_n\otimes b_n\rVert_c\to 0$ as $n\to \infty$.
\end{lemma}

\begin{proof}
First, note that the sum $\sum_{i=0}^\infty a_n\otimes b_n$ trivially diverges if $\lVert a_n\otimes b_n\rVert_c$ does not converge to 0. On the other hand, if the individual simple tensors do converge to 0 in norm, then the convergence of the series follows from the ultrametric triangle inequality.
\end{proof}

\subsection{The element $\alpha_n$ }\label{SS:alpha convergence}
Our main goal in this subsection is to define an element $\alpha_n\in \mathbb{T}^r_{\theta}\otimes \mathbb{C}_{\infty}^r$ for each $n\in \mathbb{Z}_{\geq 1}$ which will be useful 
to interpret the right hand side of \eqref{E:Gnexpression} in terms of matrices $\Pi_n$ and $\Psi_n$ in \eqref{E:Gn tensor factorization Drinfeld}.

Recall the matrix $B$ from \eqref{D:B def} and set $\mathfrak{D}:=\det(B)$ and write
$
B^{-1}=\frac{1}{\mathfrak{D}}(\mathfrak{c}_{ji})_{ij}$ where $\mathfrak{c}_{ji}$ is the $(j,i)$-cofactor of $B$. 
By the construction of $B$, for each $1\leq \ell \leq r$, we  obtain 
\begin{equation}\label{E:cof}
\mathfrak{c}_{1\ell}=\begin{cases}
-\mathfrak{c}_{r\ell}^q & \text{ if $r$ is even}\\ 
\mathfrak{c}_{r\ell}^q & \text{ if $r$ is odd.}
\end{cases}
\end{equation}
Since $\xi_1,\dots,\xi_r$ are elements in $\phi[\theta]$, we have 
\begin{equation}\label{E:inverse}
B^{(1)}=\begin{pmatrix}
0&1 & & & \\
& & \ddots& & \\
& &  &\ddots & \\
& & &  &1 \\
-\frac{\theta}{k_r}&-\frac{k_1}{k_r}&\dots&\dots&-\frac{k_{r-1}}{k_r}
\end{pmatrix}B.
\end{equation}
This relation shows that 
\begin{equation}\label{E:det}
\mathfrak{D}^q=\begin{cases}
\frac{\theta \mathfrak{D}}{k_r}  & \text{ if $r$ is even}\\ 
-\frac{\theta \mathfrak{D}}{k_r} & \text{ if $r$ is odd.}
\end{cases}
\end{equation}
Hence, for each $m\geq 1$, we have 
\begin{equation}\label{E:invofB}
(B^{-1})^{(m)}=\begin{cases}
\frac{k_r^{1+q+\dots+q^{m-1}}}{\theta^{1+q+\dots+q^{m-1}}\mathfrak{D}}(\mathfrak{c}_{ji}^{q^m})_{ij}  & \text{ if $r$ is even}\\ 
(-1)^m\frac{k_r^{1+q+\dots+q^{m-1}}}{\theta^{1+q+\dots+q^{m-1}}\mathfrak{D}}(\mathfrak{c}_{ji}^{q^m})_{ij}& \text{ if $r$ is odd.}
\end{cases}
\end{equation}

For any positive integer $n$, in what follows, we define $\alpha_n\in \mathbb{T}_{\theta}^{r}\otimes \mathbb{C}_{\infty}^{r}$, a quantity related to the right-hand side of \eqref{E:Gnexpression} without the $\gamma$ term and after factoring out $\Pi_{n-1}\twist$ on the left and $(\Psi_n\twistinv)^\tr$ on the right,
\begin{multline}\label{E:alphadef}
\alpha_n:=(\mathfrak{U}^{-1})^{(n-(r-1))}\mathfrak{e}_{1}\otimes \mathfrak{e}_{r}^{\tr}\mathfrak{U}^{(n)} +F^{(n-(r-2))}(\mathfrak{U}^{-1})^{(n-(r-2))}\mathfrak{e}_{1}\otimes \mathfrak{e}_{r-1}^{\tr}\mathfrak{U}^{(n)}+\\
F^{(n-(r-2))}F^{(n-(r-3))} (\mathfrak{U}^{-1})^{(n-(r-3))}\mathfrak{e}_{1}\otimes \mathfrak{e}_{r-2}^{\tr}\mathfrak{U}^{(n)}+\dots+\\
F^{(n-(r-2))}\cdots F^{(n)} (\mathfrak{U}^{-1})^{(n)}\mathfrak{e}_{1}\otimes \mathfrak{e}_{1}^{\tr}\mathfrak{U}^{(n)} \in \mathbb{C}_{\infty}(t)^r\otimes \mathbb{C}_{\infty}^r.
\end{multline}

By \cite[Prop. 3.5]{KP23}, we have 
\[
\tilde{F}:=F-\Id_r=-\frac{t}{t-\theta}B^{-1}\begin{pmatrix}
    \xi_1&\xi_2&\cdots&\xi_r\\
    0&\cdots & \cdots & 0\\
    \vdots & & & \vdots\\
    0&\cdots & \cdots & 0
    \end{pmatrix}=-\frac{t}{t-\theta}\frac{1}{\mathcal{D}}\begin{pmatrix}
    \mathfrak{c}_{11}\xi_1 & \mathfrak{c}_{11} \xi_2&\cdots& \mathfrak{c}_{11}\xi_r\\
    \mathfrak{c}_{21}\xi_1& \mathfrak{c}_{21} \xi_2&\cdots& \mathfrak{c}_{21}\xi_r\\
    \vdots & & & \vdots\\
    \mathfrak{c}_{r1}\xi_1& \mathfrak{c}_{r1} \xi_2&\cdots& \mathfrak{c}_{r1}\xi_r    \end{pmatrix}.\]
Recall that for each $1\leq j \leq r$, we have $|\xi_j|=q^{1/(q^r-1)}$. On the other hand, since $|\mathcal{D}|=q^{1/(q-1)}$ and $|\mathfrak{c}_{ij}|\leq q^{\frac{q+\dots+q^{r-1}}{q^r-1}}$, for $n\geq 1$, we see that 
\begin{equation}\label{E:boundtildeF}
||\tilde{F}^{(n)}_{|t=\theta}||\leq q^{1-q^n}.
\end{equation}

Now using $F=\tilde{F}+\Id_r$ in \eqref{E:alphadef}, we further let 
\[
\alpha_n:=\tilde{\alpha}_n+\beta^{(n-(r-2))}
\]
so that 

%\begin{remark} Let us set $(a\otimes b)^{(1)}:=a^{(1)}\otimes b^{(1)}$ for any $a,b\in \mathbb{T}$. One can indeed show that $\alpha_n^{(1)}=\alpha_n$ which implies that $\alpha_n\in \mathbb{F}_q(t)^r\otimes \mathbb{F}_q^r$. In what follows, we will make a more explicit calculation and then provide the exact value of $\alpha_n$ in Theorem \ref{T:alpha}. The precise relationship between \eqref{E:Gnexpression} and $\alpha_n$ will be also given in \eqref{E:Gn tensor factorization Drinfeld2}.
%\end{remark}

\begin{multline}\label{E:beta1}
\beta:=(\mathfrak{U}^{-1})^{(-1)}\mathfrak{e}_{1}\otimes \mathfrak{e}_{r}^{\tr}\mathfrak{U}^{(r-2)} +\mathfrak{U}^{-1}\mathfrak{e}_{1}\otimes \mathfrak{e}_{r-1}^{\tr}\mathfrak{U}^{(r-2)}+(\mathfrak{U}^{-1})^{(1)}\mathfrak{e}_{1}\otimes \mathfrak{e}_{r-2}^{\tr}\mathfrak{U}^{(r-2)}\\
+\dots+ (\mathfrak{U}^{-1})^{(r-2)}\mathfrak{e}_{1}\otimes \mathfrak{e}_{1}^{\tr}\mathfrak{U}^{(r-2)}\\
=\frac{1}{k_{r}^{(-1)}}B^{-1}\mathfrak{e}_{r}\otimes k_r^{(-1)} \mathfrak{e}_1^{\tr} B^{(r-1)}+\frac{1}{k_{r}}(B^{-1})^{(1)}\mathfrak{e}_{r}\otimes (k_{r-1},k_r,0,\dots,0) B^{(r-1)}+\\
\frac{1}{k_{r}^{(1)}}(B^{-1})^{(2)}\mathfrak{e}_{r}\otimes (k_{r-2}^q,k_{r-1}^q,k_r^q,0,\dots,0)B^{(r-1)}
+\dots 
+\\
\frac{1}{k_{r}^{(r-2)}}(B^{-1})^{(r-1)}\mathfrak{e}_{r}\otimes (k_1^{q^{r-2}},\dots,k_{r-1}^{q^{r-2}},k_r^{q^{r-2}})B^{(r-1)}\in \mathbb{C}_{\infty}^{r}\otimes \mathbb{C}_{\infty}^{r}
\end{multline}

\begin{remark}\label{R:1} \textbf{Important Notational Comment:} In this remark, to distinguish the base spaces where our tensor products are over, we explicitly state them in our notation. Since $\mathfrak{c}^{\phi}$ ($\mathfrak{d}^{\phi}$ respectively) forms a $\mathbb{C}_{\infty}[t]$-basis for $M_\phi$ (for $N_\phi$ respectively), we conclude that $\{\mathfrak{d}_{i}^{\phi}\otimes \mathfrak{c}_j^{\phi}\}$ for $1\leq i,j\leq r$ forms a $\C_\infty[t]\otimes_{\C_\infty} \mathbb{C}_{\infty}[t]$-basis for $M_\phi\otimes_{\C_\infty} N_\phi$. We then tensor each of those motives with $\TT$ over $\C_\infty[t]$ to get $(\TT\otimes_{\C_\infty[t]}  M_\phi)\otimes_{\C_\infty} (\TT\otimes_{\C_\infty[t]} N_\phi)$ and view this as a $\TT\otimes_{\C_\infty}\TT$-free module.  We consider the map $f: (\TT\otimes_{\C_\infty[t]}  M_\phi)\otimes_{\C_\infty} (\TT\otimes_{\C_\infty[t]} N_\phi)\to \Mat_{r\times r}(\TT\otimes_{\C_\infty}\TT)$ sending each 
\[
g=\sum_{i,j=1}^{r}b_{ij}\mathfrak{d}_{i}^{\phi}\otimes \mathfrak{c}_j^{\phi}\in (\TT\otimes_{\C_\infty[t]}  M_\phi)\otimes_{\C_\infty} (\TT\otimes_{\C_\infty[t]} N_\phi)
\]
to $f(g):= (b_{ij})\in \Mat_{r\times r}(\TT\otimes_{\C_\infty}\TT)$.
Now observe that
\begin{align*}
(\TT\otimes_{\C_\infty[t]}  M_\phi)\otimes_{\C_\infty} (\TT\otimes_{\C_\infty[t]} N_\phi) & = (\TT\otimes_{\C_\infty[t]}  \C_\infty[t]^r)\otimes_{\C_\infty} (\TT\otimes_{\C_\infty[t]} \C_\infty[t]^r)\\
&= \TT^r\otimes_{\C_\infty} \TT^r\\
&= (\TT\otimes_{\C_\infty}\TT)^{r^2}\\
&= \Mat_{r\times r}(\TT\otimes_{\C_\infty}\TT),
\end{align*}
with $\TT\otimes_{\C_\infty}\TT$-basis given by $\{\mathfrak{d}_{i}^{\phi}\otimes \mathfrak{c}_j^{\phi}\}$ as above. Thus, $f$ forms an isomorphism of $\TT\otimes_{\C_\infty}\TT$-modules. This calculation applies equally to $\TT_c$. However, we will primarily use this construction in three specific cases, where it reduces significantly.

First, for an element such as $\beta$ given above, we have $\beta \in \C_\infty^r\otimes_{\C_\infty} \C_\infty^r \in \TT^r\otimes_{\C_\infty} \TT^r$. Thus $\beta \in \Mat_{r\times r}(\C_\infty\otimes_{\C_\infty} \C_\infty)$, so we will shortcut to viewing $\beta \in \Mat_{r\times r}(\C_\infty)$ using the natural isomorphism $\C_\infty\otimes_{\C_\infty}\C_\infty = \C_\infty$. In particular we have
\[
f\left(\sum_{i=1}^{r}\mathfrak{d}^{\phi}_i\otimes \mathfrak{c}^{\phi}_i\right)=\Id_r.
\]

Second, we use it for the element $\alpha_n \in \TT_\theta^r \otimes_{\C_\infty} \C_\infty^r$, so we view it in $\Mat_{r\times r}(\TT_\theta\otimes_{\C_\infty} \C_\infty) = \Mat_{r\times r}(\TT_\theta)$. 

%Third, we will use it on elements in $\TT_\theta^r \otimes_{\C_\infty} \TT^r$ as in Theorem \ref{T:convDrinfeld}, to which we then apply the map $\delta_0^{N_\phi}\otimes \Id$, as we do in the proof of Theorem \ref{T:logdrinfeld}. Since $\delta_0^{N_\phi}:\TT_\theta^r \to \C_\infty$, the image is in $\C_\infty \otimes_{\C_\infty} \TT^r = \TT^r$, and so we state the formulas for Theorem \ref{T:logdrinfeld} as living in $\Mat_{1\times r}(\TT)$. Thus, we will never really need to work with matrices over $\TT_\theta \otimes_{\C_\infty} \TT$, but we give the calculation here to explain the common space where all of our identities and calculations can be seen inside.
\end{remark}
Recall the matrix $\mathfrak{U}$ defined in \eqref{D:M def}. The next lemma will be crucial to determine the limiting behavior of $\alpha_n$. 
\begin{lemma}\label{L:bound} Let $c\in \mathbb{C}_{\infty}^{\times}$ and  $\mathfrak{J}$ be an element in $\Mat_{r\times r}(\mathbb{T}_c)$ such that for sufficiently large $n$, each entry of $\mathfrak{J}^{(n)}$ has $\lVert \cdot \rVert_{c}$-norm less than 1. Let $m$ be a non-negative integer. Then, for each $2\leq j \leq r$, in $\mathbb{T}_c^r\otimes \mathbb{C}_{\infty}^r$, we have
	\[
	\lim_{n\to \infty} \mathfrak{J}^{(n-m)}(\mathfrak{U}^{-1})^{(n-(r-j))}\mathfrak{e}_{1}\otimes \mathfrak{e}_{r-(j-1)}^{\tr}\mathcal{M}^{(n)}=0.
	\]
\end{lemma}
\begin{proof} Note that 
	\begin{equation}\label{E:arg1}
	\begin{split}
	(\mathfrak{U}^{-1})^{(j-2)}\mathfrak{e}_{1}\otimes \mathfrak{e}_{r-(j-1)}^{\tr}\mathfrak{U}^{(r-2)}&=(B^{-1})^{(j-1)}((V^{\tr})^{-1})^{(j-2)}\mathfrak{e}_1\otimes \mathfrak{e}_{r-(j-1)}^{\tr}(V^{\tr})^{(r-2)}B^{(r-1)}\\
	&=(B^{-1})^{(j-1)}k_r^{^{-q^{j-2}}}\mathfrak{e}_r\otimes (k_{r-(j-1)}^{q^{j-2}},\dots,k_{r-1}^{q^{j-2}},k_r^{q^{j-2}},0,\dots,0) B^{(r-1)}.
	\end{split}
\end{equation}
	On the other hand, observe that 
	\begin{equation}\label{E:arg2}
	(\mathfrak{U}^{-1})^{(n-(r-j))}\mathfrak{e}_{1}\otimes \mathfrak{e}_{r-(j-1)}^{\tr}\mathfrak{U}^{(n)}=\left((\mathfrak{U}^{-1})^{(j-2)}\mathfrak{e}_{1}\otimes \mathfrak{e}_{r-(j-1)}^{\tr}\mathfrak{U}^{(r-2)}\right)^{(n-r+2)}.
	\end{equation}
For each $1\leq i \leq r-1$, let 
\[
\mathcal{F}_i:=(B^{-1})^{(i)}\mathfrak{e}_{r}\otimes (k_{r-i}^{q^{i-1}},\dots,k_{r-1}^{q^{i-1}},k_r,0,\dots,0) B^{(r-1)}\in \mathbb{C}_{\infty}^r\otimes\mathbb{C}_{\infty}^r\cong \Mat_{r}(\mathbb{C}_{\infty}).
\]
We realize $\mathcal{F}_i\in \Mat_{r}(\mathbb{C}_{\infty})$ and let $\lVert \mathcal{F}_i\rVert$ be the maximum among the norms of the entries of $\mathcal{F}_i$. Since $k_r\in \mathbb{F}_q^{\times}$ and each entry of $\mathfrak{J}^{(n)}$ has $\lVert\cdot \rVert_c$-norm less than 1 for sufficiently large $n$, by \eqref{E:arg1},  \eqref{E:arg2} and the continuity of the twisting operation, it suffices to show that $\log_q(\lVert \mathcal{F}_i\rVert)\leq 0$ for each $i$. 

Note, from \eqref{E:det}, that $\inorm{\mathfrak{D}}=q^{1/(q-1)}$.  Finally, for any $1\leq \mu \leq r$, since $\xi_{\mu}$ is a $\theta$-torsion point, one obtains
\[
k_{r-i}^{q^{i-1}}\xi_\mu^{q^{r-1}}+\dots+k_{r-1}^{q^{i-1}}\xi_\mu^{q^{r+i-2}}+k_{r}\xi_\mu^{q^{r+i-1}}=-\theta^{q^{i-1}}\xi_{\mu}^{q^{i-1}}-k_1^{q^{i-1}}\xi_{\mu}^{q^{i}}-\dots-k_{r-i-1}^{q^{i-1}}\xi_{\mu}^{q^{r-2}}.
\]
Since $\inorm{k_i}\leq 1$, we see that 
\[
\log_q(\inorm{k_{r-i}^{q^{i-1}}\xi_\mu^{q^{r-1}}+\dots+k_{r-1}^{q^{i-1}}\xi_\mu^{q^{r+i-2}}+k_{r}\xi_\mu^{q^{r+i-1}}})\leq q^{i-1}+\frac{q^{i-1}}{q^r-1}.
\]
Similarly, a direct calculation implies that, for each $1\leq \nu \leq r$, $\inorm{\mathfrak{c}_{r\nu}}$ is bounded by $q^{1+q+\dots+q^{r-2}}/(q^r-1)$. Combining all these facts above, we obtain
\begin{align*}
\log_q(\lVert \mathcal{F}_i\rVert)&\leq -\frac{q^i}{q-1}+\frac{q^i+\cdots+q^{i+r-2}}{q^r-1}+q^{i-1}+\frac{q^{i-1}}{q^r-1}\\
&= -\frac{q^i}{q-1}+q^i\frac{1+q+\cdots+q^{r-1}}{q^r-1}\\
&=0
\end{align*}
as desired.
\end{proof}
Let $\alpha_n(\theta)$ denote the substitution $t=\theta$ on the left hand side of the tensor product in \eqref{E:alphadef} and we further set $\tilde{\alpha}_n(\theta):=\alpha_n(\theta)-\beta^{(n-(r-2))}$. The proof of Lemma \ref{L:bound} together with \eqref{E:boundtildeF} immediately implies our next lemma.

\begin{lemma}\label{L:tildelim} In $\mathbb{C}_{\infty}^r\otimes \mathbb{C}_{\infty}^r$, we have
\[
\lim_{n\to \infty}\tilde{\alpha}_n(\theta)=0.
\]    
\end{lemma}

In what follows, we state our next theorem whose proof will be provided in \S\ref{S:proofofbeta}.

\begin{theorem}\label{T:alpha}In $\mathbb{C}_{\infty}^r\otimes \mathbb{C}_{\infty}^r$, we have \[
\beta=\sum_{i=1}^{r}\mathfrak{d}^{\phi}_i\otimes \mathfrak{c}^{\phi}_i=\mathfrak{e}_1\otimes \mathfrak{e}_1^{\tr}+\dots+\mathfrak{e}_r\otimes \mathfrak{e}_r^{\tr}.\]
In other words, via the identification in Remark \ref{R:1}, $\beta=\Id_r$.
\end{theorem}

As a consequence of Lemma \ref{L:tildelim} and Theorem \ref{T:alpha}, we obtain our next corollary.
\begin{corollary}\label{C:alpha} In $\mathbb{C}_{\infty}^r\otimes \mathbb{C}_{\infty}^r$, we have
\[
\lim_{n\to \infty}\alpha_n(\theta)=\lim_{n\to \infty}(\tilde{\alpha}_n(\theta)+\beta^{(n-(r-2))})=\sum_{i=1}^{r}\mathfrak{d}^{\phi}_i\otimes \mathfrak{c}^{\phi}_i=\mathfrak{e}_1\otimes \mathfrak{e}_1^{\tr}+\dots+\mathfrak{e}_r\otimes \mathfrak{e}_r^{\tr}.
\]    
In other words, via the identification in Remark \ref{R:1}, $\lim_{n\to \infty}\alpha_n(\theta)=\Id_r$.
\end{corollary}

\subsection{Proof of Theorem \ref{T:alpha}}\label{S:proofofbeta} The proof of Theorem \ref{T:alpha} occupies \S \ref{SS:Even rank case} and \S \ref{SS:Odd rank case}.  Note, as it is used in the proof of Lemma \ref{L:bound}, that since $\xi_1,\dots,\xi_{r}$ are elements in $\phi[\theta]$, we have 
\begin{equation}\label{E:torsionid}
k_i\xi_j^{q^{i}}+\cdots + k_r\xi_j^{q^{r}}=-\theta \xi_j-k_1\xi_j^{q}-\cdots -k_{i-1}\xi_j^{q^{i-1}}
\end{equation}
for any $2\leq i \leq r$ and $1\leq j \leq r$. We comment that by definition, and since Frobenius twisting is continuous, we know that 
\[\alpha\twist =\lim_{n\to \infty}\beta^{(n-(r-2)+1)} = \alpha,\]
which shows that $\alpha$ is defined over $\F_q$. The proof we give here shows that directly by demonstrating that in fact $\alpha$ equals the identity.

\subsubsection{Even rank case}\label{SS:Even rank case} Let us set $r=2n$ for some positive integer $n$. By using \eqref{E:det}, \eqref{E:invofB}, \eqref{E:torsionid} as well as the definition of $\beta$ given in \eqref{E:beta1}, we obtain
\begin{multline*}
\beta=\frac{1}{\mathfrak{D}}\begin{pmatrix}
\frac{\mathfrak{c}_{(2n)1}}{k_{{2n}}^{(-1)}} & \mathfrak{c}_{(2n)1}^q & \dots & \mathfrak{c}_{(2n)1}^{q^{2n-1}}\\
\frac{\mathfrak{c}_{(2n)2}}{k_{{2n}}^{(-1)}} & \mathfrak{c}_{(2n)2}^q & \dots & \mathfrak{c}_{(2n)2}^{q^{2n-1}}\\
\vdots &\vdots &  & \vdots \\
\frac{\mathfrak{c}_{(2n)(2n)}}{k_{{2n}}^{(-1)}} & \mathfrak{c}_{(2n)(2n)}^q & \dots & \mathfrak{c}_{(2n)(2n)}^{q^{2n-1}}
\end{pmatrix}
\begin{pmatrix}
0& \dots & \dots & \dots & 0 &k_{{2n}}^{(-1)}\\
-1& -\frac{k_1}{\theta} & -\frac{k_2}{\theta}  & \dots & -\frac{k_{2n-2}}{\theta} &0\\
& -\frac{k_{2n}}{\theta}  & -\frac{k_{2n}k_1^q}{\theta^{1+q}}  & & -\frac{k_{2n}k_{2n-3}^q}{\theta}  &\vdots\\
&  & -\frac{k_{2n}^{1+q}}{\theta^{1+q}}  &\ddots  & \vdots &\vdots\\
&  &   & \ddots &-\frac{k_{2n}^{1+\dots+q^{2n-4}}k_{1}^{q^{2n-3}}}{\theta^{1+\dots+q^{2n-3}}} &\vdots\\
&  &   & &-\frac{k_{2n}^{1+\dots+q^{2n-3}}}{\theta^{1+\dots+q^{2n-3}}}  & 0
\end{pmatrix}B.
\end{multline*}
Let us set 
$
\mathfrak{B}:=\beta B^{-1}.
$
Our goal is to show that $\mathfrak{B}=B^{-1}$. Firstly, by 
\eqref{E:cof} and a simple calculation, the first and last column of $\mathfrak{B}$ and $B^{-1}$ are equal. Hence $\mathfrak{B}=B^{-1}$ when $n=1$. Now assume that $n>1$. Note that \eqref{E:inverse} also implies 
\begin{equation}\label{E:inv2}
(B^{(1)})^{-1}\begin{pmatrix}
0&1 & & & \\
& & \ddots& & \\
& &  &\ddots & \\
& & &  &1 \\
-\frac{\theta}{k_{2n}}&-\frac{k_1}{k_{2n}}&\dots&\dots&-\frac{k_{2n-1}}{k_{2n}}
\end{pmatrix}=B^{-1}=\frac{1}{\mathfrak{D}}\begin{pmatrix}
-\mathfrak{c}_{(2n)1}^q & \mathfrak{c}_{21} & \dots &\mathfrak{c}_{(2n-1)1}& \mathfrak{c}_{(2n)1}\\
-\mathfrak{c}_{(2n)2}^q & \mathfrak{c}_{22} & \dots & \mathfrak{c}_{(2n-1)2}&\mathfrak{c}_{(2n)2}\\
\vdots &\vdots &  & \vdots & \vdots \\
-\mathfrak{c}_{(2n)(2n)}^q & \mathfrak{c}_{2(2n)} & \dots & \mathfrak{c}_{(2n-1)(2n)}& \mathfrak{c}_{(2n)(2n)}
\end{pmatrix}.
\end{equation}
For each $2\leq m \leq 2n-1$ and $1\leq i \leq 2n$, we claim  that 
\begin{multline}\label{E:claim}
\mathfrak{c}_{mi}=-\bigg(\frac{\mathfrak{c}_{(2n)i}^{q^m}k_{2n}^{1+q+\dots+q^{m-2}}}{\theta^{1+q+\dots+q^{m-2}}}+\frac{k_1^{q^{m-2}}\mathfrak{c}_{(2n)i}^{q^{m-1}}k_{2n}^{1+q+\dots+q^{m-3}}}{\theta^{1+q+\dots+q^{m-2}}}+\frac{k_2^{q^{m-3}}\mathfrak{c}_{(2n)i}^{q^{m-2}}k_{2n}^{1+q+\dots+q^{m-4}}}{\theta^{1+q+\dots+q^{m-3}}}+\dots+\\
\frac{k_{m-2}^q\mathfrak{c}_{(2n)i}^{q^2}k_{2n}}{\theta^{1+q}}+\frac{k_{m-1}\mathfrak{c}_{(2n)i}^q}{\theta} \bigg).
\end{multline}
When $m=2$, we have 
$
\mathfrak{c}_{2i}=-\frac{\mathfrak{c}_{(2n)i}^{q^2}k_{2n}}{\theta}-\frac{k_1\mathfrak{c}_{(2n)i}^{q}}{\theta}.
$
Assume that it holds for $m$. Note, by \eqref{E:det} and \eqref{E:inv2}, we have 
\[
\mathfrak{c}_{(m+1)i}= \frac{\mathfrak{c}_{mi}^qk_{2n}}{\theta}-\frac{k_m\mathfrak{c}_{(2n+1)i}^q}{\theta}.
\]
Using the induction hypothesis, we obtain
\begin{align*}
\mathfrak{c}_{(m+1)i}&=\frac{\mathfrak{c}_{mi}k_{2n}}{\theta}-\frac{k_m\mathfrak{c}_{(2n+1)i}^q}{\theta}\\
&=-\bigg(\frac{\mathfrak{c}_{(2n)i}^{q^{m+1}}k_{2n}^{1+q+\dots+q^{m-1}}}{\theta^{1+q+\dots+q^{m-1}}}+\frac{k_1^{q^{m-1}}\mathfrak{c}_{(2n)i}^{q^{m}}k_{2n}^{1+q+\dots+q^{m-2}}}{\theta^{1+q+\dots+q^{m-1}}}+\frac{k_2^{q^{m-2}}\mathfrak{c}_{(2n)i}^{q^{m-1}}k_{2n}^{1+q+\dots+q^{m-3}}}{\theta^{1+q+\dots+q^{m-2}}}+\dots+\\
&\ \ \ \ \ \ \ \  \frac{k_{m-2}^{q^2}\mathfrak{c}_{(2n)i}^{q^3}k_{2n}^{1+q}}{\theta^{1+q+q^2}}+\frac{k_{m-1}^q\mathfrak{c}_{(2n)i}^{q^2}k_{2n}}{\theta^{1+q}}+ \frac{k_m\mathfrak{c}_{(2n+1)i}^q}{\theta}\bigg)
\end{align*}
which proves our claim. Note that the right hand side of \eqref{E:claim} is the $(i,m)$-entry of $\mathfrak{B}$. This immediately implies that $\mathfrak{B}=B^{-1}$ and hence we have
$
\beta=\Id_{2n}.
$
%Furthermore, by \eqref{E:twsit}, we obtain
%$
%\alpha=\sum_{i=1}^{2n}\mathfrak{d}^{\phi}_i\otimes \mathfrak{c}^{\phi}_i.
%$

\subsubsection{Odd rank case}\label{SS:Odd rank case} Let us set $r=2n+1$ for some positive integer $n$. Using \eqref{E:det}, \eqref{E:invofB}, \eqref{E:torsionid} and the definition of $\beta$ given in \eqref{E:beta1}, we see that
\begin{multline*}
\beta=\frac{1}{\mathfrak{D}}\begin{pmatrix}
\frac{\mathfrak{c}_{(2n+1)1}}{k_{2n+1}^{(-1)}} & \mathfrak{c}_{(2n+1)1}^q & \dots & \mathfrak{c}_{(2n+1)1}^{q^{2n}}\\
\frac{\mathfrak{c}_{(2n+1)2}}{k_{2n+1}^{(-1)}} & \mathfrak{c}_{(2n+1)2}^q & \dots & \mathfrak{c}_{(2n+1)2}^{q^{2n}}\\
\vdots &\vdots &  & \vdots \\
\frac{\mathfrak{c}_{(2n+1)(2n+1)}}{k_{2n+1}^{(-1)}} & \mathfrak{c}_{(2n+1)(2n+1)}^q & \dots & \mathfrak{c}_{(2n+1)(2n+1)}^{q^{2n}}
\end{pmatrix}\\ 
\times \begin{pmatrix}
0& \dots & \dots & \dots & 0 &k_{2n+1}^{(-1)}\\
1& \frac{k_1}{\theta} & \frac{k_2}{\theta}  & \dots & \frac{k_{2n-1}}{\theta} &0\\
& -\frac{k_{2n+1}}{\theta}  & -\frac{k_{2n+1}k_1^q}{\theta^{1+q}}  & & -\frac{k_{2n+1}k_{2n-2}}{\theta^{1+q}} &\vdots\\
&  & \ddots  &\ddots  & \vdots &\vdots\\
&  &   & \frac{k_{2n+1}^{1+\dots+q^{2n-3}}}{\theta^{1+\dots+q^{2n-3}}} &\frac{k_{2n+1}^{1+\dots+q^{2n-3}}k_{1}^{q^{2n-3}}}{\theta^{1+\dots+q^{2n-2}}} &\vdots\\
&  &   & &-\frac{k_{2n+1}^{1+\dots+q^{2n-2}}}{\theta^{1+\dots+q^{2n-2}}}  & 0
\end{pmatrix}B.
\end{multline*}
Consider 
$
\mathfrak{C}:=\beta B^{-1}.
$
Our goal is to show that $\mathfrak{C}=B^{-1}$. Firstly, by 
\eqref{E:cof} and a simple calculation, the first and last column of $\mathfrak{C}$ and $B^{-1}$ are equal. On the other hand, similar to \eqref{E:inv2}, observe that, by \eqref{E:inverse},  \eqref{E:cof} and \eqref{E:det}, we have
\begin{multline}\label{E:inv222}
\frac{k_{2n+1}}{\theta\mathfrak{D}}\begin{pmatrix}
-\mathfrak{c}_{(2n+1)1}^{q^2} & -\mathfrak{c}_{21}^q & \dots &-\mathfrak{c}_{(2n)1}^q& -\mathfrak{c}_{(2n+1)1}^q\\
-\mathfrak{c}_{(2n+1)2}^{q^2} & -\mathfrak{c}_{22} & \dots & -\mathfrak{c}_{(2n)2}^q&-\mathfrak{c}_{(2n+1)2}^q\\
\vdots &\vdots &  & \vdots & \vdots \\
-\mathfrak{c}_{(2n+1)(2n+1)}^{q^2} & -\mathfrak{c}_{2(2n+1)}^q & \dots & -\mathfrak{c}_{(2n)(2n+1)}^q& -\mathfrak{c}_{(2n+1)(2n+1)}^q
\end{pmatrix}\times \\
\begin{pmatrix}
0&1 & & & \\
& & \ddots& & \\
& &  &\ddots & \\
& & &  &1 \\
-\frac{\theta}{k_{2n+1}}&-\frac{k_1}{k_{2n+1}}&\dots&\dots&-\frac{k_{2n}}{k_{2n+1}}
\end{pmatrix}
=\frac{1}{\mathfrak{D}}\begin{pmatrix}
\mathfrak{c}_{(2n+1)1}^q & \mathfrak{c}_{21} & \dots &\mathfrak{c}_{(2n)1}& \mathfrak{c}_{(2n+1)1}\\
\mathfrak{c}_{(2n+1)2}^q & \mathfrak{c}_{22} & \dots & \mathfrak{c}_{(2n)2}&\mathfrak{c}_{(2n+1)2}\\
\vdots &\vdots &  & \vdots & \vdots \\
\mathfrak{c}_{(2n+1)(2n)}^q & \mathfrak{c}_{2(2n+1)} & \dots & \mathfrak{c}_{(2n)(2n+1)}& \mathfrak{c}_{(2n+1)(2n+1)}
\end{pmatrix}.
\end{multline}

For each $2\leq m \leq 2n$ and $1\leq i \leq 2n+1$, we claim  that 
\begin{multline}\label{E:claim2}
\mathfrak{c}_{mi}=\frac{(-1)^{m-1}\mathfrak{c}_{(2n+1)i}^{q^m}k_{2n+1}^{1+q+\dots+q^{m-2}}}{\theta^{1+q+\dots+q^{m-2}}}+\frac{(-1)^{m}k_1^{q^{m-2}}\mathfrak{c}_{(2n+1)i}^{q^{m-1}}k_{2n+1}^{1+q+\dots+q^{m-3}}}{\theta^{1+q+\dots+q^{m-2}}}\\
+\frac{(-1)^{m+1}k_2^{q^{m-3}}\mathfrak{c}_{(2n+1)i}^{q^{m-2}}k_{2n+1}^{1+q+\dots+q^{m-4}}}{\theta^{1+q+\dots+q^{m-3}}}+\dots+
\frac{(-1)^{2m-4}k_{m-3}^{q^{2}}\mathfrak{c}_{(2n+1)i}^{q^{3}}k_{2n+1}^{1+q}}{\theta^{1+q+q^2}}\\+\frac{(-1)^{2m-3}k_{m-2}^q\mathfrak{c}_{(2n+1)i}^{q^2}k_{2n+1}}{\theta^{1+q}}+\frac{(-1)^{2m-2}k_{m-1}\mathfrak{c}_{(2n+1)i}^q}{\theta}.
\end{multline}
When $m=2$, we have 
\[
\mathfrak{c}_{2i}=-\frac{\mathfrak{c}_{(2n+1)i}^{q^2}k_{2n+1}}{\theta}+\frac{k_1\mathfrak{c}_{(2n+1)i}^{q}}{\theta}.
\]
Assume that it holds for $m$. Note, by \eqref{E:inv222}, we have 
\[
\mathfrak{c}_{(m+1)i}= -\frac{\mathfrak{c}_{mi}^qk_{2n+1}}{\theta}+\frac{k_m\mathfrak{c}_{(2n+1)i}^q}{\theta}.
\]
By the induction hypothesis, we have
\begin{align*}
\mathfrak{c}_{(m+1)i}&=-\frac{\mathfrak{c}_{mi}}{\theta}+\frac{k_m\mathfrak{c}_{(2n+1)i}^q}{\theta}\\
&=\frac{(-1)^{m}\mathfrak{c}_{(2n+1)i}^{q^{m+1}}k_{2n+1}^{1+q+\dots+q^{m-1}}}{\theta^{1+q+\dots+q^{m-1}}}+\frac{(-1)^{m+1}k_1^{q^{m-1}}\mathfrak{c}_{(2n+1)i}^{q^{m}}k_{2n+1}^{1+q+\dots+q^{m-2}}}{\theta^{1+q+\dots+q^{m-1}}}\\
&\ \ \ \ +\frac{(-1)^{m+2}k_2^{q^{m-2}}\mathfrak{c}_{(2n+1)i}^{q^{m-1}}k_{2n+1}^{1+q+\dots+q^{m-3}}}{\theta^{1+q+\dots+q^{m-2}}}+\dots+ \frac{(-1)^{2m-3}k_{m-3}^{q^{3}}\mathfrak{c}_{(2n+1)i}^{q^{4}}k_{2n+1}^{1+q+q^2}}{\theta^{1+q+q^3}}\\
&\ \ \ \ +\frac{(-1)^{2m-2}k_{m-2}^{q^2}\mathfrak{c}_{(2n+1)i}^{q^3}k_{2n+1}^{1+q}}{\theta^{1+q+q^2}}+\frac{(-1)^{2m-1}k_{m-1}^q\mathfrak{c}_{(2n+1)i}^{q^2}k_{2n+1}}{\theta^{1+q}}+\frac{(-1)^{2m}k_m\mathfrak{c}_{(2n+1)i}^q}{\theta}
\end{align*}
which proves our claim. It is easy to see that the right hand side of \eqref{E:claim2} is the $(i,m)$-entry of $\mathfrak{C}$. Therefore $\mathfrak{C}=B^{-1}$ and thus we have
$
\beta=\Id_{2n+1}.
$
%Furthermore \eqref{E:twsit} now yields
%$
%\alpha=\sum_{i=1}^{2n+1}\mathfrak{c}^{\phi}_i\otimes \mathfrak{d}^{\phi}_i
%$
%as desired.

\subsection{Formulas for the logarithms} Recall the identities given in \eqref{E:sigmaeq} and \eqref{E:tau}. Observe, by  \eqref{E:sigmatau} and \eqref{E:Gnexpression}, that, for $n\geq r-1$, we have
\begin{equation}\label{E:Gn tensor factorization Drinfeld2}
\begin{split}
((1\otimes t) - (t\otimes 1))G_n^\otimes+\gamma &=\sum_{\ell=0}^{r-1}\sigma^{\ell-n}(\mathfrak{d}^{\phi}_1)  \otimes \tau^{n+1}(\mathfrak{c}^{\phi}_{\ell+1})\\
&=\sum_{\ell=0}^{r-1}\mathcal{P}_{n-\ell}\mathfrak{e}_1\otimes \mathfrak{e}_{\ell+1}^{\tr}\mathcal{S}_{n}\\
&=V^{\tr}\Pi_{n-r}^{(1)}(\mathfrak{U}^{-1})^{(n-(r-1))}\mathfrak{e}_1\otimes \mathfrak{e}_r^{\tr} \mathfrak{U}^{(n)}(\Psi_n^{(-1)})^{\tr}\\
& \ \ \ \ \ \ \ \ +V^{\tr}\Pi_{n-(r-1)}^{(1)}(\mathfrak{U}^{-1})^{(n-(r-2))}\mathfrak{e}_1\otimes \mathfrak{e}_{r-1}^{\tr} \mathfrak{U}^{(n)}(\Psi_n^{(-1)})^{\tr}+\cdots\\
& \ \ \ \ \ \ \ \ + V^{\tr}\Pi_{n-1}^{(1)}(\mathfrak{U}^{-1})^{(n)}\mathfrak{e}_1\otimes \mathfrak{e}_1^{\tr}\mathfrak{U}^{(n)} (\Psi_n^{(-1)})^{\tr}\\
&=V^{\tr}\Pi_{n-r}^{(1)}\alpha_n (\Psi_n^{(-1)})^{\tr}
\end{split}
\end{equation}
where the last equality follows from 
$
\Pi_{n-\ell}^{(1)}=\Pi_{n-r}^{(1)}F^{(n-(r-2))}\cdots F^{(n-(\ell-1))}.
$ We remind the reader that the condition $n\geq r-1$ is necessary to obtain the second equality above.

Recall the fundamental periods $\lambda_1,\dots,\lambda_r\in \mathbb{C}_{\infty}^{\times}$ of $\phi$ defined at beginning of the present section and consider the matrix $\Psi\in \GL_r(\mathbb{T})$ introduced in \S\ref{SS:Dual t-motives for Drinfeld} which is constructed by using  $\{\lambda_1,\dots,\lambda_r\}$ forming an $A$-basis for $\Ker(\exp_{\phi})$. 

We are now ready to prove the main result of this section. Recall the map $\mathcal{M}$ and $\mathcal{M}_{\zz}$ defined in \S\ref{SS:intro1.2}.
\begin{theorem}\label{T:logdrinfeld} Let $\opi = (\lambda_1,\dots,\lambda_r)$ be a vector of fundamental periods of $\phi$. We have
\[
\log_{\phi}=\mathcal{M}\left(-\frac{1}{t-\theta}\opi(\Psi^{\tr})^{(-1)}\right).
\]
Moreover, for any $\zz\in \mathbb{C}_{\infty}$ in the domain of convergence of $\log_{\phi}$, we have
\[
	\log_{\phi}(\zz)=\mathcal{M}_{\zz}\left(-\frac{1}{t-\theta}\opi(\Psi^{\tr})^{(-1)}\right).
	\]
\end{theorem}
\begin{proof} 
Observe that for each $0\leq \ell \leq r-1$, if we set $\sigma^{\ell-n}(\mathfrak{d}^{\phi}_1)=[\zeta_{(\ell+1)1,n},\dots, \zeta_{(\ell+1)r,n}]^{\tr}\in \mathbb{C}_{\infty}(t)^{r}$, then $\delta_0^{N_{\phi}}(\sigma^{\ell-n}(\mathfrak{d}_1^{\phi}))=\zeta_{(\ell+1)1,n}(\theta)$. On the other hand, we have 
\begin{equation}\label{E:proofDr1}
\begin{split}
    \varphi\left(\lim_{n\to \infty}\sum_{i=0}^{n}\delta_0^{N_{\phi}}(\sigma^{-i}(\mathfrak{d}_1^{\phi}))\tau^i(m_1)\right)&=\varphi\left(\lim_{n\to \infty}\sum_{i=0}^{n}\beta_i\tau^i(m_1)\right)\\    
    &=\lim_{n\to \infty}\tilde{\varphi}\left(\sum_{i=0}^n\beta_i\tau^i(m_1)\right)\\
    &=\lim_{n\to \infty}\frac{1}{t-\theta}\tilde{\varphi}\left(\sum_{\ell=0}^{r-1}\beta_{n-\ell}\tau^{n+1}(\mathfrak{c}_{\ell+1})\right)\\
    &=\lim_{n\to \infty}\frac{1}{t-\theta}\mathfrak{e}_1^{\tr}V^{\tr}(\Pi_{n-r}^{(1)})_{|t=\theta}\alpha_n(\theta) (\Psi_n^{(-1)})^{\tr}\\
    &=\frac{1}{t-\theta}\mathfrak{e}_1^{\tr}V^{\tr}\Upsilon^{(1)}|_{t=\theta}(\Psi^{\tr})^{(-1)}\in \Mat_{1\times r}(\mathbb{T}).
    \end{split}
    \end{equation}
Here the first equality follows from the fact that $\delta_0^{N_{\phi}}(\sigma^{-i}(\mathfrak{d}_1^{\phi}))=\beta_i$ for each $i\geq 0$ (\cite[Cor. 4.5]{Gre22}). For the second equality we note that 
\begin{equation}\label{E:proofDr2}
\lim_{n\to \infty}\sum_{i=0}^{n}\beta_i\tau^i(m_1)=\lim_{n\to \infty} \sum_{i=0}^{n}\beta_ik_r^{-1}\tau^i(\mathfrak{c}_r)=\lim_{n\to \infty}\sum_{i=0}^{n}\beta_i\tau^i=\sum_{i=0}^{\infty}\beta_i\tau^i=\log_{\phi}\in \mathbb{M},
\end{equation}
where we used the identification between $k_r^{-1}\tau^{i}(\mathfrak{c}_r)$ and $\tau^i$. The third equality follows from Lemma \ref{L:logcoef} and the fact that $\tilde{\varphi}$ is a $\mathbb{C}_{\infty}[t]$-linear map, the fourth equality follows from applying $\delta_{0}^{N_{\phi}}\otimes 1$ to \eqref{E:Gn tensor factorization Drinfeld2} as well as the structure of $\delta_{0}^{N_{\phi}}$-map described above and finally the last equality follows from Theorem \ref{T:product}, \eqref{E:rigiddual} and Corollary \ref{C:alpha}.

On the other hand, observe that $\mathfrak{e}_1^{\tr} V^{\tr} = (k_1,\dots,k_r)$. Let $f_i=s_{\phi}(\lambda_i;t)$ be as introduced in \S \ref{SS:t-motive for Drinfeld}. Then, by Proposition \ref{P:AGF}, we find that for $1\leq i\leq r$, the $i$-th entry of $\mathfrak{e}_1^{\tr} V^{\tr} \Upsilon\twist$ is given by
\[k_1f_i\twist + k_2 f_i\twistk{2} + \dots + k_r f_i\twistk{r} = (t-\theta)f_i.\]
Therefore evaluating this at $t=\theta$ gives $\Res_\theta f_i$ which equals $-\lambda_i$ by \cite[(3.4.3)]{CP12}. Putting this all together gives
\begin{equation}\label{E:proofDr3}
\mathfrak{e}_1^{\tr} V^{\tr} \Upsilon\twist_{t=\theta} = -\opi,
\end{equation}

Sine, by Proposition \ref{P:inj}, the map $\varphi$ is injective, combining \eqref{E:proofDr1}, \eqref{E:proofDr2} and \eqref{E:proofDr3}, we obtain 
\begin{multline*}
\log_{\phi}=\varphi^{-1}\left( \frac{1}{t-\theta}\mathfrak{e}_1^{\tr}V^{\tr}\Upsilon^{(1)}|_{t=\theta}(\Psi^{\tr})^{(-1)} \right)=\varphi^{-1}\left(-\frac{1}{t-\theta}\opi(\Psi^{\tr})^{(-1)}\right)\\
=\mathcal{M}\left(-\frac{1}{t-\theta}\opi(\Psi^{\tr})^{(-1)}\right).
\end{multline*}
Finally, by Theorem \ref{T:log} and the first assertion, we obtain
\[
\log_{\phi}=\mathcal{M}_{\zz}\left(-\frac{1}{t-\theta}\opi(\Psi^{\tr})^{(-1)}\right)
\]
as desired.
\end{proof}

By specializing the value $\bz$ at certain prescribed points, we may conclude that the left-hand side of Theorem \ref{T:logdrinfeld} evaluates to a Taelman $L$-value (see \cite{Tae12} for more details). We further evaluate terms on the right-hand side to show that it includes periods and exponential functions which also indicates that our next result may be interpreted as a Mellin transform formula for Taelman $L$-values.

\begin{corollary}\label{C:Mellin for Drinfeld modules}
Let $\phi$ be a Drinfeld module as in Theorem \ref{T:Main theorem intro Drinfeld case} so that each $k_i\in \mathbb{F}_q^{\times}$. Then, letting $\zz=1$, we have
\[L(\phi^{\vee},0) =\mathcal{M}_{\zz}\left(-\frac{1}{t-\theta}\opi(\Psi^{\tr})^{(-1)}\right).\]
\end{corollary}

\begin{proof} Observe that we have $K_{\infty}=\mathfrak{M} \oplus A$. Since $\log_{\phi}$ converges at any element $z\in \mathbb{C}_{\infty}$ satisfying $|z|<q^{q^r/(q^r-1)}$, $\log_{\phi}(1)$ is well-defined and $\mathfrak{M}$ is in the domain of convergence of $\log_{\phi}$. Moreover,  by \cite[Thm. 3.3]{EGP13}, one can calculate the logarithm coefficients of $\phi$, which yields the fact that $\log_{\phi}(\mathfrak{M})\subseteq \mathfrak{M}$. 
	
To proceed, we define the $A$-module $H(\phi/A)$ given by the quotient
	\[
	H(\phi/A):=\frac{\phi(K_{\infty})}{\exp_{\phi}(K_{\infty})+\phi(A)}.
	\]
	Here, by $\phi(K_{\infty})$ and $\phi(A)$, we mean the $A$-modules $K_{\infty}$ and $A$ equipped with the $A$-module structure induced from $\phi$. Since $\exp_{\phi}$ is the formal inverse of $\log_{\phi}$, we now see that $\exp_{\phi}(K_{\infty})\supseteq \mathfrak{M}$. Thus, $\exp_{\phi}(K_{\infty})+\phi(A)\supseteq \phi(K_{\infty})$, implying that $H(\phi/A)$ is trivial. On the other hand, if we set $U(\phi/A):=\{ u\in K_{\infty} \ \ | \exp_{\phi}(u)\in A \}$, by \cite[Thm. 1.10]{Fan15}, we know that $U(\phi/A)$ is an $A$-module of rank one. Indeed, since the norm of $\log_{\phi}(1)$, being equal to 1, is minimal among the elements of $U(\phi/A)$, we obtain that $U(\phi/A)=A\log_{\phi}(1)$. Thus, by \cite[Rem. 5, Thm. 1]{Tae12} (see also \cite[\S3]{ChangEl-GuindyPapanikolas}), we obtain $L(\phi^{\vee},0) =\log_{\phi}(1)$.
The result then follows from Theorem \ref{T:logdrinfeld}.
\end{proof}

\begin{remark}\label{R:Log* and Drinfeld Modular Forms}
At the present, we do not know if our formulas provide a connection between Drinfeld modular forms and $L$-series. However, there are some hints in this direction provided by the case of the Carlitz module. In this setting, for $\zz=1$, our formulas give
\[\mathcal{M}_{\zz}(-\tpi \Omega) = \zeta_A(1).\]
In seeking to connect the LHS of this formula with a Drinfeld modular form, we are inspired to write $\Omega$ in terms of the commonly used Drinfeld modular form uniformizer, $u(z) := 1/\exp_C(\tpi z)$. We then write
\[1/\Omega\twistinv = \omega_C = \exp_C\left (\frac{\tpi}{\theta-t}\right ) = \sum_{i=0}^\infty \exp_C\left (\frac{\tpi}{\theta^{i+1}}\right )t^i.\]
Thus the reciprocal of $\Omega\twistinv$ can be written as a sum of $u(z)$ evaluated at certain powers of $\theta$. This construction is somewhat forced, and seems unlikely to lead to a meaningful connection with Drinfeld modular forms in our opinion. More natural is to do the following. Recall the adjoint of the Carlitz module, $C_\theta^*(z) = \theta z + z^{1/q}$ (see \cite[\S 3.7]{Gos96}). It comes equipped with an exponential function $\exp_C^*(z)$ which satisfies
\[\theta \exp_C^*(z) = \exp_C^*(C_\theta^*(z)).\]
Formally, $C^*$ also has a logarithm series $\log_C^*$, which is the formal (fractional) power series inverse of $\exp_C^*$, which satisfies
\[C_t^*(\log_C^*(z)) = \log_C^*(\theta z).\]
However, this construction produces a power series with 0 radius of convergence! If we had a way to rigorously construct the function $\log_C^*$, it should produce a function with a free rank 1 period generated by an element $\pi^*$, and we would use this to define
\[g(t) = \log_C^*\left (\frac{\pi^*}{\theta - t}\right ),\]
and we would have that both $g(t)$ and $\Omega$ satisfy 
\[t g(t) = C_{\theta}^*(g(t)),\quad t\Omega = C_{\theta}^*(\Omega).\]
Thus the two functions are equal up to normalization. Finally, we use this identification to rewrite our main theorem
\[\mathcal{M}_{\zz}(-\tpi \Omega) = \mathcal{M}_{\zz}\left(-\tpi \sum_{i=0}^\infty \log_C^*(\pi^* \theta^{-i-1})t^i\right).\]

We anticipate that there seems to be a more natural connection between the logarithm function of the adjoint Carlitz module $\log_C^*$ (see \cite[\S 3.7]{Gos96}) and Drinfeld modular forms. However, we are unsure how to make this connection rigorous, so this is a topic for future study.

\end{remark}

\section{Logarithms of tensor product of Drinfeld modules with the tensor powers of the Carlitz module}\label{S:Log for D otimes C}

Throughout this section, we fix a positive integer $k\geq 1$ and continue to assume that $\phi$ is a Drinfeld module given by
\[\phi_{\theta} = \theta + k_1 \tau + \dots + k_r\tau^r\in \mathbb{C}_{\infty}[\tau]\]
so that $\inorm{k_i}\leq 1$ for each $1\leq i \leq r-1$ and $k_r\in \mathbb{F}_q^{\times}$. We also remark that, throughout this section, our tensor products, except those used to denote tensor products of Drinfeld modules and tensor powers of Carlitz module, are still considered over $\mathbb{C}_{\infty}$.

We examine the case where our Anderson $t$-module is chosen to be $\phi\otimes C^{\otimes k}=(\mathbb{G}_{a/\mathbb{C}_{\infty}}^{rk+1},\rho)$ detailed in Example \ref{Ex:1}(iii). From \S\ref{SS:t-motive for D tensor C} and \S\ref{SS:Dual t-motive for D tensor C}, recall the Anderson $t$-motive $ M_{\phi\otimes C^{\otimes k}}$ and the dual $t$-motive $N_{\phi\otimes C^{\otimes k}}$ attached to $\phi\otimes C^{\otimes k}$. To simplify the notation, in this section, we set $N_{\rho}:=N_{\phi\otimes C^{\otimes k}}$ and $M_{\rho}:=M_{\phi\otimes C^{\otimes k}}$.

From \S\ref{SS:t-motive for D tensor C} and \S\ref{SS:Dual t-motive for D tensor C} again, consider the $\mathbb{C}_{\infty}[t]$-basis $\{\mathfrak{c}_1,\dots,\mathfrak{c}_r\}$ and $\{\mathfrak{d}_1,\dots,\mathfrak{d}_r\}$ as well as the $\mathbb{C}_{\infty}[\tau]$-basis $\{g_1,\dots,g_{rk+1}\}$ and 
$\mathbb{C}_{\infty}[\sigma]$-basis $\{h_1,\dots,h_{rk+1}\}$ for $M_{\rho}$ and $N_{\rho}$ respectively. Let us set 
\[
\tilde{V}:=\begin{pmatrix}
0&k_2^{(-1)} &k_3^{(-2)}& \dots &k_r^{(1-r)}\\
\vdots &\vdots&\vdots&  \iddots & \\
\vdots & \vdots& k_r^{(-2)} & & \\
0&k_r^{(-1)} & & & \\
1 &  & & & 
\end{pmatrix}\in \GL_r(\mathbb{C}_{\infty}).
\]
Then we have 
\begin{equation}\label{E:tmot}
\tilde{V}\begin{bmatrix}
\mathfrak{d}_1\\
\vdots\\
\vdots\\
\mathfrak{d}_r
\end{bmatrix}=\begin{bmatrix}
h_{r(k-1)+2}\\
\vdots\\
\vdots\\
h_{rk+1}
\end{bmatrix}  \text{  and  } [g_1,\dots,g_{r-1},\mathfrak{c}_1]=[\mathfrak{c}_1,\dots, \mathfrak{c}_{r}](\tilde{V}^{(-1)})^{-1}.
\end{equation}

Next we consider the matrices $\mathcal{S}_n$ and $\mathcal{P}_n$ from \S\ref{SS:varphi map} and \S\ref{S:product} respectively. For the convenience of the reader, we precisely write 
\[
\mathcal{P}_n=(\mathcal{P}_{n}^{(-1)})^{(1)}=((\Phi^{\tr})^{-1}((\Phi^{\tr})^{-1})^{(1)}\cdots ((\Phi^{\tr})^{-1})^{(n-1)})^{(1)}=V^{\tr} \Pi_{n-1}^{(1)} (\mathfrak{U}^{-1})^{(n)}
\]
and
\[
\mathcal{S}_n=(\Phi^{\tr})^{(n)}(\Phi^{\tr})^{(n-1)}\cdots (\Phi^{\tr})=\mathfrak{U}^{(n)}\Pi_n^{-1}((V^{(-1)})^{-1})^{\tr}=\mathfrak{U}^{(n)}(\Psi_{n}^{(-1)})^{\tr}.
\]
 Recall the polynomials $\mathbb{S}_n\in A[t]$ defined in Example \ref{Ex:11}. For each $k\geq 1$, we further set 
\[
\widetilde{\mathcal{P}}^k_n:=(\mathbb{S}_{n-1}^{(1)})^{-k}\mathcal{P}_n \text{ and }\widetilde{\mathcal{S}}^k_n:=\mathbb{S}_n^{k}\mathcal{S}_n.\]
Recall the definition of $\tilde{N}_{\rho}$ from \S\ref{SS:Log of And t-modules}. Thus, for any $\tilde{n}\in \tilde{N}_{\rho}$ ($\tilde{m}\in M_{\rho}$ respectively) given by $\tilde{n}=\sum_{i=1}^ra_i \mathfrak{d}_i$ ($\tilde{m}=\sum_{i=1}^rb_i \mathfrak{c}_i$ respectively), using \eqref{E:tauactiontens}, \eqref{E:sigmatensor} and \eqref{E:sigmainv}, we have 
\begin{equation}\label{E:sigmatautensor}
\sigma^{-n}(\tilde{n})=\widetilde{\mathcal{P}}^k_n\begin{bmatrix}
a_1\\
\vdots\\
a_r
\end{bmatrix}^{(n)} \text{ and } \ \tau^{n}(\tilde{m})=[b_1,\dots,b_r]^{(n)} \widetilde{\mathcal{S}}^k_{n-1}.
\end{equation}

\subsection{The structure of $\delta_0^{N_{\rho}}$-map}\label{SS:delta_0 map D otimes C}  In what follows, we analyze the behavior of $\delta_0^{N_{\rho}}$. In particular, we define an explicit isomorphism of $\mathbb{C}_{\infty}[t,\sigma]$-modules which allows us to compute the values of the map $\delta_0^{N_{\rho}}$. For more details on such construction, we refer the reader to \cite[\S4,6]{GN21}.

Consider the $\mathbb{C}_{\infty}[t,\sigma]$-module $\mathcal{N}:=\Mat_{1\times (rk+1)}(\mathbb{C}_{\infty}[\sigma])$ whose $\mathbb{C}_{\infty}[t]$-module structure is given by 
\[
ct^i\cdot \mathfrak{n}:=c\mathfrak{n} \rho_{\theta^i}^{*}, \ \ c\in \mathbb{C}_{\infty},\ \ \mathfrak{n} \in \mathcal{N}.
\]
For any $1\leq i \leq rk+1$, let $\mathfrak{f}_i\in \Mat_{1\times (rk+1)}(\mathbb{F}_q)$  be the $i$-th unit vector. For any $1\leq i \leq r$, we set $\mathfrak{n}_i:=\mathfrak{f}_{r(k-1)+i+1}\in \mathcal{N}$. Note, as it is already observed in \cite[(45)]{GN21}, that we have $(t-\theta)^{k}\mathfrak{n}_r=\mathfrak{f}_1$ and for $1\leq \mu \leq k$, one obtains $(t-\theta)^{k-\mu}\mathfrak{n}_i=\mathfrak{f}_{r(\mu-1)+i+1}$. Furthermore, a direct calculation implies that the set $\{\mathfrak{n}_1,\dots,\mathfrak{n}_r\}$ forms a $\mathbb{C}_{\infty}[t]$-basis for $\mathcal{N}$. 

There exists a $\mathbb{C}_{\infty}[t,\sigma]$-module isomorphism $\iota:N_{\rho}\to \mathcal{N}$ given by 
\[
\iota\left(\sum_{j=1}^r\mathfrak{r}_jh_{r(k-1)+j+1}\right):=\mathfrak{r}_1\cdot \mathfrak{n}_1+\cdots +\mathfrak{r}_r\cdot \mathfrak{n}_r, \ \ \mathfrak{r}_1,\dots, \mathfrak{r}_r\in \mathbb{C}_{\infty}[t].
\]
We further define certain elements $v_{ij}\in \mathbb{C}_{\infty}$ so that 
\[
\tilde{V}^{-1}=\begin{pmatrix}
&  & &  &v_{1r}\\
& & &  v_{2(r-1)} & 0\\ 
&  & \iddots & & \vdots\\
& \iddots & & & \vdots \\
v_{r1} &\cdots  & \cdots & v_{r(r-1)} & 0
\end{pmatrix}.
\]
This implies, by \eqref{E:tmot}, that $\iota(\mathfrak{d}_1)=v_{1r}\mathfrak{n}_{r}$ and for $2\leq \ell \leq r$, we have 
\[
\iota(\mathfrak{d}_{\ell})=v_{\ell (r-\ell+1)}\mathfrak{n}_{r-\ell+1}+\cdots +v_{\ell (r-1)}\mathfrak{n}_{r-1}.
\]
Thus, by the definition of $\delta_{0}^{N_\rho}$, if $n=\sum_{j=1}^r\left(\sum_{\ell=0}^{m_{j}}a_{j,\ell}(t-\theta)^{\ell}\right)\mathfrak{d}_j\in N_\rho$, then 
\begin{equation}\label{E:deltatensor}
\delta_{0}^{N_{\rho}}(n)=\begin{pmatrix}
*\\
\vdots \\
*\\
a_{r0}v_{r1}\\
\sum_{j=r-1}^ra_{j0}v_{j2}\\
\vdots \\
\sum_{j=2}^ra_{j0}v_{j(r-1)}\\
a_{10}v_{1r}
\end{pmatrix}.
\end{equation}
Since $(t-\theta)^{k+1}N_{\rho}\subset \sigma N_\rho$, the map $\delta_{0}^{N_\rho}$ may be calculated similarly at $\sigma^{-\ell}(n)$ for any non-negative integer $\ell$.

\subsection{An analysis on elements in $\mathcal{D}_{\phi\otimes C^{\otimes k}}$} In this subsection, our goal is to introduce a bound on the entries of $\zz\in \mathbb{C}_{\infty}^{rk+1}$ so that $\zz$ lies in the domain of convergence $\mathcal{D}_{\phi\otimes C^{\otimes k}}$ of $ \Log_{\phi\otimes C^{\otimes k}}$. Our main result Proposition \ref{P:02} in this subsection may be compared with the analysis of Anderson and Thakur on the logarithm function of $C^{\otimes k}$ \cite[Prop. 2.4.3]{AT90}. 

Recall the elements $\mathcal{B}_n(t)\in \mathbb{C}_{\infty}(t)$ defined in \S\ref{S:product}. For $n\geq 1$, we further let
\[
R_n:=(\Phi^{-1})^{(n)}\cdots (\Phi^{-1})^{(1)} \in \GL_r(\mathbb{C}_{\infty}(t)).
\]

\begin{lemma} [{cf. \cite[Prop. 5.2.27] {Hua23}}] \label{L:01} Let $R_{n}=(\alpha^{[n]}_{j i})$. Then for each $1\leq j \leq r$, we have
\begin{itemize}
	\item[(i)] $ \alpha^{[n]}_{j 1}=\mathcal{B}_{n-(j-1)}(t)$.
	\item[(ii)] $\alpha^{[n]}_{j  r}=\frac{\mathcal{B}_{n-j}^{(1)}(t)}{t-\theta^q}$.
	\item[(iii)] For each $1\leq m \leq r-2$, we have
	\[
	\alpha^{[n]}_{j (r-m)}=\frac{\mathcal{B}_{n-m-j}^{(m+1)}(t)}{t-\theta^{q^{m+1}}}+\sum_{u=1}^{m}\frac{\mathcal{B}^{(u)}_{(n-1)-(u-1)-(j-1)}(t)}{t-\theta^{q^u}}k_{r-m+u-1}^{(u-r+1)}.
	\]
	\item[(iv)] Recall from \S\ref{S:product} that $\mathcal{P}_n=(R_n)^{\tr}=((\Phi^{-1})^{\tr})^{(1)}\dots ((\Phi^{-1})^{\tr})^{(n)}=(\alpha^{[n]}_{i j}) \in \GL_r(\mathbb{C}_{\infty}(t))$. Then we have 
    \[
    ||(\alpha^{[n]}_{1 j},\dots,  \alpha^{[n]}_{r j })^{\tr}||\leq ||\mathcal{B}_{n-(j-1)}(t)|| + C(q,r)  \]
    for some constant $C(q,r)$ depending only on $q$ and $r$.
\end{itemize}
\end{lemma}

\begin{proof} The first part follows from \cite[Lem. 3.1.4]{Che24}. Observe that 
	\begin{equation}\label{E:prod}
	R_{n+1}=R_{n}^{(1)}(\Phi^{-1})^{(1)}=\begin{pmatrix}
	\mathcal{B}_{n}(t)&\alpha_{12}^{[n]}&\cdots& \cdots &\alpha_{1r}^{[n]}\\
	\mathcal{B}_{n-1}(t)&\vdots & & &\vdots \\
	\vdots&\vdots & & &\vdots \\
	\vdots&\vdots &  & &\vdots \\
	\mathcal{B}_{n-(r-1)}(t)&\alpha_{r2}^{[n]}&\cdots &\cdots & \alpha_{rr}^{[n]}
	\end{pmatrix}^{(1)}\begin{pmatrix}
	\frac{k_1}{t-\theta^q}&\frac{k_2^{(-1)}}{t-\theta^q}&\cdots& \frac{k_{r-1}^{(-(r-2))}}{t-\theta^q}&\frac{1}{t-\theta^q}\\
	1& & & & \\
	 &\ddots & & & \\
	 & & \ddots & & \\
	 & & & 1& 0
	\end{pmatrix}.
	\end{equation}
	Note that part (ii) easily follows from the first part and \eqref{E:prod}. On the other hand, part (iii) follows from the recursive use of the equality
    \[
    \alpha_{j(r-m)}^{[n+1]}=\frac{\mathcal{B}^{(1)}_{n-(j-1)}(t)}{t-\theta^q}k_{r-m}^{(-(r-m)+1)}+(\alpha_{j(r-(m-1))}^{[n]})^{(1)}
    \]
	which indeed follows from part (i) and \eqref{E:prod}. Finally, the last assertion follows from part (i--iii) as well as the fact that 
    \[
    \log_q(||\mathcal{B}_{n}(t)||)\leq -\frac{q^{n+r}-q^r}{q^r-1}    \]
    which is a consequence of our conditions on $k_1,\dots,k_r$ combined with \cite[Prop. 6.9]{EGP14}.
\end{proof}

Let $\Log_{\phi\otimes C^{\otimes k}}=\sum_{n\geq 0}P_n\tau^n$ and note that $P_0=\Id_{rk+1}$. Our next goal is to analyze the norm of the certain entries of $P_n$.  For each $1\leq \ell \leq rk+1$, consider $\mathfrak{b}_\ell^{[n]}=[b_{1,\ell}^{[n]},\dots,b_{r,\ell}^{[n]}]^{\tr}\in \Mat_{r\times 1}(\mathbb{C}_{\infty})$ such that $\mathfrak{b}_\ell^{[n]}$ consists of the last $r$ entry of the $\ell$-th column of $P_n$. More precisely, for $n\geq 0$, we have 
\[
P_n=\begin{pmatrix}
*&\cdots & \cdots & *\\
\vdots & & & \vdots\\
*&\cdots &\cdots  &  *\\
b^{[n]}_{1,1}&\cdots &\cdots & b^{[n]}_{1,rk+1}\\
\vdots&  & & \vdots\\
b^{[n]}_{r,1}&\cdots &\cdots &  b^{[n]}_{r,rk+1}
\end{pmatrix}.
\]

\begin{proposition}[{cf. \cite[Prop. 2.4.3]{AT90}}]\label{P:02} Let $0\leq u \leq k-1$ and $1\leq j \leq r$. The following statements hold.
\begin{itemize}
 \item[(i)] We have 
    \[
\log_q||\mathfrak{b}^{[n]}_{ru+j+1}||\leq -q^n\left(u+1+\frac{q^j}{q^r-1}+\frac{k}{q-1}\right)+\frac{q^r}{q^r-1}+\frac{kq}{q-1}.
\]
   and
    \[
    \log_q||\mathfrak{b}_1^{[n]}||\leq -q^n\left(\frac{q^r}{q^r-1}+\frac{k}{q-1}\right)+\frac{q^r}{q^r-1}+\frac{kq}{q-1}.   
    \]
\item[(ii)] For any tuple $(z_1,\dots,z_{rk+1})\in \mathbb{C}_{\infty}^{rk+1}$ satisfying  
\[
\log_q(|z_1|)<\frac{q^r}{q^r-1}+\frac{k}{q-1} \text{ and } \log_q(|z_{ru+j+1}|)<u+1+\frac{q^j}{q^r-1}+\frac{k}{q-1},
\]
we have, for any $1\leq \ell \leq rk+1$,
$
\sum_{n=0}^{\infty}b_{j,\ell}^{[n]}z_{\ell}^{q^n}<\infty
$.
\end{itemize}
\end{proposition}
\begin{proof} We again note that \cite[Prop. 6.9]{EGP14} and our conditions on the coefficients of $\phi$ yield
\[
\log_q(|\beta_n|)\leq -\frac{q^{n+r}-q^r}{q^r-1}.
\]
On the other hand, by \eqref{E:sigmatautensor} and Lemma \ref{L:01}(iv), we have 
\[
\log_q||\sigma^{-n}(h_{ru+j+1})||\leq \log_q(||(t-\theta^{q^n})^{k-u-1}\mathcal{B}_{n-(r-j)}(t)(\mathbb{S}_{n-1}^{(1)})^{-k}||).
\]
Finally, since by \cite[Cor. 4.5]{Gre22}, $\delta_0^{N_{\rho}}(\sigma^{-n}(h_{ur+j+1}))$ is the $ru+j+1$-st column of $P_n$, by the structure of $\delta_0^{N_{\rho}}$-map described in \S\ref{SS:delta_0 map D otimes C} and Lemma \ref{L:01}, we see that 
\begin{align*}
\log_q(||\delta_0^{N_{\rho}}(\sigma^{-n}(h_{ru+j+1}))||)&\leq \log_q(|(\theta-\theta^{q^n})^{k-u-1}\beta_{n-(r-j)}((\mathbb{S}_{n-1}^{(1)})^{-k}))_{|t=\theta}|)\\
&\leq  (k-u-1)q^n-\frac{q^{n+j}-q^r}{q^r-1}-k\left(\frac{q^{n+1}-q}{q-1}\right)\\
&=q^n\frac{(k-u-1)(q^r-1)-q^j-kq(1+q+\dots+q^{r-1})}{q^r-1}\\
&\ \ \ \ \ +\frac{q^r+kq(1+q+\dots+q^{r-1})}{q^r-1}\\
&=-q^n\frac{(u+1)(q^r-1)+q^j+k(1+q+\dots+q^{r-1})}{q^r-1} \\
&\ \ \ \ \ +\frac{q^r+kq(1+q+\dots+q^{r-1})}{q^r-1}\\
&=-q^n\left(u+1+\frac{q^j}{q^r-1}+\frac{k}{q-1}\right)+\frac{q^r}{q^r-1}+\frac{kq}{q-1}.
\end{align*}
Similarly, for $\mathfrak{b}_1^{[n]}$, we obtain 
\begin{align*}
    \log_q(||\mathfrak{b}_1^{[n]}||)&\leq \log_q(||(\theta-\theta^{q^n})^{k}\beta_{n}((\mathbb{S}_{n-1}^{(1)})^{-k}))_{|t=\theta}|)\\
    &=\log_q(||\beta_{n}((\mathbb{S}_{n-2}^{(1)})^{-k}))_{|t=\theta}|)\\
    &\leq -\frac{q^{n+r}-q^r}{q^r-1}-k\left(\frac{q^n-q}{q-1}\right)\\
    &=-q^n\frac{q^r+k(1+q+\dots+q^{r-1})}{q^r-1}+\frac{q^r+kq(1+q+\dots+q^{r-1})}{q^r-1}\\
    &=-q^n\left(\frac{q^r}{q^r-1}+\frac{k}{q-1}\right)+\frac{q^r}{q^r-1}+\frac{kq}{q-1}
    \end{align*}
as desired. Finally, the last assertion is a consequence of part (i).
\end{proof}

\subsection{The map $\varphi_{\text{tens}}$} \label{SS:varphi map2} Recall the matrix $\Theta\in \GL_r(\mathbb{C}_{\infty}[t])$ from \S\ref{SS:t-motive for Drinfeld}. Set $\mathfrak{S}_{-1}:=\Id_r$, $\mathfrak{S}_0:=\Theta$ and for $n\geq 1$ consider $\mathfrak{S}_{n}:=\Theta^{(n)}\cdots \Theta^{(1)}\Theta $. For $1\leq j \leq r$, recall the $j$-th unit vector $\mathfrak{e}_j\in \Mat_{r\times 1}(\mathbb{F}_q)$ as well as the $\mathbb{C}_{\infty}[\tau]$-basis $\{g_1,\dots,g_{rk+1}\}$ for $M_{\rho}$ from \S\ref{SS:t-motive for D tensor C}. Note  that for $0\leq u \leq k-1$, $g_{ru+j}=(t-\theta)^{u}\mathfrak{m}_j$ and $g_{rk+1}=(t-\theta)^{k}\mathfrak{m}_1$ where each $\mathfrak{m}_j$ is as constructed in \S\ref{SS:t-motive for D tensor C}. In this subsection, analogous to \S\ref{SS:varphi map}, we extend the isomorphism $\widetilde{\varphi_{\text{tens}}}:M_{\rho}\cong \Mat_{1\times r}(\mathbb{C}_{\infty}[t])$ of $\mathbb{C}_{\infty}[t,\tau]$-modules given by  
\begin{multline*}
\widetilde{\varphi_{\text{tens}}}\left(\left(\sum_{n\geq 0}a_{1,n}\tau^n,\dots,\sum_{n\geq 0}a_{rk+1,n}\tau^n\right)\right):=\sum_{\substack{0\leq u \leq k-1\\1\leq j \leq r}}\sum_{n\geq 0}a_{ru+j,n}\mathfrak{e}_j^{\tr}\mathbb{S}_{n-1}^k(t-\theta^{q^n})^u\mathfrak{S}_{n-1}\\+\sum_{n\geq 0}a_{rk+1,n}\mathfrak{e}_1^{\tr}\mathbb{S}_{n}^k\mathfrak{S}_{n-1}, \ \ a_{ru+j,n},a_{rk+1,n}\in \mathbb{C}_{\infty}.
\end{multline*}

Our first goal is to analyze the norm of $
\widetilde{\varphi_{\text{tens}}}(\tau^n(g_{ru+j}))
$ and 
$
\widetilde{\varphi_{\text{tens}}}(\tau^n(g_{rk+1}))
$
for each $n\geq 0$ to establish a well-defined extension of the above isomorphism as well as to prove that it is injective in Proposition \ref{P:inj2}.  For each $n\geq 0$, we first define $\mathfrak{N}_{j,n}\in \Mat_{1\times r}(\mathbb{C}_{\infty}[t])$ so that 
\[
\mathfrak{S}_{n}=\begin{bmatrix}
    \mathfrak{N}_{1,n}\\
    \vdots\\
    \mathfrak{N}_{r,n}
\end{bmatrix}.
\]

\begin{lemma}\label{L:02} We have 
\[
\log_q\left(||\mathfrak{N}_{j,0}||\right)= \begin{cases} 0 & \text{ if } 1\leq j<r\\
1 &\text{ if } j=r.
\end{cases}
\]
Moreover, for $n\geq 1$, we have 
\[
\log_q\left(||\mathfrak{N}_{j,n}||\right)\leq \begin{cases} 0 & \text{ if } n<r-j\\
q^{n-(r-j)}+\log_q\left(||\mathbb{S}_{n-(2r-j)}||  \right) &\text{ if } n\geq r-j.
\end{cases}
\]
\end{lemma}
\begin{proof} Since $k_1,\dots,k_r$ have $|\cdot|$-norm less than or equal to one, the first part immediately follows.  We now prove the second part. Using the conditions on $k_1,\dots,k_{r}$ and a simple computation yield the lemma when $n<r-j$. On the other hand, one can also see that 
\[
\log_q\left(||\mathfrak{N}_{j,r-j}||\right)\leq 1=q^{(r-j)-(r-j)}
\]
and hence, the lemma holds for $n=r-j$. Assume that it holds for $n>r-j$. Note that
\begin{equation}\label{E:induct}
\mathfrak{S}_{n+1}= \begin{bmatrix}
    \mathfrak{N}_{1,n+1}\\
    \vdots\\
    \mathfrak{N}_{r-1,n+1}\\
    \mathfrak{N}_{r,n+1}
\end{bmatrix}=\Theta^{(n+1)} \mathfrak{S}_{n}= \begin{pmatrix}
&1& & &  \\
& & \ddots & & \\
& & & \ddots & \\
& & & &  1\\
t-\theta^{q^{n+1}}&-k_1^{q^{n+1}} & \dots & \dots & -k_{r-1}^{q^{n+1}}
\end{pmatrix}\mathfrak{S}_{n}=\begin{bmatrix}
    \mathfrak{N}_{2,n}\\
    \vdots\\
    \mathfrak{N}_{r,n}\\
    \mathfrak{N}_{r,n+1}
\end{bmatrix}.
\end{equation}
For $1\leq j \leq r-1$, we have
\begin{multline*}
\log_q\left(||\mathfrak{N}_{j,n+1}||\right)=\log_q\left(||\mathfrak{N}_{j+1,n}||\right)\\
\leq q^{n-(r-(j+1))}+\log_q(||\mathbb{S}_{n-(r-(j+1))-r}||= q^{n+1-(r-j)}+\log_q(||\mathbb{S}_{n+1-(r-j)-r}||),
\end{multline*}
implying the desired statement for such $j$. On the other hand, since $|k_j|\leq 1$, using the induction hypothesis for $||\mathfrak{N}_{1,n}||$ and \eqref{E:induct}, we have 
\[
\log_q\left(||\mathfrak{N}_{r,n+1}||\right)\leq q^{n+1}+(q^{n-(r-1)}+\log_q(||\mathbb{S}_{n-(r-1)-r}||))<q^{n+1}+\log_q(||\mathbb{S}_{n+1-r}||),
\]
finishing the proof of the lemma.
\end{proof}

Finally, a direct computation combined with Lemma \ref{L:02} immediately implies the following proposition which establishes bounds for the $||\cdot||$-norm of 
\[
\widetilde{\varphi_{\text{tens}}}(\tau^n(g_{ru+j}))=\mathfrak{e}_j^{\tr}\mathbb{S}_{n-1}^k(t-\theta^{q^n})^u\mathfrak{S}_{n-1}=(t-\theta)^k\cdots (t-\theta^{q^{n-1}})^k(t-\theta^{q^n})^u\mathfrak{N}_{j,n-1}
\]  and 
\[
\widetilde{\varphi_{\text{tens}}}(\tau^n(g_{rk+1}))=\mathfrak{e}_1^{\tr}\mathbb{S}_{n}^k\mathfrak{S}_{n-1}=(t-\theta)^k\cdots (t-\theta^{q^{n}})^k\mathfrak{N}_{1,n-1}
\]
 for each $n\geq 0$.
\begin{proposition}\label{P:01} Let $0\leq u \leq k-1$ and $1\leq j \leq r$.  We have
\begin{multline*}
\log_q|| \widetilde{\varphi_{\text{tens}}}(\tau^n(g_{ru+j}))||
\leq
\begin{cases} q^nu &\text{ if }n\leq 1 \text{ and } j\neq r\\
q^nu+1 &\text{ if } n\leq 1 \text{ and }j=r\\
q^n\left(u+\frac{k}{q-1}\right)-\frac{k}{q-1} &\text{ if } 1< n<r-j\\
q^n\left(u+\frac{k}{q-1}+q^{-1-(r-j)}\right)-\frac{k}{q-1} &\text{ if } r-j\leq n <2r-j\\
q^n\left(u+\frac{k}{q-1}+q^{-1-(r-j)}+\frac{q^{-(2r-j)}}{q-1}\right)-\frac{k+1}{q-1} &\text{ if } n\geq 2r-j
\end{cases}
\end{multline*}
and 
\[
\log_q|| \widetilde{\varphi_{\text{tens}}}(\tau^n(g_{rk+1}))||\leq \begin{cases} q^nk & \text{ if }n\leq1\\
q^{n}\left(k+\frac{k}{q-1}\right)-\frac{k}{q-1} & \text{ if } 1<n<r-1 \\
q^n\left(k+\frac{k}{q-1}+q^{-r}\right)-\frac{k}{q-1} & \text{ if } r-1\leq n <2r-1\\
q^n\left(k+\frac{k}{q-1}+q^{-r}+\frac{q^{-(2r-1)}}{q-1}\right)-\frac{k+1}{q-1} & \text{ if } n\geq 2r-1.
\end{cases}
\]
\end{proposition}

For each $1\leq \ell \leq rk+1$, $1\leq j \leq r$ and $n\geq 0$, define $f_{j,n,\ell}\in \mathbb{C}_{\infty}[t]$ such that 
\begin{equation}\label{E:defeltsf}
\widetilde{\varphi_{\text{tens}}}(\tau^n(g_{\ell}))=[f_{1,n,\ell},\dots,f_{r,n,\ell}]\in \Mat_{1\times r}(\mathbb{C}_{\infty}[t]).
\end{equation}
We also set $f_{j,n,\ell}:=0$ if $\ell>rk+1$.

Recall the polynomial $\mathfrak{p}_{\ell,m}(t)$ for any $\ell\in \mathbb{Z}_{\geq 0}$ and $0\leq m \leq r-1$ introduced in \S\ref{SS:varphi map} as well as the subset $\mathfrak{M}\subset \mathbb{C}_{\infty}$ from Lemma \ref{L:t degree of phi}.  

\begin{lemma}\label{L:t degree of phi2}
Let $n = sr + j$ for $1\leq j\leq r$, $s\in \mathbb{Z}_{\geq 0}$ and let $0\leq u \leq k$.
The following statements hold.
\begin{itemize}
\item[(i)]  For each $1\leq i \leq r$, $f_{i,n,ru+r}$ can be written as an $\mathfrak{M}$-linear combination of polynomials $\mathbb{S}_{n-1}^k(t-\theta^{q^n})^u \mathfrak{p}_{\tilde{s}+1,\tilde{j}-1}(t)$ so that $0\leq \tilde{s} \leq s$, $1\leq \tilde{j}\leq r$ and $\tilde{s}r+\tilde{j}\leq n$. Moreover, 
$\deg_t(f_{j,n,ru+r}) = s+1+kn+u$ and
we have 
\[
f_{j,n,ru+r}=\mathbb{S}_{n-1}^k(t-\theta^{q^n})^u\left( a\mathfrak{p}_{s+1,j-1}(t)+ \sum_{\substack{0\leq \tilde{s}<s\\1\leq \tilde{j}\leq r}}\beta_{\tilde{s},\tilde{j}}\mathfrak{p}_{\tilde{s}+1,\tilde{j}-1}(t)\right)
\]
for some  $a\in \F_q^\times$ and $\beta_{\tilde{s},\tilde{j}}\in \mathfrak{M}$ for each $\tilde{s}$ and $\tilde{j}$.
\item[(ii)] For $0\leq v \leq r-1$, we have $\tau^n(g_{v})=\tau^n(\mathfrak{m}_{v})=\tau^{n-1}(\mathfrak{m}_{v+1})=\tau^{n-1}(g_{v+1})$  and 
$
\tau^{u}(g_v)=\mathfrak{e}_{u+v}^{\tr}
$ whenever $u+v\leq r$. In particular, we have 
\[
\tau^{n}(g_{ru+r-v})=\tau^{n}((t-\theta)^u\mathfrak{m}_{r-v})=\tau^{n-v}((t-\theta)^{u}\mathfrak{m}_{r})=\tau^{n-v}(g_{ru+r}).
\]
 Moreover, if  we let
\[
\overline{j-v}:=\begin{cases}
    j-v & \text{ if } j>v\\
    r-(v-j) & \text{ if } j\leq v
\end{cases}
\ \ \text{and } \ \  
\overline{s}:=\begin{cases}
    s & \text{ if } j>v\\
    s-1 & \text{ if } j\leq v,
\end{cases}
\]
then $\deg_t(f_{\overline{j-v},n,ru +r-v}) = \overline{s}+1+kn+u$. Furthermore, we have 
\[
f_{\overline{j-v},n,ru +r-v}= \mathbb{S}_{n-1}^k(t-\theta^{q^n})^{u}\left( a\mathfrak{p}_{\overline{s}+1,\overline{j-v}-1}(t)+\sum_{\substack{0\leq \tilde{s}<\overline{s}\\1\leq \tilde{j}\leq r}}\beta_{\tilde{s},\tilde{j}}\mathfrak{p}_{\tilde{s}+1,\tilde{j}-1}(t)\right)
\]
for some  $a\in \F_q^\times$ and $\beta_{\tilde{s},\tilde{j}}\in \mathfrak{M}$.
\item[(iii)] $\deg_t(f_{\overline{j-v},n,ru +r-v}) \geq \deg_t(f_{i,n,ru +r-v})$ for $i<\overline{j-v}$.
\item[(iv)] $\deg_t(f_{\overline{j-v},n,ru +r-v}) > \deg_t(f_{i,n,ru +r-v})$ for $i>\overline{j-v}$.
\end{itemize}
\end{lemma}
\begin{proof} The part (i), (iii) and (iv) simply follow from the same analysis applied to $\prod_{\mu=1}^{n}(\Phi^{\tr})^{(n-\mu)}$ in the proof of Lemma \ref{L:t degree of phi} for the matrix $\mathbb{S}_{n-1}^k(t-\theta^{q^n})^u\mathfrak{S}_{n-1}$. We now comment about the proof of part (ii). The first assertion of (ii)   follows from the observation that if $\mathfrak{S}_n=(c_{i,j})_{ij}$ for some $c_{i,j}\in \mathbb{T}$, then 
\[
\mathfrak{S}_{n+1}=\Theta^{(n+1)} \mathfrak{S}_{n}= \begin{pmatrix}
&1& & &  \\
& & \ddots & & \\
& & & \ddots & \\
& & & &  1\\
t-\theta^{q^{n+1}}&-k_1^{q^{n+1}} & \dots & \dots & -k_{r-1}^{q^{n+1}}
\end{pmatrix}\mathfrak{S}_{n}= \begin{pmatrix}
    c_{2,1} & c_{2,2} & \cdots & \cdots & c_{2,r}^{(1)}\\
    \vdots  &  & &  & \vdots\\
    \vdots & & & & \vdots \\
    c_{r-1,1} & c_{r-1,2} & \cdots & \cdots & c_{r-1,r}\\
    * & * & \cdots & \cdots & *\\
    \end{pmatrix}.
\]
The second and the last assertion also follow from the first assertion and part (i).
\end{proof}

For $1\leq \mu \leq r$, $n,\nu\geq 0$ and $1\leq \ell \leq rk+1$, we further define elements $c_{\mu,n,\ell,\nu}\in \mathbb{C}_{\infty}$ given by the equality 
\begin{equation}\label{E:coefc}
f_{\mu,n,\ell}=\sum_{\nu\geq 0}c_{\mu,n,\ell,\nu}t^{\nu}\in \mathbb{C}_{\infty}[t].
\end{equation}

Our next lemma can be also obtained by using Lemma \ref{L:t degree of phi2} and the same idea in the proof of Lemma \ref{L:maxnorm}. We leave the details of its proof to the reader.

\begin{lemma}\label{L:maxnormtensor} Let $n=sr+j$  for $s\in \mathbb{Z}_{\geq 0}$ and $1\leq j\leq r$. Choose $0\leq v\leq r-1$ and $0\leq u \leq k$ as well as consider $\widetilde{\varphi_{\text{tens}}}(\tau^{n}(g_{ru+r-v}))\in \Mat_{1\times r}(\mathbb{C}_{\infty}[t])$. Let $n_0\leq n$, $0\leq u_0\leq k$ and $ \ell\in \mathbb{Z}_{\geq 0} $ be such that $\ell+kn_0+u_0\leq \overline{s}+1+kn+u$. Set \[
\alpha_{n,ru+r-v,\ell+kn_0+u_0}:=\max\{|c_{1,n,ru+r-v,\ell+kn_0+u_0}|,\dots,|c_{r,n,ru+r-v,\ell+kn_0+u_{0}}|\}, 
\]
that is, the maximum among the $|\cdot|$-norms of the $t^{\ell+kn_0+u_0}$-coefficients of the entries of $\widetilde{\varphi_{\text{tens}}}(\tau^{n}(g_{ru+r-v}))$. Then we have
\[
\log_q(\alpha_{n,ru+r-v,\ell+kn_0+u_0})\leq \frac{q^{n+r-1}-q^{n_0+r-1}}{q^r-1}+uq^{n} - u_0q^{n_0}+ k\left(\frac{q^{n}-1}{q-1}-\frac{q^{n_0}-1}{q-1}\right).
\]
\end{lemma}

 Identifying $M_{\rho}$ with $\Mat_{1\times (rk+1)}(\mathbb{C}_{\infty}[\tau])$ by sending $\tau^n(g_{ru+j})$ to  $\tau^n\mathfrak{f}_{ru+j}$ and $\tau^n(g_{rk+1})$ to $\tau^n\mathfrak{f}_{rk+1}$, by a slight abuse of notation, we now denote the aforementioned isomorphism of $\mathbb{C}_{\infty}[t,\tau]$-modules by the map $\widetilde{\varphi_{\text{tens}}}:M_{\rho}\to \Mat_{1\times r}(\mathbb{C}_{\infty}[t])$ given by
\[
\widetilde{\varphi_{\text{tens}}}\left(\left(\sum_{n\geq 0}a_{n,1}\tau^n,\dots,\sum_{n\geq 0}a_{n,rk+1}\tau^n\right)\right):=\left[ \sum_{\ell=1}^{rk+1}\sum_{n\geq 0}a_{n,\ell}f_{1,n,\ell},\dots, \sum_{\ell=1}^{rk+1}\sum_{n\geq 0}a_{n,\ell}f_{r,n,\ell} \right]
\]
where $a_{n,\ell}\in \mathbb{C}_{\infty}$ for each $n\geq0$ and $1\leq \ell rk+1$. 

We now construct the domain of the extension of the map $\widetilde{\varphi_{\text{tens}}}$. To ease the notation in what follows, let us set $\mathfrak{v}_{u}:=q^{\frac{q^{r-1}}{q^r-1}+u+\frac{k}{q-1}}$ for each $0\leq u \leq k$ and define 
\begin{multline*}
\mathbb{M}_{\text{tens}}:=\{\left(\sum_{n=0}^{\infty}a_{n,1}\tau^n,\dots,\sum_{n=0}^{\infty}a_{n,rk+1}\tau^n\right)\in \Mat_{1\times (rk+1)}(\mathbb{C}_{\infty}[[\tau]])\ \ | \ \ 1\leq j \leq r  \\ |a_{n,ru+j}|\mathfrak{v}_u^{q^n}\to 0 
 \text{ as } n\to \infty \}.
\end{multline*}
We further set
\[
\left|\left(\sum_{n=0}^{\infty}a_{n,1}\tau^n,\dots,\sum_{n=0}^{\infty}a_{n,rk+1}\tau^n\right)\right|_{\mathfrak{v}_{\text{tens}}}:=\max_{n}\left\{|a_{n,ru+j}|\mathfrak{v}_u^{q^n}\right\}.
\]
It is clear that $(\mathbb{M}_{\text{tens}},|\cdot|_{\mathfrak{v}_{\text{tens}}})$ forms a normed $\mathbb{C}_{\infty}$-vector space and $(M_{\rho},|\cdot|_{\mathfrak{v}_{\text{tens}}})$ is a dense normed $\mathbb{C}_{\infty}$-vector subspace of $(M_{\rho},|\cdot|_{\mathfrak{v}_{\text{tens}}})$. Moreover, by Proposition \ref{P:02}(ii), we see that for each $1\leq j \leq r$,
\[
\mathcal{H}_{j}:=\left(\sum_{n=0}^{\infty} b_{j,1}^{[n]}\tau^n,\dots, \sum_{n=0}^{\infty} b_{j,rk+1}^{[n]}\tau^n\right) \in \mathbb{M}_{\text{tens}}.
\]
We note that $\mathcal{H}_j$ is the $rk+1-(r-j)$-th entry of the logarithm series $\Log_{\rho}$. Furthermore, comparing the bounds for the $||\cdot||$-norm of  $\tau^n(g_{ru+j})$ in Proposition \ref{P:01}  with $\mathfrak{v}_{u}$,  if we have $(\sum_{n=0}^{\infty}a_{n,1}\tau^n,\dots,\sum_{n=0}^{\infty}a_{n,rk+1}\tau^n)\in \mathbb{M}_{\text{tens}}$ then $\sum_{n=0}^{\infty}a_{n,ru+j}\tau^n(g_{ru+j})$  converges in $\Mat_{1\times (rk+1)}(\mathbb{T})$.

Let $\mathcal{G}\in M_{\rho}\subset \mathbb{M}_{\text{tens}}$. By the ultrametric property of $||\cdot||$ on $\Mat_{1\times r}(\mathbb{C}_{\infty}[t])$, we see that 
\[
||\widetilde{\varphi_{\text{tens}}}(\mathcal{G})||\leq |\mathcal{G}|_{\mathfrak{v}_{\text{tens}}}.
\]
Hence, $\widetilde{\varphi_{\text{tens}}}$ is a continuous and bounded $\mathbb{C}_{\infty}$-linear map. Since  $(\Mat_{1\times r}(\mathbb{T}),||\cdot||)$ is a Banach space over $\mathbb{C}_{\infty}$ and $M_{\rho}$ is dense in $ \mathbb{M}_{\text{tens}}$, there exists a unique bounded extension $\varphi_{\text{tens}}:\mathbb{M}_{\text{tens}}\to \Mat_{1\times r}(\mathbb{T})$ of $\widetilde{\varphi_{\text{tens}}}$ defined by 
\[
\varphi_{\text{tens}}\left(\lim_{n\to \infty}\mathcal{G}_n\right):=\lim_{n\to \infty}\widetilde{\varphi_{\text{tens}}}(\mathcal{G}_n)
\]
provided that $\lim_{n\to \infty}\mathcal{G}_n $ exists and lies in $\mathbb{M}_{\text{tens}}$ (see \cite[Thm. 5.19]{HN01}).

Our final goal is to prove that $\varphi_{\text{tens}}$ is injective. 

\begin{proposition}\label{P:inj2}
Let
\[f = \varphi_{\text{tens}}\left(\left(\sum_{n=0}^{\infty}a_{n,1}\tau^n,\dots, \sum_{n=0}^{\infty}a_{n,rk+1}\tau^n\right)\right). \]
 Then $f=0$ if and only if each $a_{n,\ell}=0$ for $1\leq \ell \leq rk+1$. In particular, $\varphi_{\text{tens}}$ is injective.
\end{proposition}

\begin{proof} We note that the idea of the proof is exactly the same as the idea used to prove Proposition \ref{P:inj} up to certain technical details which we will explain below. Since one direction is obvious, we prove the other direction. Using the elements $c_{\mu,n,\ell,\nu}\in \mathbb{C}_{\infty}$ defined in \eqref{E:coefc},  we have
\begin{multline*}
\varphi_{\text{tens}}\left(\left(\sum_{n=0}^{\infty}a_{n,\ell}\tau^n,\dots, \sum_{n=0}^{\infty}a_{n,rk+1}\tau^n\right)\right)=\left[\sum_{\ell=1}^{rk+1}\sum_{n=0}^{\infty}a_{n,\ell}f_{1,n,\ell},\dots,\sum_{\ell=1}^{rk+1}\sum_{i= 0}^{\infty}a_{n,\ell}f_{r,n,\ell} \right]\\
=\left[\sum_{\nu= 0}^{\infty}\left(\sum_{\ell=1}^{rk+1}\sum_{n=0}^{\infty}a_{n,\ell}c_{1,n,\ell,\nu}\right)t^\nu,\dots, \sum_{\nu= 0}^{\infty}\left(\sum_{\ell=1}^{rk+1}\sum_{n=0}^{\infty}a_{n,\ell}c_{r,n,\ell,\nu}\right)t^{\nu} \right].
\end{multline*}
We then write 
\begin{equation}\label{E:Tate module infinite sums}
f=\begin{pmatrix}
(\sum_{\ell=1}^{rk+1}\sum_{n= 0}^{\infty}a_{n,\ell}c_{1,n,\ell,0}) + (\sum_{\ell=1}^{rk+1}\sum_{n=0}^{\infty}a_{n,\ell}c_{1,n,\ell,1})t + (\sum_{\ell=1}^{rk+1}\sum_{n=0}^{\infty}a_{n,\ell}c_{1,n,\ell,2})t^2+\dots\\
\vdots\\
(\sum_{\ell=1}^{rk+1}\sum_{n=0}^{\infty}a_{n,\ell}c_{r,n,\ell,0}) + (\sum_{\ell=1}^{rk+1}\sum_{n=0}^{\infty}a_{n,\ell}c_{r,n,\ell,1})t + (\sum_{\ell=1}^{rk+1}\sum_{n=0}^{\infty}a_{n,\ell}c_{r,n,\ell,2})t^2+ \dots\\
\end{pmatrix}^{\tr}\in \mathbb{T}^r. 
\end{equation}
  Now let $f=0$. Thus, we have a sequence of infinite series so that $\sum_{\ell=1}^{rk+1}\sum_{n=0}^{\infty}a_{n,\ell}c_{\mu,n,\ell,\nu} = 0$
for all $1\leq \mu\leq r$ and $\nu\geq 0$. 

Assume to the contrary that there exist integers $n_0\in \mathbb{Z}_{\geq 0}$, $0\leq u_0\leq k$ and $1\leq j_0\leq r$ such that $a_{n_0,ru_0+j_0}\neq 0$. We then write $n_0 = s_0r+j_0'$ with $s_0\in \mathbb{Z}_{\geq 0}$ and $1\leq j_0'\leq r$. Let us denote $(j_{0}+j_{0}')\pmod{r}$ by $\mathfrak{j}_{0}$ (with the convention that $\mathfrak{j}_{0}=0$ if $j_0+j_0'\equiv 0\pmod{r}$). Note, by Lemma \ref{L:t degree of phi2}(ii) that the $\mathfrak{j}_{0}$-th coordinate (with the convention that we refer to the $r$-th coordinate if $\mathfrak{j}_{0}=0$) of $\tau^{n_0}(g_{ru_0+j_0})$ is given as 
\begin{equation}\label{E:form2}
f_{\mathfrak{j}_0,n_0,ru_0+j_0}=a(t-\theta^{q^{n_0}})^{u_0}\mathbb{S}_{n_0-1}^k\mathfrak{p}_{\overline{s_0}+1,\mathfrak{j}_0-1}(t)+\text{terms in $t$ degree lower than $ \overline{s_0}+1+kn_0+u_0 $ }
\end{equation}
where $a\in \mathbb{F}_q^{\times}$, $\overline{s_0}=s_0$ if $j_0'+j_0> r$ and $\overline{s_0}=s_0-1$ if $j_0'+j_0\leq  r$. By Lemma \ref{L:t degree of phi2}(ii), we see that $c_{\mathfrak{j}_0,n_0,ru_0+j_0,kn_0+u_0+\overline{s_0}+1}=a$ and 
\begin{equation}\label{E:series3}
\sum_{\ell=1}^{rk+1}\sum_{n= 0}^{\infty}a_{n,\ell}c_{\mathfrak{j}_0,n,\ell,kn_0+u_0+\overline{s_0}+1} = a_{n_0,ru_0+j_0}a+\sum_{(n,\ell)\neq (n_0,ru_0+j_0)}a_{n,\ell}c_{\mathfrak{j}_0,n,\ell,kn_0+u_0+\overline{s_0}+1} =0.
\end{equation}
Now we examine the $t^{\overline{s_0}+1+kn_0+u_0}$-coefficient of the $\mathfrak{j}_0$-th coordinate of $f$, which is the series
in the left hand side of \eqref{E:series3}. Since $a_{n_0,ru_0+j_0}\neq 0$ and the norm $| \cdot |$ is nonarchimedian, there must exist integers $n_1\in \mathbb{Z}_{\geq 0}$, $0\leq u_1\leq k$ and $1\leq j_1\leq r$  such that
\begin{equation}\label{E:boundan2}
|a_{n_0,ru_0+j_0}a|=|a_{n_0,ru_0+j_0}|  \leq  |a_{n_1,ru_1+j_1} c_{\mathfrak{j}_0,n_1,ru_1+j_1,kn_0+u_0+\overline{s_0}+1} |.
\end{equation}
Since, by Lemma \ref{L:t degree of phi2}, for each $0\leq \tilde{u}\leq k$ and $1\leq \tilde{i} \leq r$, the coefficient of $t^{\overline{s_0}+1+kn_0+u_0}$ in the $\mathfrak{j}_0$-th coordinate of $\tau^{w}(g_{r\tilde{u}+\tilde{i}})$ is zero for $w<n_0$, we must have $n_1>n_0$. Now let us write $n_1=s_1r+j_1'$ with $s_1\in \mathbb{Z}_{\geq 0}$ and $1\leq j_1'\leq r$. Then, by Lemma \ref{L:maxnormtensor}, we obtain
\begin{multline}\label{E:est1}
\log_q(|c_{\mathfrak{j}_0,n_1,ru_1+j_1,kn_0+u_0+\overline{s_0}+1}|)\leq \log_q(\alpha_{n_1,ru_1+j_1,kn_0+u_0+\overline{s_0}+1}) \\
\leq \frac{q^{n_1+r-1}}{q^r-1}+u_1q^{n_1}+\frac{kq^{n_1}-k}{q-1}-\left(\frac{q^{n_0+r-1}}{q^r-1}+u_0q^{n_0}+\frac{kq^{n_0}-k}{q-1}\right).
\end{multline}

 Thus, \eqref{E:boundan2} and \eqref{E:est1} yield
\[
|a_{n_0,ru_0+j_0}||\theta|^{\frac{q^{n_0+r-1}}{q^r-1}+u_0q^{n_0}+\frac{kq^{n_0}-k}{q-1}} \leq |a_{n_1,ru_1+j_1}||\theta|^{\frac{q^{n_1+r-1}}{q^r-1}+u_1q^{n_1}+\frac{kq^{n_1}-k}{q-1}}.
\]
Finally, as in the proof of Proposition \ref{P:inj}, applying this algorithm once again and hence continuing in this manner, we obtain a chain of integers  $n_0<n_1<n_2<\cdots<n_{w}<\cdots $ and an increasing sequence $\{|a_{n_w,ru_{w}+j_{w}} \theta^{\frac{q^{n_{w}+r-1}}{q^r-1}+u_{w}q^{n_{w}}+\frac{kq^{n_{w}}-k}{q-1}}|\}_{w\geq 0}$. On the other hand, by the assumption on elements in $\mathbb{M}_{\text{\text{tens}}}$, we obtain
\begin{multline*}
\left |a_{n_w,ru_{w}+j_{w}} \theta^{\frac{q^{n_{w}+r-1}}{q^r-1}+u_{w}q^{n_{w}}+\frac{kq^{n_{w}}-k}{q-1}}\right | = |a_{n_w,ru_{w}+j_{w}}|(|\theta|^{\frac{q^{r-1}}{q^r-1}+u_{w}+\frac{k}{q-1}})^{q^{n_{w}}}\\
= |a_{n_w,ru_{w}+j_{w}}|\mathfrak{v}_{u_{w}}^{q^{n_{w}}}  \to 0
\end{multline*}
as $w\to \infty$. But this contradicts to the fact that $\{|a_{n_w,ru_{w}+j_{w}} \theta^{\frac{q^{n_{w}+r-1}}{q^r-1}+u_{w}q^{n_{w}}+\frac{kq^{n_{w}}-k}{q-1}}|\}_{w\geq 0}$ is an increasing sequence. Hence $a_{n_0,ru_0+j_0}$ must be equal to zero, finishing the proof of the proposition.
\end{proof}

\subsection{The element $\eta_n$}\label{S:theelementeta} Our goal in this subsection is similar to what we aim in \S \ref{SS:alpha convergence}. More precisely, we define an element $\eta_n\in \mathbb{C}_{\infty}^r\otimes \mathbb{C}_{\infty}^r$ for each $n\in \mathbb{Z}_{\geq 1}$ so that in \eqref{E:tensfactor}, we use it 
to interpret $((1\otimes t) - (t\otimes 1))G_{k,n}^\otimes\in \mathbb{T}^r\otimes \mathbb{T}^r$ in terms of matrices $\Pi_n$ and $\Psi_n$ defined in \S\ref{S:product}.

 For any positive integer $n$, we now consider
\begin{multline*}
\eta_n:=(\mathfrak{U}^{-1})^{(n)}(\tilde{V}^{\tr})^{(n)}\mathfrak{e}_{1}\otimes \mathfrak{e}_{1}^{\tr}(((\tilde{V}^{(-1)})^{-1})^{\tr})^{(n+1)}\mathfrak{U}^{(n)} +\dots +\\
(\mathfrak{U}^{-1})^{(n)}(\tilde{V}^{\tr})^{(n)}\mathfrak{e}_{r}\otimes \mathfrak{e}_{r}^{\tr}(((\tilde{V}^{(-1)})^{-1})^{\tr})^{(n+1)}\mathfrak{U}^{(n)} \in \mathbb{T}^{r}\otimes \mathbb{T}^{r}.
\end{multline*}
Observe that 
\[
\eta_n=(\mathfrak{U}^{-1})^{(n)}(\tilde{V}^{\tr})^{(n)}(\mathfrak{e}_{1}\otimes \mathfrak{e}_{1}^{\tr}+\dots+ \mathfrak{e}_{r}\otimes \mathfrak{e}_{r}^{\tr}) ((\tilde{V}^{-1})^{\tr})^{(n)}\mathfrak{U}^{(n)}.
\]
Thus, since $\mathfrak{e}_{1}\otimes \mathfrak{e}_{1}^{\tr}+\dots+ \mathfrak{e}_{r}\otimes \mathfrak{e}_{r}^{\tr}=\eta_n=\sum_{i=1}^{r}\mathfrak{d}_i\otimes \mathfrak{c}_i$, we finally obtain our next theorem. 
\begin{theorem}\label{T:alpha2} We have
	\[
	\eta_n=\sum_{i=1}^{r}\mathfrak{d}_i\otimes \mathfrak{c}_i=\mathfrak{e}_{1}\otimes \mathfrak{e}_{1}^{\tr}+\dots+ \mathfrak{e}_{r}\otimes \mathfrak{e}_{r}^{\tr}.
	\]
	In particular, via the identification in Remark \ref{R:1}, $\eta_n=\Id_r$ for each $n\geq 1$.
\end{theorem}

\subsection{Proof of Theorem \ref{T:2}}\label{S:proofoflastresult} To prove our second main result, we this subsection, we first consider the matrix 
\[
\mathfrak{T}:=\rho_{\theta}^{\tr}= \begin{bmatrix}
\theta& \bovermat{$r(k-1)$-many}{0&\dots &0} &\tau& 0&\dots&0 &\tau k_1^{(-1)}\\
&\ddots& & & &\ddots & & & \vdots\\
& & \ddots &  & & &\ddots & &   \tau k_{r-1}^{(-1)}\\
& &  &   \ddots& & & &\tau &  \tau k_r^{(-1)} \\
1 &  &   & & \ddots& & & &  0\\
&\ddots &  & & &\ddots & & &  \vdots\\
& & \ddots &  & & & \ddots& &  \vdots\\
& &  & \ddots  & & & &\ddots &0\\ 
& &  &   & 1& & & &  \theta
\end{bmatrix}.
\]
Thus, we have
$
t\cdot \mathfrak{g}=\mathfrak{T}^{\tr} \mathfrak{g}
$
and 
$
t\cdot \mathfrak{h}=(\mathfrak{T}^{*})^{\tr}\mathfrak{h}.
$ In this case, we write
\[
\Theta_{\rho,\tau}=\begin{bmatrix}
0&\dots&\dots & \dots &  &\dots& 0 \\
\vdots & & & & & &  \\
0&\dots&\dots & \dots & &\dots& 0 \\
\tau &  & & & & &   \\
& \ddots & & & & &   \\
& &\tau& & \\
k_1\tau &\dots &k_{r-1}\tau& k_r\tau &0 &\dots& 0 
\end{bmatrix}
\]
where we note that the first $r(k-1)+1$-rows of $\Theta_{\rho,\tau}$ are zero. Furthermore, the formula given in \eqref{E:Gn tensor def} for the Anderson $t$-module $\rho$, which we denote as $G_{k,n}^{\otimes}$, reduces to
\[
G_{k,n}^{\otimes}= \sum_{i=0}^n\sum_{\ell=1}^{rk+1}\sigma^{-i}(h_\ell)\otimes \tau^{i}(g_\ell).
\]

By Proposition \ref{P:Gntensor factorization}, we obtain the following.
\begin{proposition}\label{P:tensorcase} We have
\begin{multline*}
(1\otimes t-t\otimes 1)G_{k,n}^{\otimes}=\sigma^{-n}(h_{r(k-1)+2})\otimes \tau^{n+1}(g_1)+\dots+\sigma^{-n}(h_{rk})\otimes \tau^{n+1}(g_{r-1})\\+\sigma^{-n}(h_{rk+1})\otimes \tau^{n+1}(\mathfrak{c}_1)
-\sum_{j=2}^{r+1}\sigma(h_{r(k-1)+j})\otimes g_{j-1}-\sum_{j=1}^{r}k_j^{(-1)}\sigma(h_{rk+1})\otimes g_{j}.
\end{multline*}
\end{proposition}

Put $\tilde{\gamma}:=\sum_{j=2}^{r+1}\sigma(h_{r(k-1)+j})\otimes g_{j-1}+\sum_{j=1}^{r}k_j^{(-1)}\sigma(h_{rk+1})\otimes g_{j}$. Then, by using Proposition \ref{P:tensorcase} as well as the definition of $\sigma-$ and $\tau-$ action on $\tilde{N}_{\rho}$ and $M_{\rho}$ respectively, we have
\begin{equation}\label{E:tensfactor}
\begin{split}
&(1\otimes t-t\otimes 1)G_{k,n}^{\otimes}+\tilde{\gamma}\\
&=\sigma^{-n}(h_{r(k-1)+2})\otimes \tau^{n+1}(g_1)+\dots+\sigma^{-n}(h_{rk})\otimes \tau^{n+1}(g_{r-1}) +\sigma^{-n}(h_{rk+1})\otimes \tau^{n+1}(\mathfrak{c}_1)\\
&=\sum_{\mu=1}^{rk+1}\sigma^{-n}(h_\mu)\tau^n(\mathfrak{g}^{\tr}\Theta_{\rho,\tau}^{\tr}\mathfrak{f}_\mu )\\
&=\widetilde{\mathcal{P}}^k_n(\tilde{V}^{\tr})^{(n)}\mathfrak{e}_{1}\otimes \mathfrak{e}_{1}^{\tr}(((\tilde{V}^{(-1)})^{-1})^{\tr})^{(n+1)}\widetilde{\mathcal{S}}^k_n +\dots + \widetilde{\mathcal{P}}^k_n(\tilde{V}^{\tr})^{(n)}\mathfrak{e}_{r}\otimes \mathfrak{e}_{r}^{\tr}(((\tilde{V}^{(-1)})^{-1})^{\tr})^{(n+1)}\widetilde{\mathcal{S}}^k_n\\
&=(t-\theta^q)^{-k}\cdots (t-\theta^{q^n})^{-k}V^{\tr}\Pi_{n-1}^{(1)}\eta_n (\Psi_n^{(-1)})^{\tr}(t-\theta)^{k}(t-\theta^q)^{k}\cdots (t-\theta^{q^n})^{k}\\
&=(-1)^k \left((-\theta)^{q/(q-1)}\prod_{i=1}^n\left(1-\frac{t}{\theta^{q^i}}\right)^{-1}\right)^k V^{\tr}\Pi_{n-1}^{(1)}\eta_n (\Psi_n^{(-1)})^{\tr}\left((-\theta)^{-1/(q-1)}\prod_{i=0}^n\left(1-\frac{t}{\theta^{q^i}}\right)\right)^k.
\end{split}
\end{equation}

We now let $\tilde{\mathfrak{e}}_{1}:=v_{1r}\mathfrak{e}_1^{\tr}$ and for $2\leq j \leq r$, set 
\[
\tilde{\mathfrak{e}}_{j}:=\sum_{i=j}^rv_{i(r-(j-1))}\mathfrak{e}_i^{\tr}.
\]
Recall the projection $p_{i}:\mathbb{C}_{\infty}^{rk+1}\to \mathbb{C}_{\infty}$ onto the $i$-th coordinate as well as the entire functions $F_{\tau^i}:\mathbb{C}_{\infty}\to \mathbb{C}_{\infty}$ for each $1\leq i \leq r-1$ defined in \S \ref{S:Log of Drinfeld Modules}. Recall also the fundamental periods $\lambda_1,\dots,\lambda_r$ of $\phi$ defined in \S \ref{S:Log of Drinfeld Modules}.

Recall from \S\ref{SS:Tate twists intro}, the maps $\mathcal{M}_{\text{tens}}:=\varphi_{\text{tens}}^{-1}$ and $\mathcal{M}_{\text{tens},\zz}:=\delta_{1,\zz}^{M_{\rho}}\circ \varphi_{\text{tens}}^{-1}$ for $\zz\in \mathbb{C}_{\infty}^{rk+1}$. Recall also the fundamental periods $\lambda_1,\dots,\lambda_r$ of $\phi$ and the row vector $\opi=(\lambda_1,\dots,\lambda_r)$.

\begin{theorem}\label{T:last} We have 
	\begin{multline*}
p_{rk+1-(j-1)}(\Log_{\rho})=\\ \begin{cases}
\mathcal{M}_{\text{tens}}\left(\frac{\tilde{\pi}^k}{\omega_C^k(\theta-t)}\opi(\Psi^{\tr})^{(-1)}\right) & \text{ if } j=1\\
\mathcal{M}_{\text{tens}}\left(\frac{\tilde{\pi}^k}{\omega_C^k(t-\theta)}(F_{\tau^{r-(j-1)}}(\lambda_1),\dots, F_{\tau^{r-(j-1)}}(\lambda_r))(\Psi^{\tr})^{(-1)}\right) & \text{ if } 2\leq j \leq r.
\end{cases}
\end{multline*}
Let $\zz\in \mathbb{C}_{\infty}^{rk+1}$ be an element in the domain of convergence of $\Log_{\phi\otimes C^{\otimes k}}$. Then
	\begin{multline*}
p_{rk+1-(j-1)}(\Log_{\rho}(\zz))=\\ \begin{cases}
\mathcal{M}_{\text{tens},\zz}\left(\frac{\tilde{\pi}^k}{\omega_C^k(\theta-t)}\opi(\Psi^{\tr})^{(-1)}\right) & \text{ if } j=1\\
\mathcal{M}_{\text{tens},\zz}\left(\frac{\tilde{\pi}^k}{\omega_C^k(t-\theta)}(F_{\tau^{r-(j-1)}}(\lambda_1),\dots, F_{\tau^{r-(j-1)}}(\lambda_r))(\Psi^{\tr})^{(-1)}\right) & \text{ if } 2\leq j \leq r.
\end{cases}
\end{multline*}
\end{theorem}
\begin{proof} Observe that for each $1\leq j \leq r$ and $n\geq 0$, if we set $\sigma^{-n}(h_{r(k-1)+j+1})=[\xi_{j1,n},\dots, \xi_{jr,n}]^{\tr}\in \mathbb{C}_{\infty}(t)^{r} $, then 
\begin{multline*}
\delta_0^{N_{\rho}}(\sigma^{-n}(h_{r(k-1)+j+1}))=[
	*,
	\dots,
	*,\\
	(\xi_{jr,n})|_{t=\theta}v_{r1},
	\sum_{j=r-1}^r(\xi_{ji,n})|_{t=\theta}v_{j2},
	\dots,
	\sum_{i=2}^r(\xi_{ji,n})|_{t=\theta}v_{i(r-1)},
	(\xi_{j1,n})|_{t=\theta}v_{jr}
	]^{\tr}.
\end{multline*}
To ease the notation, for each $n\geq 1$, we further set 
\[
\mathfrak{u}_{1,n}:=\left((-\theta)^{q/(q-1)}\prod_{i=1}^n\left(1-\frac{t}{\theta^{q^i}}\right)^{-1}\right) \text{ and }
\mathfrak{u}_{2,n}:=\left((-\theta)^{-1/(q-1)}\prod_{i=0}^n\left(1-\frac{t}{\theta^{q^i}}\right)\right).
\]
Note that 
\begin{equation}\label{E:prooftens0}
\begin{split}
&\varphi_{\text{tens}}\left(\lim_{n\to \infty}p_{rk+1-(j-1)}\left(\sum_{\ell=1}^{rk+1}\sum_{i=0}^{n}\delta_0^{N_{\rho}}(\sigma^{-i}(h_{\ell}))\tau^{i}(g_{\ell})\right)\right)\\
    &=\varphi_{\text{tens}}\left(\lim_{n\to \infty}\sum_{\ell=1}^{rk+1}\sum_{i=0}^{n}b_{(r-(j-1)),\ell}^{[i]}\tau^{i}(g_{\ell})\right)\\
    &=\lim_{n\to \infty}\widetilde{\varphi_{\text{tens}}}\left(\sum_{\ell=1}^{rk+1}\sum_{i=0}^{n}b_{(r-(j-1)),\ell}^{[i]}\tau^{i}(g_{\ell})\right)\\
    &=\lim_{n\to \infty}\frac{1}{t-\theta}\widetilde{\varphi_{\text{tens}}}\left(\sum_{\ell=1}^{rk+1}\delta_0^{N_{\rho}}(\sigma^{-n}(h_{\ell}))\tau^n(\mathfrak{g}^{\tr}\Theta_{\rho,\tau}^{\tr}\mathfrak{f}_\mu)\right)\\
    &=\lim_{n\to \infty}\frac{(-1)^k}{t-\theta}\tilde{\mathfrak{e}}_{j}V^{\tr}\mathfrak{u}_{1,n}^k(\Pi_{n-1}^{(1)})_{|t=\theta}\eta_n (\Psi_n^{(-1)})^{\tr}\mathfrak{u}_{2,n}^k \\
    &=\frac{(-1)^k}{t-\theta}\tilde{\mathfrak{e}}_jV^{\tr}(\tilde{\Upsilon}^{(1)})_{|t=\theta}(\tilde{\Psi}^{\tr})^{(-1)}\in \Mat_{1\times r}(\mathbb{T}).
    \end{split}
    \end{equation}
Here the first equality follows from \cite[Cor. 4.5]{Gre22}. For the second equality, under the identification between $M_{\rho}$ and $\Mat_{1\times (rk+1)}(\mathbb{C}_{\infty}[\tau])$ described in \S\ref{SS:varphi map2}, we note that 
\begin{multline*}
\lim_{n\to \infty}\sum_{\ell=1}^{rk+1}\sum_{i=0}^{n}b_{(r-(j-1)),\ell}^{[i]}\tau^{i}(g_{\ell})=\lim_{n\to \infty}(\sum_{i=0}^{n}b_{(r-(j-1)),1}^{[i]}\tau^{i},\dots, \sum_{i=0}^{n}b_{(r-(j-1)),rk+1}^{[i]}\tau^{i})\\
=(\sum_{i=0}^{\infty}b_{(r-(j-1)),1}^{[i]}\tau^{i},\dots, \sum_{i=0}^{\infty}b_{(r-(j-1)),rk+1}^{[i]}\tau^{i})
=\mathcal{H}_{r-(j-1)}\in \mathbb{M}_{\text{tens}}.
\end{multline*}
The third equality, recalling the last $r$-entry of $d[\theta]$ from Example \ref{Ex:1}(iii) for $\phi\otimes C^{\otimes k}$, follows from Proposition \ref{P:deltazeromap} and the fact that the map $\widetilde{\varphi_{\text{tens}}}$ is $\mathbb{C}_{\infty}[t]$-linear, the fourth equality follows from applying $\delta_{0}^{N_{\phi}}\otimes 1$ to \eqref{E:tensfactor} as well as the structure of $\delta_{0}^{N_{\phi}}$-map described above and finally the last equality follows from Theorem \ref{T:product}, \eqref{E:rigiddual} and Theorem \ref{T:alpha2}.

On the other hand, using the definition of $\tilde{\Upsilon}$ and $\tilde{\Psi}$, we first obtain
\begin{equation}\label{E:prooftens1}
\frac{(-1)^k}{t-\theta}\tilde{\mathfrak{e}}_jV^{\tr}(\tilde{\Upsilon}^{(1)})_{|t=\theta}(\tilde{\Psi}^{\tr})^{(-1)}=\frac{\tilde{\pi}^k}{\omega_C^k(t-\theta)}\tilde{\mathfrak{e}}_jV^{\tr}(\Upsilon^{(1)})_{|t=\theta}(\Psi^{\tr})^{(-1)}.
\end{equation}
Moreover, observe that the identity $(\tilde{V}^{-1})^{\tr}\tilde{V}^{\tr}=\Id_{r}$ implies
	\begin{equation}\label{E:prooftens2}
	\tilde{\mathfrak{e}}_jV^{\tr}=\begin{cases} (k_1,\dots,k_r) & \text{ if } j=1\\
	\mathfrak{e}_{r-(j-1)}^{\tr} & \text{ if } 2\leq j \leq r.
	\end{cases}
	\end{equation}
	Thus, using \eqref{E:prooftens1}, \eqref{E:prooftens2} and \cite[(3.4.3), (3.4.5)]{CP12}, we obtain
    \begin{multline}\label{E:prooftens3}
\frac{(-1)^k}{t-\theta}\tilde{\mathfrak{e}}_jV^{\tr}(\tilde{\Upsilon}^{(1)})_{|t=\theta}(\tilde{\Psi}^{\tr})^{(-1)}\\=\begin{cases}
\frac{\tilde{\pi}^k}{\omega_C^k(\theta-t)}(\lambda_1,\dots, \lambda_r)(\Psi^{\tr})^{(-1)} & \text{ if } j=1\\
\frac{\tilde{\pi}^k}{\omega_C^k(t-\theta)}(F_{\tau^{r-(j-1)}}(\lambda_1),\dots, F_{\tau^{r-(j-1)}}(\lambda_r))(\Psi^{\tr})^{(-1)} & \text{ if } 2\leq j \leq r.
\end{cases}
    \end{multline}
Since, by Proposition \ref{P:inj2}, the map $\varphi_{\text{tens}}$ is injective, using \eqref{E:prooftens0} and \eqref{E:prooftens3}, we obtain
\begin{multline*}
    p_{rk+1-(j-1)}(\Log_{\rho})=\mathcal{H}_{r-(j-1)}=\varphi_{\text{tens}}^{-1}\left( \frac{(-1)^k}{t-\theta}\tilde{\mathfrak{e}}_jV^{\tr}(\tilde{\Upsilon}^{(1)})_{|t=\theta}(\tilde{\Psi}^{\tr})^{(-1)} \right)\\=\varphi_{\text{tens}}^{-1}\left( \frac{\tilde{\pi}^k}{\omega_C^k(\theta-t)}(\lambda_1,\dots, \lambda_r)(\Psi^{\tr})^{(-1)}  \right) =\mathcal{M}_{\text{tens}}\left( \frac{\tilde{\pi}^k}{\omega_C^k(\theta-t)}(\lambda_1,\dots, \lambda_r)(\Psi^{\tr})^{(-1)}  \right)
    \end{multline*}
if $j=1$
and 
\begin{multline*}
    p_{rk+1-(j-1)}(\Log_{\rho})=\mathcal{H}_{r-(j-1)}=\varphi_{\text{tens}}^{-1}\left( \frac{(-1)^k}{t-\theta}\tilde{\mathfrak{e}}_jV^{\tr}(\tilde{\Upsilon}^{(1)})_{|t=\theta}(\tilde{\Psi}^{\tr})^{(-1)} \right)\\=\varphi_{\text{tens}}^{-1}\left( \frac{\tilde{\pi}^k}{\omega_C^k(\theta-t)}(\lambda_1,\dots, \lambda_r)(\Psi^{\tr})^{(-1)}  \right) \\=\mathcal{M}_{\text{tens}}\left( \frac{\tilde{\pi}^k}{\omega_C^k(t-\theta)}(F_{\tau^{r-(j-1)}}(\lambda_1),\dots, F_{\tau^{r-(j-1)}}(\lambda_r))(\Psi^{\tr})^{(-1)} \right)
    \end{multline*}
if $2\leq j \leq r$ which finishes the proof of the first assertion. Finally, by Theorem \ref{T:log} and the first assertion, we obtain the second assertion.
\end{proof}
We finish this subsection with the proof of Corollary \ref{C:corspecvaluestens}.
\begin{proof}[{Proof of Corollary \ref{C:corspecvaluestens}}] Let $\phi$ be a Drinfeld module of rank $2$ given as in \eqref{E:drinfeld} such that $k_1\in \mathbb{F}_q$ and $k_2\in \mathbb{F}_q^{\times}$.  We recall the Drinfeld module $\tilde{\phi}$ given by
\[
\tilde{\phi}_{\theta}=\theta-k_1k_2^{-1}\tau+k_2^{-1}\tau^2.
\] 
By \cite[Rem. 5.6]{Gez21}, we know that $L(\phi,1)=L(\tilde{\phi}^{\vee},0)$. Using \cite[Thm. 5.9]{Gez21}, we have 
\[
L(\phi,k+1)=\det\begin{bmatrix}
p_{2k}(\Log_{\rho}(\zz_{2k}))&p_{2k}(\Log_{\rho}(\zz_{2k+1}))\\
p_{2k+1}(\Log_{\rho}(\zz_{2k}))&p_{2k+1}(\Log_{\rho}(\zz_{2k+1}))
\end{bmatrix}.
\]
Now, the corollary is a simple consequence of Theorem \ref{T:last}.
\end{proof}

%%%%%%%%%%%%%%%%%%%%%%%%%%%%%

\end{document}